\documentclass[11pt, leqno]{amsart}
\usepackage{latexsym, amsfonts,mathrsfs, amsmath, amssymb, amscd, epsfig}
\usepackage{mathrsfs,enumerate}
\usepackage[usenames,dvipsnames]{xcolor}
\usepackage{amsthm}
\usepackage{array}
\usepackage{psfrag}
\usepackage{bm}
\usepackage{setspace}
\usepackage{soul}
\usepackage{graphicx}
\usepackage{pdfpages}
\usepackage{comment}
\numberwithin{equation}{section}

\newtheorem{lemma}{Lemma}[section]
\newtheorem{theorem}[lemma]{Theorem}%[section]
\newtheorem{proposition}[lemma]{Proposition}%[section]
\newtheorem{corollary}[lemma]{Corollary}%[section]
\newtheorem{definition}[lemma]{Definition}%[section]
\newtheorem{remark}[lemma]{Remark}%[section]
\newtheorem{problem}[lemma]{{Problem}}%[section]
\newtheorem{condition}[lemma]{Condition}

\newtheorem*{notation}{Notations}
\newtheorem*{proclaim}{Proclamation}
\theoremstyle{definition}

\def\beq#1\eeq{\begin{equation}#1\end{equation}}
\def\balign #1 #2 \ealign{\begin{aligned} #1 #2  \end{aligned} }

\def\Div{{\rm div}}

\def\sgn{{\rm sgn}}

\def\bu{\mathbf{u}}

\def\msE{\mathscr{E}}

\def \msB {\mathscr{B}}

\newcommand \alp{\alpha}
\newcommand \eps{\varepsilon}
\newcommand \vphi{\varphi}

\newcommand \Gam{\Gamma}
\newcommand \gam{\gamma}
\newcommand \om{\omega}
\newcommand \tx{\text}
\newcommand \R{\mathbb{R}}
\newcommand \til{\tilde}
\newcommand \wtil{\widetilde}

\newcommand \der{\partial}
\newcommand \mcl{\mathcal}

\newcommand \ol{\overline}
\newcommand \Om{\Omega}

\newcommand \Gamen{\Gamma_0}
\newcommand \Gamex{\Gamma_L}
\newcommand \Gamw{\Gamma_w}
\newcommand \Gamexg{\Gam_{R}}

\newcommand \rx{{\rm{x}}}

\newcommand \tpsi{\til{\psi}}
\newcommand \tPsi{\til{\Psi}}
\def \msB {\mathscr{B}}

\def\Div{{\rm div}}

\def\sgn{{\rm sgn}}

\newcommand \bx{\mathbf{x}}

\newcommand \us{u_s}
\newcommand{\mfrak}{\mathfrak}
\newcommand \gs{g_{\rm s}}
\newcommand \ls{l_{\rm s}}

\newcommand \barrhoi{\bar{\rho}_I}

\newcommand \barui{\bar u_{I}}
\newcommand \fsonic{\mathfrak{f}_{\rm sn}}

\newcommand \Tac{\mcl{T}_{\rm acc}}

\newcommand \umax{u_{\rm max}}
\newcommand \lmax{l_{\rm max}}
\newcommand \tphi{\tilde{\phi}}
\newcommand \tilT{\tilde{T}}
\newcommand \iterV{\mcl{V}}
\newcommand \iterT{\mcl{T}}
\newcommand \iterP{\mcl{P}}
\newcommand \iterseta{\iterP_{r_3}^m\times \iterP_{r_3}^m}

\newcommand \bJ{\bar{J}}
\newcommand \ubJ{\underbar{J}}

\newcommand \kz{ {\mathcal{K}}}

\begin{document}
\title[Accelerating transonic solutions to E-P system]
{Classical solutions to a mixed-type PDE with a Keldysh-type degeneracy and accelerating transonic solutions to the Euler-Poisson system}

\author{Myoungjean Bae}
\address{Department of Mathematical Sciences, KAIST, 291 Daehak-ro, Yuseong-gu, Daejeon, 43141, Korea}
\email{mjbae@kaist.ac.kr}

\author{Ben Duan}
\address{School of Mathematics, Jilin University, Changchun, Jilin Province, 130012, China}
\email{bduan@jlu.edu.cn}

\author{Chunjing Xie}
\address{School of Mathematical Sciences, Institute of Natural Sciences, Ministry of Education Key Laboratory of Scientific and Engineering Computing, CMA-Shanghai, Shanghai Jiao Tong University, Shanghai, 200240, China
}
\email{cjxie@sjtu.edu.cn}

\begin{abstract}
	In this paper, we first prove the existence of classical solutions to a class of Keldysh-type equations. Next, we apply this existence result to prove the structural stability of one-dimensional smooth transonic solutions to the steady Euler-Poisson system.  %The idea is to find a multiplier to get energy estimate for the linearized problem.
    Most importantly, the solutions constructed in this paper are classical solutions to the Euler-Poisson system, thus their sonic interfaces are not weak discontinuities in the sense that all the flow variables, such as density, velocity and pressure, are at least $C^1$ across the interfaces.
	
\end{abstract}


\date{\today}

\maketitle

\tableofcontents

\section{Introduction}
The mixed type equations are important topics in modern mathematics. They appear in many open problems in  physics and geometry. See \cite{Bers, HanHong} and references therein. And much remains to be explored.

In the first part of this paper, we prove the well-posedness of a linear boundary value problem that contains a Keldysh-type equation in the form of
\[
a_{11}\der_{x_1x_1}u+2a_{12}\der_{x_1x_2}u
	+\der_{x_2x_2}u+a_1\der_{x_1}u=f.
    \]
    In particular, we prove that if the coefficients $a_{11}, a_{12}$ and $a_1$ belong to $H^{m-1}$ for $m\ge 4$, then we have $u\in H^m$ provided that the coefficients satisfy the structure condition \eqref{KZ-condition0} and some compatibility conditions. This result is a generalization of \cite[Theorem 1.7]{KZ}. In the second part of the paper, we apply this general result to study the structural stability of smooth transonic flows of the Euler-Poisson system.
    \iffalse
    The underlying structures of the mixed type nonlinear PDE have been one of the main motivations in developing new methods/techniques/ideas for unified approaches.
    \fi

Given a constant $R>0$, define
\begin{equation}\label{defOmega}
	\begin{aligned}
		&\Om=\{\rx=(x_1, x_2)\in \R^2: 0<x_1<R,\,\,|x_2|<1\}\\
	\end{aligned}
\end{equation}
with
\begin{equation}\label{defboundary}
	\Gamen:=\der\Om\cap \{x_1=0\},\quad  \Gamw:=\der\Om\cap\{|x_2|=1\},\quad
	\Gamexg:=\der\Om\cap\{x_1=R\}.
\end{equation}
For a function $u:\Om\rightarrow \R$, let $\mcl{L}$ be a differential operator of the form
\begin{equation}
	\label{definition:L-general}
	\mcl{L}u:=a_{11}\der_{11}u+2a_{12}\der_{12}u
	+\der_{22}u+a_1\der_{1}u
\end{equation}
with the following properties:
\begin{enumerate}[(i)]
\item $a_{11}, a_{12}, a_1\in H^{m-1}(\Om)$ for $m\ge 4$;
	
\item	there exists a constant $\mu>0$ satisfying
    \begin{equation*}
    \begin{pmatrix}a_{11}& a_{12}\\a_{12}&1\end{pmatrix}\ge \mu \mathbb{I}\,\,\tx{on $\Gamen$},\quad  a_{11}\le -\mu \,\,\tx{on $\Gamexg$}.
    \end{equation*}
\end{enumerate}
The operator $\mcl{L}$ is elliptic near $\Gamen$, and hyperbolic on $\Gamexg$, and it becomes degenerate on $\{{\rx}\in\Om: (a_{11}-a_{12}^2)(\rx)=0\}$. Note that the set $\{{\rx}\in\Om: (a_{11}-a_{12}^2)(\rx)=0\}$ is nonempty because $a_{11}-a_{12}^2$ is continuous in $\Om$. A Keldysh-type equation \[
\displaystyle{\left(\frac R2-x_1\right)\der_{11}u +\der_{22}u=0}
\]
is one of the examples. In this paper, $\der_{x_i}$ and $\der_{x_ix_j}$ are denoted as $\der_i$ and $\der_{ij}$($i,j=1,2$), respectively.

If the degenerate interface $\Gam_d:=\{{\rx}\in\Om: (a_{11}-a_{12}^2)(\rx)=0\}$ is a $C^1$-curve (codimension 1 in $\Om$), then one may try to solve $\mcl{L}u=f$ by the method of vanishing viscosity. Namely, one can introduce and solve a modified problem
\[
\mcl{L}_{\eps}u_{\eps}=f\quad \text{for}\,\, |\eps|>0
\,\,\text{small}
\]
such that $\mcl{L}_{\eps}$ is strictly elliptic or hyperbolic on either side of $\Gam_{d}$ and $\mcl{L}_{\eps}=\mcl{L}$ when $\eps=0$ then patch two solutions obtained from passing to the limit of $u_{\eps}$ as $|\eps|$ tends to 0 on each side of $\Gam_d$. See \cite{bae2009regularity, wang2019smooth} and references therein.

In \cite{Kuzmin}, however, it is proved by the method of a singular perturbation that a Keldysh-type equation of the form
\begin{equation*}
	k(x_1,x_2)\der_{11}u+\der_{2}(a(x_1,x_2)\der_{2}u)-\alpha(x_1,x_2) \der_{1}u +c(x_1,x_2)u=f
\end{equation*}
has a solution $u\in H^4$ provided that smooth coefficients $k,a,\alp$ and $c$ satisfy the following conditions:
\begin{itemize}
\item[(i)] $2\alp+(2j-1)\der_1k\ge \delta$ for $j=0,1,2,3$;
\item[(ii)] $\der_1a\le 0$, $c\le 0$, $\der_1c\ge 0$
\end{itemize}
for some constant $\delta>0$.
In \cite{Kuzmin}, this equation is derived from the potential flow equation by a hodograph transformation, when a transonic flow passes through a convergent-divergent nozzle (see  \cite[Chapter 2]{Kuzmin}).

In the first part of this paper, we prove the following theorem as a generalization of the result from \cite{Kuzmin}.

\iffalse
When the background flows satisfy certain accelerating conditions, it was proved that the boundary value problem for this equation  is well-posed.

Our first goal in this paper is to establish the well-posedness theory for a quite general class of mixed type PDEs. Then we use this general theory to study the stability of smooth transonic flows of the Euler-Poisson system. The underlying structures of the nonlinear PDEs of mixed type have been one of the main motivations in developing new methods/techniques/ideas for unified approaches.

Let linear differential operator $\mcl{L}$ be defined as follows
\begin{equation}
	\label{definition:L-general}
	\mcl{L}u:=a_{11}\der_{11}u+2a_{12}\der_{12}u
	+\der_{22}u+a_1\der_1u,
\end{equation}
where the coefficients $a_{ij}(i,j=1,2)$ and $a_1$ satisfy
\begin{enumerate}[(i)]
	\item	$a_{11}\ge \mu$ on $\Gamen$, $a_{11}\le -\mu$ on $\Gamexg$ for some constant $\mu>0$;
	\item $a_{12}=0$ on $\Gamw$.
\end{enumerate}

Given a constant $R>0$, define
\begin{equation}\label{defOmega}
	\begin{aligned}
		&\Om=\{\rx=(x_1, x_2)\in \R^2: 0<x_1<R,\,\,|x_2|<1\},\\
	\end{aligned}
\end{equation}
with
\begin{equation}\label{defboundary}
	\Gamen:=\der\Om\cap \{x_1=0\},\quad  \Gamw:=\der\Om\cap\{|x_2|=1\},\quad
	\Gamexg:=\der\Om\cap\{x_1=R\}.
\end{equation}
\fi

\begin{theorem} \label{theorem-1}
	Let $m\ge 4$ be fixed in $\mathbb{N}$. Assume that
	$a_{ij}, a_1\in W^{1,\infty}(\Om)\cap H^{m-1}(\Om)$ and $f\in H^{m-1}(\Omega)$.
	Suppose that there exists a constant $\lambda>0$ to satisfy
	\begin{equation}\label{KZ-condition0}
		a_1-\der_2 a_{12}+\frac{(2l-3)}{2}\der_1a_{11}\le -\lambda \quad\tx{in $\Om$} \,\,\text{for all } l=1,\cdots, m.
	\end{equation}
	In addition, if it holds that
	\begin{equation*}
		\min_{|x_2|\le 1} (a_{11}-a_{12}^2)(0,x_2)>0\,\,\tx{on $\Gamen$},\quad a_{11}<0\,\,\tx{on $\Gamexg$},
	\end{equation*}
	and
	\begin{equation*}
		a_{12}=\der_2^lf=0\quad\tx{on $\Gamw$ for $l=2j-1$ with $j\in \mathbb{N}$ and $l<m$},
	\end{equation*}
	then the boundary value problem
	\begin{equation}\label{bvp-g0}
		\begin{cases}
			\mcl{L} u=f\quad&\mbox{in $\Om$},\\
			u=0\quad&\mbox{on $\Gamen$},\quad
			\der_2u=0\quad\mbox{on $\Gamw$}
		\end{cases}
	\end{equation}
	has a unique solution $u\in H^m(\Om)$.
    \iffalse
    Furthermore, there exists a constant $C>0$ depending only on $\displaystyle{\mu, \Om, \|(a_{ij}, a_1)\|_{H^{m-1}(\Om)}}$, and $ m$ to satisfy the estimate
	\begin{equation*}
		\|u\|_{H^m(\Om)}\le C\|f\|_{H^{m-1}(\Om)}.
	\end{equation*}
    \fi
\end{theorem}

A more precise statement is given as Theorem \ref{theorem-1-full}
\iffalse
There are some remarks for Theorem \ref{theorem-1}.

\begin{remark}
The analysis developed in this paper for a class of mixed type PDEs of Keldysh type should be also useful for the study for the problems arising in fluid mechanics and geometry.
\end{remark}

\begin{remark}
	The PDE studied in this paper is more general than that in \cite{Kuzmin}, which appears naturally in the study for the Euler-Poisson system.
\end{remark}

\begin{remark}
	The condition appears naturally when the equation is the linearzied equation for the Euler-Poisson system with repsect an accelerating continuous transonic flows.
\end{remark}

Note that  many nonlinear PDEs arising in mathematics and science are no longer of standard type, but are in fact of mixed type.
Although there are some studies for the linear mixed type PDEs,  the analysis of nonlinear PDEs of mixed type is still in its early stages, and most nonlinear mixed type problems are wide open and ripe for the development of new ideas, methods, and techniques.
\fi

In contrast to the linear Keldysh-type equation, the degenerate interface of a nonlinear Keldysh-type equation (e.g. \eqref{new system with H-decomp2}) is a priori unknown, and one may expect that the nonlinearity may cause additional singularities of solutions.
\iffalse
Thus, many classical methods and techniques for linear PDEs no longer work directly for mixed-type nonlinear PDEs.
\fi
The lack of a universal method is a big obstacle in dealing with the elliptic phase and the hyperbolic phase together for mixed-type nonlinear PDEs with degeneracy. Over the last several decades, the PDE community has been largely partitioned according to the mathematical approaches taken to the analysis of different types of PDEs (elliptic/parabolic/hyperbolic). However, advances in the analysis of nonlinear PDEs over the past few decades have shown that many of the challenging questions confronting the community extend beyond this classification. More specifically, many nonlinear PDEs that arise in long-standing open problems in various fields are of mixed type, often exhibiting degeneracy in many cases. For example, there is a transonic flow problem, called a {\emph{de Laval nozzle flow problem}}, in gas dynamics.
\iffalse
The solutions to these problems will advance our understanding of transonic flows,  and lead to significant developments of these areas and related mathematics. To achieve these goals, a deep understanding of the underlying nonlinear PDEs of mixed type (for instance, the solvability, the properties of solutions, etc.) is the key issue.
\fi

Quasi-one-dimensional analysis of the compressible Euler system shows that when a subsonic flow of an inviscid compressible gas enters a convergent-divergent nozzle, with the minimal cross-section area at the throat, and if the pressure at the nozzle exit is sufficiently high, the flow transitions from subsonic to sonic at the throat and subsequently becomes supersonic in the expanding section of the nozzle. This transition is followed by the occurrence of a transonic shock, after which the flow reverts to a subsonic state. Numerous experiments and simulations have shown that this phenomenon occurs in practice; however, a rigorous mathematical study focusing on the stability and instability analysis of this phenomenon has yet to be established due to the limited results available on nonlinear mixed-type PDEs with degeneracy.
\iffalse
One of the most important nonlinear mixed type equations appeared in the study of transonic flows.   For example, a long-standing open problem related to the compressible Euler system is to prove the existence of a solution which transits from supersonic region into subsonic region continuously in a {\emph{de Laval nozzle}}.
\fi
In the pioneering work \cite{Morawetz1, Morawetz2, Morawetz3, Morawetz64} by Morawetz, it is shown that two-dimensional smooth transonic potential flows around an airfoil are generally structurally unstable. Smooth transonic solutions for transonic small disturbance equations were analyzed in \cite{Kuzmin}. Recently, transonic solutions for the Euler system with sonic interfaces have been studied; see \cite{WangXin1, WangXin2, WangXin, WengXin}, etc.

In this paper, instead of geometric effect caused by de Laval nozzle, we study smooth transonic flows under the electric field.  More specifically, we prove the structural stability of a one-dimensional smooth transonic solution to the Euler-Poisson system, which exhibits a continuous transition from the subsonic state to the supersonic state across a sonic interface.  We conjecture that the sign change of the electric field in the Euler-Poisson system has the same effect as the sign change in the derivative of cross-sectional area of de Laval nozzle in the Euler system. The results obtained in this paper provide evidence that indicates the existence of a smooth transonic flow near the throat region in a de Laval nozzle and its structural stability.

The steady Euler-Poisson system
%, which can be used to describe the motion of plasma or electrons in semiconductor devices (cf. \cite{MarkRSbook}), has the following form for the steady solutions
\begin{equation}\label{steadyEP0}
	\left\{
	\begin{aligned}
		& \Div_{\rx} (\rho \bu)=0 \\
		&\Div_{\rx} (\rho \bu \otimes \bu) +\nabla_{\bx} p=\rho \nabla_{\rx} \Phi \\
		& \Div_{\rx}(\rho\msE \bu +p\bu)=\rho \bu\cdot \nabla_{\bx}\Phi\\
		& \Delta_{\rx} \Phi=\rho-\barrhoi%
	\end{aligned}%
	\right.
\end{equation}
where $\rho$, ${\bf u}$, $p$, and $\msE$ represent the density, velocity, pressure, and the energy density of electrons, respectively, describes the motion of plasma or electrons in semiconductor devices (cf. \cite{MarkRSbook}). The Poisson equation $\tx{\eqref{steadyEP0}}_4$ describes how the two fluids of electrons and ions interact through the electric field $\nabla \Phi$. In particular, we shall consider ideal polytropic gas, where $p$ and $\msE$ are given by
\begin{equation}
	p=S\rho^{\gam}\quad\tx{and}\quad \msE=\frac{|\bu|^2}{2}+\frac{p}{(\gam-1)\rho},
\end{equation}
respectively, for a function $S>0$ and a constant $\gam>1$. The function $\ln S$ represents {\emph{the physical entropy}}, and the constant $\gam>1$ is called the {\emph{adiabatic exponent}}. The Mach number $M$ is  defined by
$$M:=\frac{|{\bf u}|}{\sqrt{\gam p/\rho}}.$$
If $M<1$, then the flow corresponding to $(\rho, {\bf u}, p, \Phi)$ is said to be subsonic. On the other hand, if $M>1$, then the flow is said to be supersonic. Finally, if $M=1$, then the flow is said to be sonic.

The continuous one-dimensional transonic solution to \eqref{steadyEP0} was first found by phase-plane analysis in \cite{LuoXin}. In this paper, we show that the solution is $C^{\infty}$-smooth at the points of $M=1$ (Lemma \ref{lemma-1d-full EP}). For other results on continuous transonic solutions to the one-dimensional Euler-Poisson system, one can refer to \cite{Gamba, GambaMorawetz, ChenMeiZhangZhang, LiMeiZhangZhang2,WeiMeiZhangZhang} and the references therein.

Now, one can raise the question whether this continuous transonic solution is stable under multidimensional perturbations of the boundary data.
Based on the results of \cite{Kuzmin}, we infer that the answer to this question is possibly yes, though much remains unknown. The goal of this work is to prove the existence of a two-dimensional solution to \eqref{steadyEP0} that produces a flow that accelerates from a subsonic state to a supersonic state in a finite-length flat nozzle, and to show that the transonic transition occurs continuously across a sonic interface, given as a curve. Furthermore, we show that the derivatives of $\rho, {\bf u}$ and $p$ are also continuous across the sonic interface.
\iffalse
In this paper, we prove the existence of two-dimensional solutions to the steady Euler-Poisson system with {\emph{continuous transonic transitions}} across {\emph{sonic interfaces}} of codimension 1.
\fi
The second main theorem of this paper can be summarized as follows.
\begin{theorem}\label{theorem-2-pre}
	The smooth one-dimensional accelerating transonic solution to \eqref{steadyEP0} is structurally stable in the sense that whenever boundary data are prescribed as small perturbations of the smooth transonic state, one can construct a corresponding classical two-dimensional solution to \eqref{steadyEP0} which is sufficiently close to the one-dimensional solution in an appropriately defined norm. Furthermore, the sonic interface $\Gamma_{\rm sn}=\{\rx=(x_1, x_2)\in \R^2:M(\rx)=1\}$ is given as a regular curve which is prescribed as a graph of a function $x_1=f_{\rm sn}(x_2)$.
    \iffalse
    If the Mach number of an accelerating background flow is close to $1$, then it must be structurally stable under two-dimensional perturbations.
    \fi
\end{theorem}

\iffalse
It was observed that the accelerating flows could be structurally stable in \cite{Kuzmin}.
%However, it is not easy to get an accelerating flow in a genereal two-dimensional setting.
An interesting discovery in \cite{LuoXin} is that the electric field can help the one-dimensional Euler-Poisson system to have an accelerating smooth transonic solution. Therefore, a natural question is whether these accelerating flows are stable or not.
\fi
%Here, the function $S$ is to be determined by solving the system \eqref{UnsteadyEP}, and
\iffalse
The system \eqref{steadyEP0} consists of two parts, a compressible Euler system with source terms and a Poisson equation, coupled in a nonlinear way.  Differently from the Euler system, the steady Euler-Poisson system has one-dimensional continuous transonic solutions, and this can be found by a phase plane analysis \cite{Markphase, LuoXin}.  In this paper, we first show that the one-dimensional accelerating transonic solution is, in fact, $C^{\infty}$ smooth (Lemma \ref{lemma-1d-full EP}).

A natural question is whether these continuous transonic solutions for the Euler-Poisson system are stable under multidimensional perturbations of the boundary data.
\fi
The primary challenge in establishing this theorem is to solve a boundary value problem of a weakly coupled quasilinear system that consists of a second-order elliptic equation and a Keldysh-type equation. Our strategy is to establish the well-posedness of linearized boundary value problems, then solve the nonlinear problem by the method of iterations. Once an a priori $H^1$ estimate of solutions to linearized boundary value problems is achieved, we can apply Theorem \ref{theorem-1} and a bootstrap argument to prove the unique existence of classical solutions. Hence, the most important part in solving linearized boundary value problems is to establish an a priori $H^1$ estimate. To achieve this, it is essential to carefully analyze the structure of the Euler-Poisson system.

    We should point out that the sonic interface presented in this paper exhibits significant differences from the ones explored in previous works (see \cite{bae2013prandtl, CF2, CFbook, elling2008supersonic}).
    %Another example of weak solutions to the unsteady Euler system with sonic interfaces, which appear as circular arcs, are given in.
    In \cite{BCF}, it is shown that the sonic interfaces appearing in self-similar potential flow  are weak discontinuities in the sense that the flow velocity is $C^{0,1}$ but not $C^1$. %This means that the flow velocity fields are continuous but their derivatives are discontinuous  across the sonic arcs.
    For studies on self-similar flow of other models, one can refer to \cite{JegdicKeyfitz, KeyfitzTesdall, Zhengregular} and references therein.
%\end{remark}

\iffalse
\begin{remark}
	For the last several decades, there have been extensive studies on transonic shocks in  nozzles, see  \cite{bae2011transonic, ChenSXnozzle, LiXinYin_CMP, LiXinYin_ARMA, XinYin_CPAM, xin2008transonic}. In this setting, the system has discontinuous coefficients in the both sides of the shock.
\end{remark}
\fi

%\begin{remark}
	Before this work, the authors of this paper studied multidimensional subsonic solutions and supersonic solutions for the Euler-Poisson system; see \cite{BCF2, BDX3, BDX}. The ultimate goal of this work and the work in \cite{BCF2, BDX3, BDX} is to establish the existence of a multidimensional solution to the steady Euler-Poisson system with a continuous transonic-transonic shock configuration in a flat nozzle. Once this goal is achieved, one may try to construct a transonic solution to the steady Euler system in a de Laval nozzle.
%\end{remark}

\iffalse
Here we give an outline for the key ideas to prove Theorems \ref{theorem-1} and \ref{theorem-2}.
For the mixed equation with degeneracy of Keldysh type, we first look for a multiplier to establish the basic energy estimates. Then we use the singular perturbation methods to achieve the higher order estimates. Since the Euler-Poisson system is qn  mixed type-elliptic coupled system, the key issue is to establish the basic energy estimate.
\fi

The rest of the paper is organized as follows.
In Section \ref{Section:2}, we state a more precise version of Theorem \ref{theorem-1}, and give its proof. In Section \ref{Section:3}, we state a precise version of Theorem \ref{theorem-2-pre}. Finally, we give its proof in Section \ref{subsection:framework}.
\iffalse
The well-posedness for a Keldysh type PDE is established in Section .  In Section, we give more detailed description for the Euler-Poisson system and main results for the stability of smooth transonic flows. The key estimate for the linearized problem of the Euler-Poisson system is presented in Section 4. Section 5 contains the higher order estimates and existence of solutions for the nonlinear problem. In Section 6, we give the proof of the stability of transonic flows even when the vorticity is nonzero. In the appendix, we show that the nozzle length could be pretty long for the background solutions we study in this paper.
\fi

%\part{PDE with a degeneracy of a Keldysh type}\label{part:1}

\section{Well-posedness for a boundary value problem of  linear Keldysh equation}\label{Section:2}

\iffalse
In this section, we mainly study the problem and give the proof of Theorem .
\fi

\subsection{Main theorem}
\begin{comment}
Given a constant $R>0$, define
\begin{equation*}
\begin{aligned}
  &\Om=\{\rx=(x_1, x_2)\in \R^2: 0<x_1<R,\,\,|x_2|<1\},\\
  &\Gamen:=\der\Om\cap \{x_1=0\},\quad  \Gamw:=\der\Om\cap\{|x_2|=1\},\quad
  \Gamexg:=\der\Om\cap\{x_1=R\}.
  \end{aligned}
\end{equation*}
\medskip

For coefficients $a_{ij}(i,j=1,2)$ and $a_1$ satisfying
\begin{itemize}
\item[(i)] $a_{ij}, a_1\in W^{1,\infty}(\Om)\cap H^{m-1}(\Om)$;
\item[(ii)]
$a_{11}\ge \mu\,\,\tx{on $\Gamen$},\quad a_{11}\le -\mu\,\,\tx{on $\Gamexg$}$ for some constant $\mu>0$;
\item[(iii)] $a_{12}=0$ on $\Gamw$,
\end{itemize}
we define a linear differential operator
\begin{equation}
\label{definition:L-general}
  \mcl{L}u:=a_{11}\der_{11}u+2a_{12}\der_{12}u
  +\der_{22}u+a_1\der_1u.
\end{equation}
\end{comment}

For the differential operator $\mcl{L}$ given by \eqref{definition:L-general}, we consider the boundary value problem
\begin{equation}\label{bvp-g}
  \begin{cases}
  \mcl{L}u=f\quad&\mbox{in $\Om$},\\
  u=0\quad&\mbox{on $\Gamen$},\\
  \der_2u=0\quad&\mbox{on $\Gamw$},
  \end{cases}
\end{equation}
for $\Om,\, \Gamen,\, \Gamw,\, \Gamexg$ given by \eqref{defOmega} and \eqref{defboundary}.
\begin{definition}
\label{definition:Kz coefficient}

 Given a constant $l\in \mathbb{N}$, we define
 %the functions
\begin{equation*}
\kz_l(\rx):=a_1+\frac{(2l-3)}{2}\der_1a_{11}-\der_2 a_{12}.
\end{equation*}
\end{definition}
This definition is inspired by \cite{Kuzmin}.

\begin{condition}
\label{condition:1.2}
Suppose that the coefficients $(a_{11}, a_{12}, a_1)$ of the operator $\mcl{L}$ satisfy the following conditions.
\begin{itemize}
\item[(i)] There exists a constant $\lambda>0$ to satisfy
\begin{equation}\label{KZ-condition}
 \max_{\ol{\Om}}\kz_l\le -\lambda \quad \text{for all } l=1,2,\cdots, m.
\end{equation}

\item[(ii)] %the differential operator $\mcl{L}$ changes its type in $\Om$ such that
    $\displaystyle{
 \min_{\Gamen} (a_{11}-a_{12}^2)>0}$ and  $\displaystyle{\max_{\Gamexg}a_{11}<}0$

\item[(iii)] $\displaystyle{ a_{12}=0}$ on $\Gamw$
\end{itemize}
%In addition, assume that $\mcl{L}$ satisfies the following small perturbation condition:
\begin{itemize}
\item[(iv)] If $m=4$, then there exists a function $\bar{a}_{11}\in C^2(\ol{\Om})$, and a small constant $\bar{\delta}>0$ such that
    \begin{equation}
\label{G: perturbation condition for m=4}
  \|(a_{11}, a_{12})-(\bar a_{11}, 0)\|_{H^{m-1}(\Om)}\le \bar{\delta},
\end{equation}
or if $m>4$, then there exists a small constant $\bar{\delta}>0$ such that
    \begin{equation}
\label{G: perturbation condition for m>4}
  \|a_{12}\|_{H^{m-1}(\Om)}\le \bar{\delta}.
\end{equation}
\end{itemize}
%Finally, assume the compatibility conditions:
\begin{itemize}
\item[(v)] For all $j\in \mathbb{N}$ satisfying $2j\le m-1$,
\begin{equation}
\label{comp for even ext}
\der_2^{2j-1}a_{11}=\der_2^{2j-1}a_{1}=\der_2^{2j-1}f=0\quad\tx{on $\Gamw$}.
\end{equation}
And, for all $k\in \mathbb{N}$ satisfying $2k+1\le m-1$,
\begin{equation}
\label{comp for odd ext}
  \der_2^{2k}a_{12}=0\quad\tx{on $\Gamw$}.
\end{equation}
\end{itemize}
\end{condition}

\begin{definition}
We shall call $\lambda>0$ from \eqref{KZ-condition} a $\kz$-constant of order $m$ for the differential operator $\mcl{L}$.
\end{definition}

\begin{comment}
\begin{theorem} \label{theorem-1}
Let $m\ge 4$ be fixed in $\mathbb{N}$.
\smallskip

Then, one can fix a small constant $\bar{\delta}>0$ depending only on $\lambda$, $\Om$, $\|(a_{ij}, a_1)\|_{H^{m-1}(\Om)}, m)$(also depending on $\|\der_{11}\bar a_{11}\|_{L^{\infty}(\ol{\Om})}$ if $m=4$) so that the boundary value problem \eqref{bvp-g} has a unique solution $u\in H^m(\Om)$. Furthermore, the solution satisfies the estimate
\begin{equation*}
  \|u\|_{H^m(\Om)}\le C\|f\|_{H^{m-1}(\Om)}
\end{equation*}
for $C>0$ depending only on $\displaystyle{(\lambda, \Om, \|(a_{ij}, a_1)\|_{H^{m-1}(\Om)}, m)}$(also depending on $\|\der_{11}\bar a_{11}\|_{L^{\infty}(\ol{\Om})}$ if $m=4$).

\end{theorem}
\end{comment}

\begin{theorem} \label{theorem-1-full}
	Fix $m\in \mathbb{N}$ with $m\ge 4$. Assume that
	$a_{11},a_{12}, a_1\in H^{m-1}(\Om)(\subset W^{1,\infty}(\Om))$ and $f\in H^{m-1}(\Omega)$. Under Condition \ref{condition:1.2},
    \iffalse
    Suppose that there exists a constant $\lambda>0$ to satisfy
	\begin{equation}\label{KZ-condition0}
		a_1-\der_2 a_{12}+\frac{(2k-1)}{2}\der_1a_{11}\le -\lambda \quad\tx{in $\Om$} \,\,\text{for all } k=0,1,\cdots, m-1.
	\end{equation}
	If, In addition, it holds that
	\begin{equation*}
		\min_{|x_2|\le 1} (a_{11}-a_{12}^2)(0,x_2)>0\,\,\tx{on $\Gamen$},\quad a_{11}<0\,\,\tx{on $\Gamexg$},
	\end{equation*}
	and
	\begin{equation*}
		a_{12}=\der_2^lf=0\quad\tx{on $\Gamw$ for $l=2j-1$ with $j\in \mathbb{N}$ and $l<m$},
	\end{equation*}
    \fi
	the boundary value problem \eqref{bvp-g}
    \iffalse
	\begin{equation}\label{bvp-g0}
		\begin{cases}
			\mcl{L} u=f\quad&\mbox{in $\Om$},\\
			u=0\quad&\mbox{on $\Gamen$},\quad
			\der_2u=0\quad\mbox{on $\Gamw$}
		\end{cases}
	\end{equation}
    \fi
	has a unique solution $u\in H^m(\Om)$. Furthermore, there exists a constant $C>0$ depending only on $\displaystyle{\lambda,  \|(a_{ij}, a_1)\|_{H^{m-1}(\Om)}, \bar{\delta}}$, $\|D^2\bar{a}_{11}\|_{L^{\infty}(\Om)}$ (if $m=4$), $\Om$ and $ m$ to satisfy the estimate
	\begin{equation*}
		\|u\|_{H^m(\Om)}\le C\|f\|_{H^{m-1}(\Om)}.
	\end{equation*}
\end{theorem}

In \S \ref{subsec:2.1}, we introduce a boundary value problem with a singular perturbation of \eqref{bvp-g}, and use it to establish an a priori global $H^1$-estimate and local $H^k$-estimate($k=2,\cdots, m$) of solutions to \eqref{bvp-g}. In \S \ref{subsec:2.2}, we introduce an extension of the boundary value problem given in \S \ref{subsec:2.1} to establish a global $H^k$-estimate($k=2,\cdots, m$) of solutions to \eqref{bvp-g}. Finally, a proof of Theorem \ref{theorem-1-full} is given in \S \ref{subsec:2.3}.

%\section{The well-posedness of BVP \eqref{bvp-g}}

\subsection{Boundary value problem with a singular perturbation } \label{subsec:2.1}
\iffalse
To prove Theorem \ref{theorem-1} by a limiting argument, we introduce a boundary value problem with a singular perturbation:
\fi
For a constant $\eps>0$, consider a boundary value problem for $v^{\eps}$:
\begin{equation}\label{bvp-aux}
  \begin{cases}
  \eps\der_{111}v^{\eps}+\mcl{L}v^{\eps}=f\quad&\mbox{in $\Om$}\\
  v^{\eps}=0,\quad \der_1 v^{\eps}=0 \quad&\mbox{on $\Gamen$}\\
  \der_2v^{\eps}=0\quad&\mbox{on $\Gamw$}\\
  \der_{11}v^{\eps}=0\quad&\mbox{on $\Gamexg$}.
  \end{cases}
\end{equation}

\subsubsection{A priori $H^1$-estimate of $v^{\eps}$}
\begin{lemma}
\label{lemma G:wp of singular pert prob-main}
There exists a constant $\bar{\eps}>0$ depending only on $\lambda$ so that whenever the inequality $0<\eps<\bar{\eps}$ holds, if $v^{\eps}$ is a smooth solution to \eqref{bvp-aux}, then it satisfies the estimate
\begin{equation}
\label{G:a priori estimate1 of vm and wm}
\begin{split}
  &\sqrt{\eps}\|\der_{11}v^{\eps}\|_{L^2(\Om)}+\|\der_2 v^{\eps}\|_{L^2(\Gamexg)}
  +\|v^{\eps}\|_{H^1(\Om)}\\
  &\le C(\lambda, \|a_{11}\|_{L^{\infty(\Om)}}, R) \|f\|_{L^2(\Om)}.
  \end{split}
\end{equation}

\end{lemma}

\begin{proof} Define
\begin{equation*}
  \mcl{L}_{\eps}v:=\eps\der_{111}v+\mcl{L}v.
\end{equation*}
For simplicity, put $v:=v^{\eps}$.
\medskip

For a constant $\mu>0$ to be determined, we integrate by parts to get
\begin{equation*}
  \begin{split}
  &\int_{\Om}e^{-\mu x_1}\der_1v \mcl{L}_{\eps}v\,d\rx\\
  &=\int_{\Gamexg}e^{-\mu R}\left(\frac{a_{11}}{2}(\der_1v)^2-(\der_2v)^2\right)\,dx_2\\
  &\phantom{=}-\int_{\Om}\eps e^{-\mu x_1} \left((\der_{11}v)^2-\mu \der_1v\der_{11}v\right)\,d\rx\\
  &\phantom{=}+\int_{\Om}e^{-\mu x_1}\left(\kz_1+\frac{\mu}{2} a_{11}\right)(\der_1 v)^2-\frac{\mu}{2}e^{-\mu x_1}(\der_2 v)^2\,d\rx.
  \end{split}
\end{equation*}

For the $\kz$-constant $\lambda>0$ given from \eqref{KZ-condition}, fix $\mu$ as
\begin{equation*}
  \mu:=\min\left\{\frac{\lambda}{4\|a_{11}\|_{L^{\infty}(\Om)}},1\right\}.
\end{equation*}
%Note that $a_{11}\in H^{m-1}(\Om)\subset C^{0}(\ol{\Om})$ for $m\ge 4$ because $\Om$ is a domain in $\R^2$.
By Condition \ref{condition:1.2}(ii) and Cauchy-Schwarz inequality,
\begin{equation*}
  \begin{split}
  &\int_{\Om} e^{-\mu x_1}\der_1v \mcl{L}_{\eps}v\,d\rx\\
  &\le -\int_{\Gamexg}(\der_2v)^2\,dx_2-\frac{\eps}{2}\int_{\Om} e^{-\mu x_1}(\der_{11}v)^2\,d\rx\\
  &\phantom{\le} -\int_{\Om} e^{-\mu x_1}\left(\left(\frac{3\lambda}{4} -\frac{\eps}{2}\right)(\der_1 v)^2+\frac{\lambda}{8\|a_{11}\|_{L^{\infty}(\Om)}}(\der_2 v)^2\right)\,d\rx.
  \end{split}
\end{equation*}
So we can easily see from
$\displaystyle{
  \int_{\Om} e^{-\mu x_1}\der_1 v\mcl{L}_{\eps}v\,d\rx=\int_{\Om} e^{-\mu x_1}f \der_1 v\,d\rx}$
that there exists a small constant $\bar{\eps}>0$ depending only on $\lambda$ such that whenever $0<\eps\le \bar{\eps}$, we have
 \begin{equation*}
   \sqrt{\eps}\|\der_{11}v\|_{L^2(\Om)}+\|\der_2 v\|_{L^2(\Gamexg)}+\|Dv\|_{L^2(\Om)}\le C(\lambda, \|a_{11}\|_{L^{\infty(\Om)}})\|f\|_{L^2(\Om)}.
 \end{equation*}
Then we apply Poincar\'{e} inequality to conclude that
\begin{equation*}
  \sqrt{\eps}\|\der_{11}v\|_{L^2(\Om)}+\|\der_2 v\|_{L^2(\Gamexg)}+\|v\|_{H^1(\Om)}\le C(\lambda, \|a_{11}\|_{L^{\infty(\Om)}}, R)\|f\|_{L^2(\Om)}.
\end{equation*}

\end{proof}

\subsubsection{Local $H^2$-estimate of $v^{\eps}$}
Set
\begin{equation*}
  \kappa_0:=\frac 12 \min_{|x_2|\le 1}(a_{11}-a_{12}^2)(0, x_2)>0.
\end{equation*}
Since $a_{11}$ and $ a_{12}$ are continuous in $\ol{\Om}$, one can fix two constants $\kappa_1>0$ and $R_1\in(0, R)$ such that
\begin{equation}
\label{G:ellipticity}
\begin{bmatrix}
a_{11}&a_{12}\\
a_{12}&1
\end{bmatrix}
\ge \kappa_1\mathbb{I}_{2\times 2} \quad\tx{in $\ol{\Om\cap\{x_1<R_1\}}$}.
\end{equation}

\begin{lemma}
\label{G:lemma for pre H2 estimate of vm, part 1}
Let $\bar{\eps}$ be from Lemma \ref{lemma G:wp of singular pert prob-main}. Suppose that $v^{\eps}$ is a smooth solution to \eqref{bvp-aux} for  $\eps\in(0,\bar{\eps}]$. Then, for any $r\in(0, \frac{R_1}{8}]$, we have
\begin{equation}
\label{G:viscous-estimate-intermediate1}
  \|D^2 v^{\eps}\|_{L^2(\Om\cap \{2r<x_1<R_1-2r\})}
  \le C\|f\|_{L^2(\Om)}
\end{equation}
for $C=C(\kappa_0,\lambda, \|(a_{11}, a_{12}, a_1)\|_{L^{\infty}(\Om)}, R, r)$. Note that the estimate constant $C$ is independent of $\eps\in(0, \bar{\eps}]$.

\begin{proof}
%[Proof of Lemma \ref{G:lemma for pre H2 estimate of vm, part 1}]
Given a constant $r\in(0, \frac{R_1}{8}]$, define a cut-off function $\chi_r\in C^{\infty}(\R)$ satisfying the following properties:
\begin{equation}
\label{G:property of cut-off function in singular perturb.}
\begin{split}
 \chi_r(x_1)=\begin{cases}
  1\quad&\mbox{for $2r\le x_1\le R_1-2r$}\\
  0\quad&\mbox{for $x_1\le r$ or $x_1\ge R_1-r$}
  \end{cases}\quad\tx{and}\quad
0\le \chi_r\le 1\,\,\tx{on $\R$}.
  \end{split}
\end{equation}
One can fix $\chi_r$ such that, for each $k\in \mathbb{N}$,
\begin{equation*}
\|\chi_r\|_{C^k(\R)}\le C_k\left(1+\frac{1}{r^k}\right)
\end{equation*}
for some constant $C_k>0$ depending only on $k$.
\medskip

Consider
\begin{equation}
\label{G:integral for vm with viscosity1}
  \int_{\Om} \chi_r^2\der_{11}v \mcl{L}_{\eps}v\,d\rx=\int_{\Om} f\chi_r^2\der_{11}v\,d\rx.
\end{equation}
Integrating by parts,
\begin{equation*}
\begin{split}
\tx{LHS of \eqref{G:integral for vm with viscosity1}}
&=\int_{\Om} \chi_r\chi_r'(2\der_2v\der_{12}v-\eps(\der_{11}v)^2)\,d\rx\\
&\phantom{=}+\int_{\Om}\left(a_{11}(\der_{11}v)^2+2a_{12}\der_{11}v\der_{12}v+(\der_{12}v)^2+a_1\der_1v\der_{11}v\right)\chi_r^2\,d\rx.
\end{split}
\end{equation*}
By \eqref{G:ellipticity} and the Cauchy-Schwarz inequality,
\begin{equation*}
\begin{split}
&\tx{LHS of \eqref{G:integral for vm with viscosity1}}\\
&\ge \kappa_1\|\chi_r D\der_1 v\|_{L^2(\Om)}^2
-C(\kappa_0, \|a_1\|_{L^{\infty(\Om)}}, r)\left(\|Dv\|_{L^2(\Om)}^2+\|\sqrt{\eps}\der_{11}v\|_{L^2(\Om)}^2\right).
\end{split}
\end{equation*}
Then it follows from \eqref{G:integral for vm with viscosity1} and Lemma \ref{lemma G:wp of singular pert prob-main} that
\begin{equation*}
 \|\chi_r D\der_1 v\|_{L^2(\Om)}\le
 C(\kappa_0, \|a_1\|_{L^{\infty(\Om)}}, \lambda, \|a_{11}\|_{L^{\infty(\Om)}}, R, r)\|f\|_{L^2(\Om)}.
\end{equation*}
By using this estimate and solving the equation $\mcl{L}_{\eps}v=f$ for $\der_{22}v$, we obtain the estimate \eqref{G:viscous-estimate-intermediate1}.

\end{proof}

\end{lemma}

\begin{lemma}%[$L^2$-estimate of $\der_{11}v_m$ up to $\Gamex$]
\label{G:lemma for pre H2 estimate of vm, part 2}
Under the same assumptions as Lemma \ref{G:lemma for pre H2 estimate of vm, part 1}, for any $\eps\in(0,\bar{\eps}]$ and $r\in(0, \frac{R}{8}]$,
we have
\begin{equation}
\label{G:viscous-estimate-intermediate2}
\begin{split}
  &\sqrt{\eps} \|\der_{111} v^{\eps}\|_{L^2(\Om\cap\{x_1>2r\})}
  %+\|\der_{12}v^{\eps}\|_{L^2(\Gamexg)}
  +\|D^2v^{\eps}\|_{L^2(\Om_L\cap\{x_1>2r\})}
  \le C\|f\|_{H^1(\Om)}
\end{split}
\end{equation}
for $C=C(\kappa_0,\lambda, \|(a_{11}, a_{12}, a_1)\|_{L^{\infty(\Om)}}, R, r)$ provided that
a constant $\bar{\delta}>0$ from (iv) of Condition \ref{condition:1.2} is fixed sufficiently small depending only on $\lambda$ and $\|a_{11}\|_{L^{\infty}(\Om)}$. (Here, we choose $\bar{\delta}$ to make the term $\|a_{12}\|_{H^{m-1}(\Om)}$ small.)

\end{lemma}

\begin{proof}

For a constant $r\in (0, \frac R8]$, define a cut-off function $\zeta_r\in C^{\infty}(\R)$ such that
\begin{equation*}
 \zeta_r(x_1)=\begin{cases}
  0\quad&\mbox{for}\quad x_1\le r\\
  1\quad&\mbox{for}\quad x_1\ge 2r
  \end{cases},\quad
  0\le \zeta_r\le 1,\quad
  \tx{and} \quad \zeta_r'\ge 0\quad\tx{on $\R$}.
\end{equation*}
One can fix $\zeta_r$ such that, for each $k\in \mathbb{N}$,
\begin{equation*}
\|\zeta_r\|_{C^k(\R)}\le C_k\left(1+\frac{1}{r^k}\right),
\end{equation*}
for a constant $C_k>0$ fixed depending only on $k$.

For a constant $\mu>0$ to be determined later (differently from the proof of Lemma \ref{lemma G:wp of singular pert prob-main}), consider
\begin{equation}
\label{G:Kz H-2 estimate}
  \int_{\Om} \zeta_r^2e^{-\mu x_1}\der_{111}v \mcl{L}_{\eps}v\,d\rx=\int_{\Om} f\zeta_r^2 e^{-\mu x_1}\der_{111}v\,d\rx.
\end{equation}

We integrate by parts, and use Lemmas \ref{lemma G:wp of singular pert prob-main} and \ref{G:lemma for pre H2 estimate of vm, part 1} to get
\begin{equation*}
  \begin{split}
  \left|\int_{\Om}\zeta_r^2 e^{-\mu x_1}f\der_{111}v\,d\rx\right|\le C\|f\|_{H^1(\Om)} \|\zeta_r\der_{11}v\|_{L^2(\Om)}
  \end{split}
\end{equation*}
for $C=C(\kappa_0,\lambda, \|(a_{11}, a_{12}, a_1)\|_{L^{\infty(\Om)}}, R, r)$.

By a lengthy but straightforward computation on the left-hand side of \eqref{G:Kz H-2 estimate}, we have
\begin{equation*}
  \begin{split}
  &\int_{\Om} \zeta^2 e^{-\mu x_1} \der_{111}v\mcl{L}_{\eps}v\,d\rx\\
  &=\int_{\Om} \eps \zeta^2 e^{-\mu x_1}(\der_{111}v)^2\,d\rx\\
  &\phantom{=}+\int_{\Om} e^{-\mu x_1}\zeta_r^2\left(\left(-\kz_2+\frac{\mu}{2} a_{11}\right)(\der_{11}v)^2
  +\mu (\der_{12}v)^2-\left(2{\der_1 a_{12}}-2\mu a_{12}\right)\der_{11}v\der_{12}v\right)\,d\rx\\
  &\phantom{=}+\int_{\Gamexg} \zeta_r^2 e^{-\mu x_1}\frac{(\der_{12}v)^2}{2}\,dx_2
  -\int_{\Gamexg}\mu\zeta_r^2 e^{-\mu x_1}\der_{12}v\der_2 v\,dx_2\\
  &\phantom{=}+\int_{\Om} \der_1\left((-2\zeta_r\zeta_r'+\mu \zeta_r^2)e^{-\mu x_1}\right)\der_2 v\der_{12}v-2\zeta_r\zeta_r'e^{-\mu x_1}\left(\frac{(\der_{11}v)^2}{2}+(\der_{12}v)^2+2a_{12}\der_{12}v\der_{11}v\right)\,d\rx\\
&\phantom{=}+  \int_{\Om} (-2\zeta_r\zeta_r'+\mu\zeta_r^2)e^{-\mu x_1}a_1\der_1 v\der_{11}v\,d\rx
  \end{split}
\end{equation*}
for $\kz_2$ given by Definition \ref{definition:Kz coefficient}.
Then, we use Condition \ref{condition:1.2} (i), and apply Lemmas \ref{lemma G:wp of singular pert prob-ext} and \ref{G:a priori estimate1 of vm and wm} to get
\begin{equation*}
  \begin{split}
 &\int_{\Om}\zeta_r^2 e^{-\mu x_1}\der_{111}v\mcl{L}_{\eps}v\,d\rx\\
 &\ge \int_{\Om} \zeta_r^2 e^{-\mu x_1}\left(\eps(\der_{111}v)^2+
 \left(\frac{\lambda}{2}-\mu\|a_{11}\|_{L^{\infty}(\Om)}\right)(\der_{11}v)^2\right)\,d\rx\\
 &\phantom{=}+\int_{\Om}\zeta_r^2 e^{-\mu x_1}\left(\frac{\mu}{2}-\frac{4\|a_{12}\|_{W^{1,\infty}(\Om)}^2(1+\mu)^2}{\lambda}\right)(\der_{12}v)^2\,d\rx\\
&\phantom{=}+\int_{\Gamexg} \zeta_r^2 e^{-\mu x_1} \left(\frac 14 (\der_{12}v)^2-\mu^2(\der_2 v)^2\right)\,dx_2-\int_{\Om}\zeta_r^2 e^{-\mu x_1}\left(\frac{\mu^3}{2}(\der_2 v)^2+\frac{\mu^2}{\lambda}a_1^2(\der_1 v)^2\right)\,d\rx\\
&\phantom{=}-C(\kappa_0,\lambda, \|(a_{11}, a_{12}, a_1)\|_{L^{\infty(\Om)}}, R, r)\|f\|_{L^2(\Om)}^2.
  \end{split}
\end{equation*}

Fix $\mu$ as
\begin{equation*}
  \mu=\frac{\lambda}{4\|a_{11}\|_{L^{\infty}(\Om)}}
\end{equation*}
from which it follows that
\begin{equation*}
\frac{\lambda}{2}-\mu\|a_{11}\|_{L^{\infty}(\Om)} =\frac{\lambda}{4}.
\end{equation*}
If
\begin{equation*}
\begin{split}
&\frac{\mu}{2}-\frac{4\|a_{12}\|_{W^{1,\infty}(\Om)}^2(1+\mu)^2}{\lambda}\ge \frac{\mu}{4}\\
&\left(\Leftrightarrow \|a_{12}\|_{W^{1,\infty}(\Om)}\le \frac{\lambda}{8\sqrt{\|a_{11}\|_{L^{\infty}(\Om)}}(1+\frac{\lambda}{4\|a_{11}\|_{L^{\infty}(\Om)}})}\right),
\end{split}
\end{equation*}
then
\begin{equation*}
  \begin{split}
  &\int_{\Om} \zeta_r^2 e^{-\mu x_1}\left(\eps (\der_{111}v)^2+\frac{\lambda}{4}(\der_{11}v)^2+\frac{\mu}{4}(\der_{12}v)^2\right)
  +\frac 14\int_{\Gamexg}  e^{-\mu x_1} (\der_{12}v)^2\,dx_2\\
  &\le C(\kappa_0,\lambda, \|(a_{11}, a_{12}, a_1)\|_{L^{\infty(\Om)}}, R, r)\|f\|_{H^1(\Om)}^2,
  \end{split}
\end{equation*}
from which we obtain
\begin{equation*}
\begin{split}
  &\sqrt{\eps} \|\der_{111} v^{\eps}\|_{L^2(\Om\cap\{x_1>2r\})}
  %+\|\der_{12}v^{\eps}\|_{L^2(\Gamexg)}
  +\|D\der_1v^{\eps}\|_{L^2(\Om_L\cap\{x_1>2r\})}\\
& \le C(\kappa_0,\lambda, \|(a_{11}, a_{12}, a_1)\|_{L^{\infty(\Om)}}, R, r)\|f\|_{H^1(\Om)}.
\end{split}
\end{equation*}

By solving $\mcl{L}_{\eps}v=f$ for $\der_{22}v$, we get
\begin{equation*}
\begin{split}
  &\|\der_{22}v\|_{L^2(\Om\cap\{x_1>2r\})}\\
  &\le
  \|\eps\der_{111}v\|_{L^2(\Om\cap\{x_1>2r\})}+\|\der_1 v\|_{H^1(\Om\cap\{x_1>2r\})}+\|f\|_{L^2(\Om)}.
  \end{split}
\end{equation*}
Finally, \eqref{G:viscous-estimate-intermediate2} is obtained from combining these two estimates.

\end{proof}

\subsection{Extension of boundary value problem \eqref{bvp-aux}} \label{subsec:2.2}
%Fix $m\ge 4$ in $\mathbb{N}$.
By a standard method, one can define extensions $(\til a_{11}, \til a_{12}, \til a_1)$  of $(a_{11}, a_{12}, a_1)$ onto
\begin{equation*}
  \Om_{2R}:=\{\rx=(x_1, x_2): 0<x_1<2R,\,\,|x_2|<1\}
\end{equation*}
such that
\begin{itemize}
\item[-]
$\displaystyle{  \til a_{11}=a_{11},\quad \til a_{12}=a_{12},\quad \til a_1=a_1\quad\tx{in $\ol{\Om}$}
}$;
\item[-] $\til a_{12}=0$ on $\der\Om_{2R}\cap\{|x_2|=1\}$;
\item[-] $\displaystyle{\til a_{11}, \til a_{12}, \til a_1\in H^{m-1}(\Om_{2R})}$ with
    \begin{equation*}
    \begin{split}
      &\|\til a_{11}\|_{H^{m-1}(\Om_{2R})}\le C(R)\|a_{11}\|_{H^{m-1}(\Om)},\\
      &\|\til a_{12}\|_{H^{m-1}(\Om_{2R})}\le C(R)\|a_{12}\|_{H^{m-1}(\Om)},\\
      &\|\til a_{1}\|_{H^{m-1}(\Om_{2R})}\le C(R)\|a_{1}\|_{H^{m-1}(\Om)};
      \end{split}
    \end{equation*}
\item[-]if $m=4$, there exists $\til{\bar a}_{11}\in C^2(\ol{\Om_{2R}})$ such that
\begin{equation}
\label{G:pert-condition-extension}
  \|(\til a_{11}, \til a_{12})-(\til {\bar{a}}_{11},0)\|_{H^{m-1}(\Om_{2R})}\le C(R)\bar{\delta}
\end{equation}
for $\bar{\delta}>0$ from \eqref{G: perturbation condition for m=4}, where $\til{\bar a}_{11}$ is a standard $C^2$-extension of $\bar{a}_{11}$ onto $\ol{\Om_{2R}}$, or if $m>4$,
\begin{equation}
\label{G:pert-condition2-extension}
  \|\til a_{12}\|_{H^{m-1}(\Om)}\le C(R)\bar{\delta}
\end{equation}
for $\bar{\delta}>0$ from \eqref{G: perturbation condition for m>4}.
\end{itemize}

Since the coefficients $\til a_{11}$, $\til a_{12}$ and $\til a_1$ are $C^1$ in $\ol{\Om}$, one can fix a constant $R_*\in(R, 2R)$ such that
  \begin{equation}\label{KZ-condition-ext}
  \til a_1-\der_2 \til a_{12}+\frac{(2k-3)}{2}\der_1\til a_{11}\le -\frac{\lambda}{2} \quad\tx{in $\Om_{R_*}$}
\end{equation}
for all $k=1,2,\cdots, m$. Put
\begin{equation*}
  \delta_*:=\frac{R_*-R}{4}.
\end{equation*}
Fix a smooth function $\chi^{(1)}$ on $\R$ such that
\begin{equation*}
  \chi^{(1)}(x_1)=
  \begin{cases}
  1\quad&\mbox{for $x_1\le R+2\delta_*$}\\
  0\quad&\mbox{for $x_1\ge R+3\delta_*$}
  \end{cases},\quad (\chi^{(1)})'\le 0\,\,\tx{on $\R$}.
\end{equation*}
Fix a smooth function $\chi^{(2)}$ on $\R$ such that
\begin{equation*}
  \chi^{(2)}(x_1)=
  \begin{cases}
  1\quad&\mbox{for $x_1\le R+\delta_*$}\\
  0\quad&\mbox{for $x_1\ge R+2\delta_*$}
  \end{cases},\quad (\chi^{(2)})'\le 0\,\,\tx{on $\R$}.
\end{equation*}
For two constants $\om$ and $\mu$ to be determined, define %modified coefficients
\begin{equation}
\label{definition:KZ ext coefficients}
\begin{split}
  &a^*_{11}(x_1, x_2):=\til a_{11}(x_1, x_2)\chi^{(1)}(x_1)+\om(1-\chi^{(1)}(x_1)),\\
&a^*_1(x_1, x_2):=\til a_1(x_1, x_2)\chi^{(1)}(x_1)+\mu(1-\chi^{(2)}(x_1)).
  \end{split}
\end{equation}
First, we fix $\om>0$ sufficiently large depending on $\|(a_{11}, a_{12})\|_{H^{m-1}(\Om)}$ such that
\begin{equation}
\label{G:ellipticity at R end}
  \begin{bmatrix}
  a_{11}^*&\til a_{12}\\
  \til a_{12}&1
  \end{bmatrix}\ge \frac{\om}{2}\mathbb{I}_{2\times 2}\quad\tx{in $\Om_{R_*}\cap\{x_1\ge R_*-2\delta_*\}$}.
\end{equation}
Note that
\begin{equation*}
\{x_1:(\chi^{(1)})'(x_1)\neq 0\}\subset \{\chi^{(2)}(x_1)=0\}.
\end{equation*}

By \eqref{KZ-condition-ext}, it holds that
\begin{equation*}
\begin{split}
&a_1^*-\der_2 \til a_{12}+\frac{2k-3}{2}\der_1a_{11}^*\\
&=\left(\til a_1-\der_2\til a_{12}+\frac{2k-3}{2}\der_1\til a_{11}\right)\chi^{(1)}+\frac{2k-3}{2}(\til a_{11}-\om)(\chi^{(1)})'+\mu(1-\chi^{(2)})\\
&\le -\frac{\lambda}{2}\chi^{(1)}+\left(\frac{2k-3}{2}(\til a_{11}-\om)(\chi^{(1)})'+\mu\right)(1-\chi^{(2)})
\end{split}
\end{equation*}
One can choose $\mu<0$ with $|\mu|$ sufficiently large depending on $(m, \|a_{11}\|_{H^{m-1}(\Om)}, R, \delta_*, \om)$ to satisfy
\begin{equation*}
\frac{2k-3}{2}(\til a_{11}-\om)(\chi^{(1)})'+\mu \le -\frac{\lambda}{2}\quad\tx{for $R+\delta_*\le x_1\le R_*$}
\end{equation*}
for $k=1,2,\cdots, m.$

For the extended coefficients, define %an extension of Kuzmin coefficient of order $k$(see Definition \ref{definition:Kz coefficient}) by
\begin{equation*}
  \kz^*_l=a_1^*+\frac{(2l-3)}{2}\der_1 a_{11}^*-\der_2 \til a_{12}\quad\tx{in $\Om_{R_*}$}
\end{equation*}
for $l=1,2,\cdots, m$ (cf. Definition \ref{definition:Kz coefficient}).
Then, we have
\begin{equation}
\label{G:KZ-condition-ext}
a_1^*-\der_2 \til a_{12}+\frac{2l-3}{2}\der_1a_{11}^*  \kz^*_l-\frac{\lambda}{2}\quad\tx{in $\Om_{R_*}$}
\end{equation}
for all $l=1,2,\cdots, m$

%\begin{lemma}
Note that the new coefficients $a_{11}^*$ and $a_1^*$ given by \eqref{definition:KZ ext coefficients} satisfy
\begin{equation*}
  (a_{11}^*, a_1^*)=(a_{11}, a_1)\quad\tx{in $\ol{\Om}$},
\end{equation*}
and
\begin{equation*}
\begin{split}
\|a_1^*\|_{H^{m-1}(\Om_{R_*})}&\le C(m,R,\lambda,\|(a_{11}, a_{12})\|_{H^{m-1}(\Om)})(\|a_1\|_{H^{m-1}}+1),\\
\|a_{11}^*\|_{H^{m-1}(\Om_{R_*})}&\le C(R,\lambda,\|(a_{11}, a_{12})\|_{H^{m-1}(\Om)})(\|a_{11}\|_{H^{m-1}(\Om)}+1).
\end{split}
\end{equation*}
In addition, if $m=4$, then there exists a function $\bar{a}_{11}^*\in C^2(\ol{\Om_{R_*}})$ such that
\begin{equation}
\label{new perturbation condition for m=4}
  \|a_{11}^*-\bar{a}_{11}^*\|_{H^{m-1}(\Om_{R_*})}\le C(
  R)\bar{\delta}
\end{equation}
for the constant $\bar{\delta}>0$ from \eqref{G: perturbation condition for m=4}.
%\end{lemma}
As an extension of the linear differential operator $\mcl{L}$ given in \eqref{definition:L-general}, we define
\begin{equation*}
  \mcl{L}^*V:=a_{11}^*\der_{11}V+2\til a_{12}\der_{12}V
  +\der_{22}V+a_1^*\der_1V\quad\tx{in $\Om_{R_*}$}.
\end{equation*}
Given $\til f:\ol{\Om_{R_*}}\rightarrow \R$, consider a boundary value problem extended from \eqref{bvp-g}:
\begin{equation}\label{bvp-aux-ext}
  \begin{cases}
 \mcl{L}^*V=\til f\quad&\mbox{in $\Om_{R_*}$}\\
 V=0 \quad&\mbox{on $\Gamen$}\\
  \der_2V=0\quad&\mbox{on $\Gamw^*:=\{\rx=(x_1, x_2):0<x_1<R_*,\,\,|x_2|=1\}$}.
  \end{cases}
\end{equation}

\begin{lemma}\label{lemma:extension-consistence}
Given a function $\til f\in H^1(\Om_{R_*})$, suppose that
\begin{equation*}
  \til f=f\quad\tx{in $\Om$}.
\end{equation*}
If $V\in H^2(\Om_{R_*})$ is a strong solution to \eqref{bvp-aux-ext}, and if $v\in H^2_{\rm loc}(\Om)$ is a strong solution to \eqref{bvp-g} , then we have
\begin{equation*}
  V=v\quad\tx{in $\Om$}.
\end{equation*}

\begin{proof}
Fix a small constant $t\in(0,R)$, and set $z:=v-V$ in $\Om\cap\{x_1<R-t\}$. Then, $z\in H^2(\Om\cap\{x_1<R-t\})$ satisfies
\begin{equation*}
  \begin{split}
  \mcl{L}^* z=0\quad&\tx{in $\Om\cap\{x_1<R-t\}$},\\
  z=0\quad&\tx{on $\Gamen$},\\
  \der_2z=0\quad&\tx{on $
  \Gamw\cap \{x_1<R-t\}$}.
  \end{split}
\end{equation*}
So we can integrate by parts to get
\begin{equation}
\label{estimate of xi}
\begin{split}
0&=
\int_{\Om\cap\{x_1<R-t\}}
\mcl{L}^* z\der_1 z\,d\rx=T_1+T_2
\end{split}
\end{equation}
for
\begin{equation*}
  \begin{split}
T_1:=&\int_{\Om\cap\{x_1<R-t\}}\left(-\der_1a_{11}^*+2a_1^*-2\der_2\til a_{12}\right)
\frac{(\der_1 z)^2}{2}\,d\rx-\int_{\{(R-t, x_2):|x_2|<1\}}\frac{(\der_2 z)^2}{2}\,dx_2,\\
T_2:=&\left(\int_{\{(R-t, x_2):|x_2|<1\}}-\int_{\Gamen}\right)a_{11}^*\frac{(\der_1 z)^2}{2}\,dx_2.
\end{split}
\end{equation*}
Since $T_1\le 0$ and $T_2\le 0$, we obtain that
\begin{equation*}
\der_1 z=0\quad\tx{in $\Om$}.
\end{equation*}
Since $z=0$ on $\Gamen$, it follows that $z\equiv 0$ in $\Om$.
\end{proof}
\end{lemma}

By using all the results from \S\ref{subsec:2.1}, we shall prove in \S\ref{subsec:2.3} that \eqref{bvp-g} has a weak solution $v\in H^1(\Om)$ to \eqref{bvp-g}, and that $v$ is in $H^2(\Om\cap\{r<x_1<R-r\})$ for $0<r<R$. In order to establish a global $H^m$-estimate of $v$ in $\Om$, we use Lemma \ref{lemma:extension-consistence} and all the other estimate results following hereafter in \S\ref{subsec:2.2}.

\medskip

A singular perturbation boundary value problem of \eqref{bvp-aux-ext} is as follows:
\begin{equation}\label{bvp-aux-ext-singp}
  \begin{cases}
  \eps\der_{111}V^{\eps}+\mcl{L}^*V^{\eps}=\til f\quad&\mbox{in $\Om_{R_*}$}\\
  V^{\eps}=0,\quad \der_1V^{\eps}=0 \quad&\mbox{on $\Gamen$}\\
  \der_2V^{\eps}=0\quad&\mbox{on $\Gamw^*$}\\
  \der_{11}V^{\eps}=0\quad&\mbox{on $\Gam_{R_*}$}.
  \end{cases}
\end{equation}

\begin{lemma}
\label{lemma G:wp of singular pert prob-ext}
There exists a constant $\bar{\eps}>0$ depending only on $\lambda$ so that, for any $\eps\in(0, \bar{\eps}]$, if $V^{\eps}$ is a smooth solution to \eqref{bvp-aux-ext-singp}, then it satisfies
\begin{equation}
\label{G:estiamte 1-sing pert-ext}
\begin{split}
  &\sqrt{\eps}\|\der_{11}V^{\eps}\|_{L^2(\Om_{R_*})}
  +\|V^{\eps}\|_{H^1(\Om_{R_*})}\le C(\lambda, R) \|\til f\|_{L^2(\Om_{R_*})}.
  \end{split}
\end{equation}

\end{lemma}
The proof of this lemma was skipped as it is almost the same as the one of Lemma \ref{lemma G:wp of singular pert prob-main}.
\medskip

Put
\begin{equation*}
  r_*:=\frac 18\min\{R_1, \delta_*\}
\end{equation*}
For a constant $r$ with $0<r \le r_*$, define
\begin{equation*}
  \Om_{R_{*},r}:=\Om_{R_*}\cap \left(\{2r<x_1<R_1-2r\}\cup \{R+3\delta_*-2r<x_1<R_*-2r\}\right).
\end{equation*}
Since the linear operator $\mcl{L^*}$ is uniformly elliptic in $ \Om_{R_{*},r}$, a simple modification of the proof of Lemma \ref{G:lemma for pre H2 estimate of vm, part 1} yields the following lemma.
\begin{lemma}
\label{lemma:2.5} Let $\bar{\eps}$ be from Lemma \ref{lemma G:wp of singular pert prob-ext}.
For any $\eps\in(0, \bar{\eps}]$ and $r\in(0, r_*]$, if $V^{\eps}$ is a smooth solution to \eqref{bvp-aux-ext-singp}, then we have
\begin{equation}
\label{G:viscous-estimate-intermediate1-ext}
  \|D^2 V^{\eps}\|_{L^2( \Om_{R_{*},r})}
  \le C\|\til f\|_{L^2(\Om_{R_*})}
\end{equation}
for $C=C(\kappa_1, \om, \lambda, \|(a_{11}^*, \til a_{12}, a_1^*)\|_{L^{\infty}(\Om_{R_*})},R_*, r)$.

\end{lemma}
Next, a minor modification of the proof of Lemma \ref{G:lemma for pre H2 estimate of vm, part 2} with applying Lemma \ref{lemma:2.5} yields the following result.
\begin{lemma}%[$L^2$-estimate of $\der_{11}v_m$ up to $\Gamex$]
\label{G:lemma for pre H2 estimate-ext} Under the same assumptions as Lemma \ref{lemma:2.5}, $V^{\eps}$ satisfies
\begin{equation}
\label{G:viscous-estimate-intermediate2-2}
\begin{split}
  &\sqrt{\eps} \|\der_{111} V^{\eps}\|_{L^2(\Om_{R_*}\cap\{2r<x_1<R_*-2r\})}
  +\|D^2V^{\eps}\|_{L^2(\Om_{R_*}\cap\{2r<x_1<R_*-2r\})}
  \le C\|\til f\|_{H^1(\Om_{R_*})}
\end{split}
\end{equation}
for $C=C(\kappa_0, \lambda, \|(a_{11}^*, \til a_{12}, a_1^*)\|_{L^{\infty}(\Om_{R_*})}, R_*, r)$ provided that
\begin{equation*}
  \|\til a_{12}\|_{W^{1,\infty}(\Om_{R_*})}\le \til{\delta}
\end{equation*}
for a constant $\til{\delta}>0$ fixed sufficiently small depending only on $(\lambda, \|a_{11}^*\|_{L^{\infty}(\Om_{R_*})})$.

\end{lemma}

\begin{definition}
\label{definition:cut-offs}
Fix $m\ge 4$. And, fix $r>0$sufficiently small to satisfy
\begin{equation}
\label{G:condition for r}
 0< r\le \frac{\min\{R_1, \delta_*\}}{3(m-1)+2}.
\end{equation}

\textbf{1. Smooth double-bumped cut-off function:}
For $k=1,\cdots,m-1$, define a smooth cut-off function $\eta_r^{(k)}(x_1)$ such that
\begin{itemize}
\item[(i)]
\begin{equation*}
  \eta_r^{(k)}(x_1):=\begin{cases}
  1\quad&\mbox{if $x_1\in[\frac{3k-1}{2}r, \frac{3k+1}{2}r]\cup [R_*-\frac{3k+1}{2}r, R_*-\frac{3k-1}{2}r]$}\\
  0\quad&\mbox{if $x_1\in [0,R_*]\setminus \bigl((\frac{3k-2}{2}r, \frac{3k+2}{2}r)\cup (R_*-\frac{3k+2}{2}r, R_*-\frac{3k-2}{2}r)\bigr)$}
  \end{cases}
\end{equation*}
\item[(ii)] $0\le \eta_r^{(k)} \le 1 $ on $\R$
\item[(iii)]
\begin{equation*}
 \left |\left(\frac{d}{dx_1}\right)^j\eta_r^{(k)}\right|\le C_j\left(\frac{1}{r}\right)^j\quad\tx{on $\R$}.
\end{equation*}
\end{itemize}

\textbf{2. Smooth single-bumped cut-off function subordinate to $\eta_r^{(k)}$:}
For each $r>0$ and $k=1,\cdots, m-1$, we define a smooth cut-off function  $\zeta_r^{(k)}(x_1)$ such that
\begin{itemize}
\item[(i)] \begin{equation*}
  \zeta_r^{(k)}(x_1):=\begin{cases}
  1\quad&\mbox{if $x_1\in [\frac{3k+1}{2}r, R_*-\frac{3k+1}{2}r]$}\\
  0\quad&\mbox{if $x_1\in [0, R_*]\setminus (\frac{3k}{2}r, R_*-\frac{3k}{2}r)$}
  \end{cases}
\end{equation*}

\item[(ii)] $0\le \zeta_r^{(k)} \le 1 $ on $\R$
\item[(iii)]
\begin{equation*}
  \left|\left(\frac{d}{dx_1}\right)^j\zeta_r^{(k)}\right|\le C_j\left(\frac{1}{r}\right)^j\quad\tx{on $\R$}.
\end{equation*}
\end{itemize}
We say that $\zeta_r^{(k)}$ is subordinate to $\eta_r^{(k)}$ in the sense that
\begin{equation*}
  {\rm spt}\left\{\left(\frac{d}{dx_1}\right)^j\zeta_r^{(k)}\right\}\subset \{\eta_{r}^{(k)}=1\}\quad\forall j\in \mathbb{N}.
\end{equation*}
Note that, for each $k=3,\cdots, m-1$
\begin{align}
\label{G-cutoff prop 1}
  &{\rm spt}\,\eta_r^{(k)}\subset \{\zeta_r^{(k-1)}=1\},\\
  \label{G-cutoff prop 2}
  &{\rm spt} \frac{d}{dx_1}\zeta_r^{(k)} \subset \{\eta_r^{(k)}=1\}.
\end{align}

\end{definition}

\begin{proposition}
\label{proposition:2.10}
Under the same assumptions as Lemma \ref{lemma:2.5}, for any  $\eps\in(0,\bar{\eps}]$ and $r\in(0, r_*]$, we have
\begin{equation}
\label{G:final H-k estimate}
\begin{split}
  &\sqrt{\eps} \|\zeta_r^{(k)}\der_1^{k+1}V\|_{L^2(\Om_{R_*})}+\|\zeta_r^{(k)} \der_1^{k-2}D^2V\|_{L^2(\Om_{R_*})}\\
&  \le C^*
\|\til f\|_{H^{k-1}(\Om_{R_*})}
  \end{split}
\end{equation}
for all $k=2,\cdots, m-2$ with
\begin{equation*}
  C^*=\begin{cases}
  C^*(r,R, m,  \lambda, \kappa_1, \om, \|(a_{11}^*,\til a_{12}, a_1^*)\|_{H^{m-1}(\Om_{R_*})},  \|\der_{11}\bar{a}_{11}^*\|_{L^{\infty}(\Om_{R_*})})&\mbox{if $m=4$}\\
  C^*(r, R, m, \lambda, \kappa_1, \om,  \|(a_{11}^*,\til a_{12}, a_1^*)\|_{H^{m-1}(\Om_{R_*})})&\mbox{if $m>4$}
  \end{cases}
\end{equation*}
provided that the constant $\bar{\delta}>0$ from \eqref{G: perturbation condition for m=4} if $m=4$, or from \eqref{G: perturbation condition for m>4} if $m>4$ is fixed sufficiently small. In addition, such a constant $\bar{\delta}$ can be fixed with the same dependence as the estimate constant $C^*$.

\begin{proof}
We prove this proposition by induction.

{\textbf{Step 1.}}
By Lemma \ref{G:lemma for pre H2 estimate-ext}, \eqref{G:final H-k estimate} holds for $k=2$.
Next, fix $k\ge 3$, and suppose that \eqref{G:final H-k estimate} holds for all $2\le j\le k-1$.
\medskip

{\textbf{Step 2.}} Fix a constant $r\in(0, r_*]$. To simplify notations,  put
\begin{equation*}
V:=V^{\eps},\quad  \zeta:=\zeta_r^{(k)},\quad \eta:=\eta_r^{(k)}.
\end{equation*}
To obtain \eqref{G:final H-k estimate}, we take
\begin{equation}
\label{Hk-5}
  \int_{\Om_{R_*}} \der_1^{k-2}(\eps\der_{111}V+\mcl{L}^*V)\zeta^2e^{-\mu x_1}\der_1^k(\der_1 V)\,d\rx=\int_{\Om_{R_*}}\der_1^{k-2}\til f\zeta^2 e^{-\mu x_1}\der_1^k(\der_1 V)\,d\rx
\end{equation}
for a constant $\mu>0$ to be determined later.

First, we estimate the right-hand side of \eqref{Hk-5}. By an integration by parts, we get
\begin{equation*}
 \begin{split}
 &\int_{\Om_{R_*}}\der_1^{k-2}\til f\zeta^2e^{-\mu x_1}\der_1^k(\der_1 V)\,d\rx
 %&=\int_{\Om_{R_*}}\der_1(\der_1^{k-2}\til f\zeta^2 \der_1^kV)-\der_1(\der_1^{k-2}\til f\zeta^2)\der_1^kV\,d\rx\\
 =-\int_{\Om_{R_*}} \der_1(e^{-\mu x_1}\der_1^{k-2}\til f) \zeta^2 \der_1^kV+(e^{-\mu x_1}\der_1^{k-2}\til f)(\der_1 \zeta^2)\der_1^kV\,d\rx.
  \end{split}
\end{equation*}
In order to estimate $\int_{\Om_{R_*}}(e^{-\mu x_1}\der_1^{k-2}\til f)(\der_1 \zeta^2)\der_1^kV\,d\rx$, we take
\begin{equation}
\label{Hk-1}
  \int_{\Om_{R_*}} \der_1^{k-2}(\eps\der_{111}V+\mcl{L}^*V)\eta^2\der_1^k V\,d\rx=\int_{\Om_{R_*}}\der_1^{k-2}\til f\eta^2\der_1^kV\,d\rx.
\end{equation}
%\begin{equation*}
%  |\int_{\Om_{R_*}}\der_1^{k-2}\til f\eta^2\der_1^kV\,d\rx|\le \|\der_1^{k-2}\til f\|_{L^2(\Om_{R_*})} \|\eta \der_1^k V\|_{L^2(\Om_{R_*})}
%\end{equation*}
By an integration by parts,
\begin{equation*}
\begin{split}
 \int_{\Om_{R_*}} \der_1^{k-2}(\eps\der_{111}V)\eta^2\der_1^k V\,d\rx=
  %&=\int_{\Om_{R_*}} \der_1\left(\eps\eta^2\frac{(\der_1^kV)^2}{2}\right)-\eps \eta\eta'(\der_1^kV)^2\,d\rx\\&=
  -\int_{\Om_{R_*}}\eps \eta\eta'(\der_1^kV)^2\,d\rx.
  \end{split}
\end{equation*}
The inductive assumption of \eqref{G:final H-k estimate} holding for $2\le j\le k-1$ combined with \eqref{G-cutoff prop 1} yields
\begin{equation*}
|\int_{\Om_{R_*}}\eps \eta\eta'(\der_1^kV)^2\,d\rx|\le C(r,k)\|\til f\|^2_{H^{k-2}(\Om_{R_*})}.
\end{equation*}

Next, we consider the term $\displaystyle{\int_{\Om_{R_*}} \der_1^{k-2}(\mcl{L}^*V)\eta^2\der_1^k V\,d\rx}$.
By direct computations, we have
\begin{equation*}
  \begin{split}
  &\int_{\Om_{R_*}}\der_1^{k-2}(\mcl{L}^*V)\eta^2 \der_1^kV\,d\rx\\
  &=\int_{\Om_{R_*}}(a_{11}^*(\der_1^kV)^2+2\til a_{12}\der_1^kV\der_1^{k-1}\der_2 V+(\der_1^{k-1}\der_2V)^2)\eta^2\,d\rx\\
  &\phantom{=}+\int_{\Om_{R_*}} a_1^*\eta^2\der_1^{k-1}V\der_1^kV+\der_1\eta^2(\der_1^{k-2}\der_2V)(\der_1^{k-1}\der_2V)\,d\rx\\
  &\phantom{=}+\int_{\Om_{R_*}} \eta^2\der_1^kV\sum_{l=1}^{k-2}\left(\der_1^l a_{11}^* \der_1^{k-l}V+2\der_1^l\til a_{12}\der_1^{k-1-l}\der_2V+\der_1^la_1^*\der_1^{k-1-l}V\right)\,d\rx\\
  &=:T_1+T_2+T_3
  \end{split}
\end{equation*}
By \eqref{G:ellipticity} and \eqref{G:ellipticity at R end}, we have
\begin{equation}
\label{Hk-2}
T_1
\ge \min\{\kappa_1, \frac{\om}{2}\}\|\eta D\der_1^{k-1}V\|^2_{L^2(\Om_{R_*})}.
\end{equation}

By applying Morrey's inequality, \eqref{G-cutoff prop 1} and the inductive assumption of \eqref{G:final H-k estimate} holding for $2\le j\le k-1$, we can directly check
\begin{equation}
\label{Hk-3}
  \begin{split}
  |T_2|
 % &\int_{\Om_{R_*}} |a_1^*\eta^2\der_1^{k-1}V\der_1^kV+\der_1\eta^2(\der_1^{k-2}\der_2V)(\der_1^{k-1}\der_2V)|\,d\rx\\
  &\le (\|a_1^*\|_{L^{\infty}(\Om_{R_*})}+1)\|\eta D\der_1^{k-1}V\|_{L^2(\Om_{R_*})}(\|\eta\der_1^{k-1}V\|_{L^2(\Om_{R_*})}+2\|\der_1\eta \der_1^{k-2}\der_2V\|_{L^2(\Om_{R_*})})\\
  &\le \til C \|\eta D\der_1^{k-1}V\|_{L^2(\Om_{R_*})} \|\til f\|_{H^{k-2}(\Om_{R_*})}
  \end{split}
\end{equation}
for a constant $ \til{C}$ having the same dependence as $C^*$ from \eqref{G:final H-k estimate}.

In the following, any estimate constant presented as $\til C$ may vary but is assumed to be fixed with the same dependence as $C^*$ from \eqref{G:final H-k estimate} unless otherwise specified.

Next, we estimate the term $T_3$.
%In the term $\displaystyle{\int_{\Om_{R_*}}\eta^2\sum_{l=1}^{k-2}\der_1^la_{11}^*\der_1^{k-l}V\der_1^kV\,d\rx}$, the case of $l=1$ is treated first.
Since $a_{11}^*\in H^{m-1}(\Om_{R_*})$ for $m\ge 4$, we apply Morrey's inequality to estimate $\|\der_1a_{11}^*\|_{L^{\infty}(\Om_{R_*})}$ by $\|a_{11}^*\|_{H^{m-1}(\Om_{R_*})}$. Then, we use \eqref{G-cutoff prop 1} and the inductive assumption to get
\begin{equation}
\label{2:36}
 \|\der_1a_{11}^* \der_1^{k-1}V\eta\|_{L^2(\Om_{R_*})}\le \til{C}\|\til f\|_{H^{k-2}(\Om_{R_*})}
\end{equation}
%for a constant $ \til{C}$ having the same dependence as $C^*$ from \eqref{G:final H-k estimate}.
By applying Sobolev inequality, Poincar\'{e} inequality and \eqref{2:36}, we get
\begin{equation*}
\begin{split}
  &\int_{\Om_{R_*}}|\eta^2\sum_{l=1}^{k-2}\der_1^la_{11}^*\der_1^{k-l}V\der_1^kV|\,d\rx\\
  &\le C(R) \|\eta\der_1^kV\|_{L^2(\Om_{R_*})}\left(\|\til f\|_{H^{k-2}(\Om_{R_*})}+\sum_{l=2}^{k-2}\|a_{11}^*\|_{H^{l+1}(\Om_{R_*})} \|D(\eta\der_1^{k-l}V)\|_{L^2(\Om_{R_*})}\right)\\
  &\le C(R,m)\|a_{11}^*\|_{H^{m-1}(\Om_{R_*})} \|\eta\der_1^kV\|_{L^2(\Om_{R_*})} \left(\|\til f\|_{H^{k-2}(\Om_{R_*})}+
  \sum_{l=2}^{k-2} \|D(\eta\der_1^{k-l}V)\|_{L^2(\Om_{R_*})}\right).
  \end{split}
\end{equation*}
Next, we use \eqref{G-cutoff prop 1} and the inductive assumption to further estimate as
\begin{equation}
 \int_{\Om_{R_*}}|\eta^2\sum_{l=1}^{k-2}\der_1^la_{11}^*\der_1^{k-l}V\der_1^kV|\,d\rx
 \le \til C \|\eta\der_1^kV\|_{L^2(\Om_{R_*})}\|\til f\|_{H^{k-2}(\Om_{R_*})}
\end{equation}
%for a constant $ \til{C}$ having the same dependence as $C^*$ from \eqref{G:final H-k estimate}.
By a similar argument, one can estimate the rest of the term $T_3$ so we get
\begin{equation}
\label{Hk-4}
  \begin{split}
 |T_3|\le \til C \|\eta\der_1^kV\|_{L^2(\Om_{R_*})}\|\til f\|_{H^{k-2}(\Om_{R_*})}.
   \end{split}
\end{equation}

We combine all of \eqref{Hk-1}--\eqref{Hk-3} and \eqref{Hk-4} to obtain
\begin{equation*}
 \begin{split}
 &\int_{\Om_{R_*}} \der_1^{k-2}(\eps\der_{111}V+\mcl{L}^*V)\eta^2\der_1^k V\,d\rx\\
 &\ge \min\{\kappa_0, \frac{\om}{2}\}\|\eta D\der_1^{k-1}V\|^2_{L^2(\Om_{R_*})}-\til C \|\eta\der_1^kV\|_{L^2(\Om_{R_*})}\|\til f\|_{H^{k-2}(\Om_{R_*})}.
 \end{split}
\end{equation*}
%for a constant $ \til{C}$ having the same dependence as $C^*$ from \eqref{G:final H-k
So we conclude from \eqref{Hk-1} that
\begin{equation}
\label{Hk-6}
  \|\eta\der_1^kV\|_{L^2(\Om_{R_*})}\le \til C \|\til f\|_{H^{k-2}(\Om_{R_*})}.
\end{equation}
%for a constant $ \til{C}$ having the same dependence as $C^*$ from \eqref{G:final H-k estimate}.
Then it follows from \eqref{G-cutoff prop 2} that
\begin{equation*}
\begin{split}
  \int_{\Om_{R_*}}|(e^{-\mu x_1}\der_1^{k-2}\til f)(\der_1 \zeta^2)\der_1^kV|\,d\rx
  &\le 2\|\til f\|_{H^{k-2}(\Om_{R_*})} \|\der_1\zeta\der_1^kV\|_{L^2(\Om_{R_*})}\\
  &\le \til C\|\til f\|_{H^{k-2}(\Om_{R_*})}^2.
  \end{split}
\end{equation*}
%for a constant $ \til{C}$ having the same dependence as $C^*$ from \eqref{G:final H-k estimate}.
This estimate directly yields
\begin{equation}
\label{RHS}
|{\tx{RHS of \eqref{Hk-5}}}|\le (1+\mu)\|\til f\|_{H^{k-1}(\Om_{R_*})} \|\zeta\der_1^kV\|_{L^2(\Om_{R_*})}+\til C\|\til f\|_{H^{k-2}(\Om_{R_*})}^2.
\end{equation}

%Starting from the next step, we estimate the left-hand side of \eqref{Hk-5}.

Now, it remains to estimate the left-hand side of \eqref{Hk-5}.
\medskip

{\textbf{Step 3.}} By lengthy but straightforward computations, we have
\begin{equation*}
  \begin{split}
  {\tx{LHS of \eqref{Hk-5}}}&=P_k(\mu)+\sum_{j=1}^4\mcl{R}_k^{(j)}(\mu)
  \end{split}
\end{equation*}
where $P_k(\mu)$ and $\mcl{R}_k^{(j)}$ are given as follows:
\begin{equation*}
  \begin{split}
  P_k(\mu)&:=\int_{\Om_{R_*}} \eps(\der_1^{k+1}V)^2\zeta^2 e^{-\mu x_1}+ \left(-\kz^*(k)+\mu a_{11}^*\right)e^{-\mu x_1}(\der_1^kV)^2\,d\rx\\
  &\phantom{=}+\int_{\Om_{R_*}}((2\mu \til a_{12}-2(k-1)\der_1\til a_{12})\der_1^{k-1}\der_2V\der_1^kV+\frac{3\mu}{2}(\der_1^{k-1}\der_2V)^2)\zeta^2 e^{-\mu x_1}\,d\rx;
  \end{split}
\end{equation*}

\begin{equation*}
  \begin{split}
  \mcl{R}_k^{(1)}(\mu):=-&\int_{\Om_{R_*}} (\der_1\zeta^2)(a_{11}^*{{e^{-\mu x_1}}})(\der_1^kV)^2 +(k-2)\der_1\left(\der_1a_{11}^*\zeta^2{{e^{-\mu x_1}}}\right)\der_1^{k-1}V\der_1^kV\,d\rx\\
  &+\int_{\Om_{R_*}} T_k
\zeta^2{{e^{-\mu x_1}}}\der_1^k(\der_1V)\,d\rx
  \end{split}
\end{equation*}
with
\begin{equation*}
 T_k :=\begin{cases}
 0\quad&\mbox{for $k=3$}\\
 \sum_{l=2}^{k-2}\binom{k-2}{l}\der_1^la_{11}^*\der_1^{k-l}V\quad&\mbox{for $3<k\le m$}
 \end{cases};
\end{equation*}

\begin{equation*}
  \begin{split}
 \mcl{R}_k^{(2)}(\mu):=& 2\int_{\Om_{R_*}} S_k\zeta^2{{e^{-\mu x_1}}}\der_1(\der_1^kV)-(k-2)\der_{11}\til a_{12}{{e^{-\mu x_1}}}\zeta^2 \der_1^{k-2}\der_2V\der_1^kV \,d\rx\\
   &-2\int_{\Om_{R_*}}\left(e^{-\mu x_1}\til a_{12}(\der_1\zeta^2 )\der_1^{k-1}\der_2V+(k-2)\der_1\til a_{12}\der_1(\zeta^2 {{e^{-\mu x_1}}})\der_1^{k-2}\der_2V\right)\der_1^kV\,d\rx
  \end{split}
\end{equation*}
with
\begin{equation*}
  S_k:=\begin{cases}
  0\quad&\mbox{for $k=3$}\\
 \sum_{l=2}^{k-2}\binom{k-2}{l}\der_1^l\til a_{12}(\der_1^{k-1-l}\der_2V)\quad&\mbox{for $3<k\le m$}
  \end{cases};
\end{equation*}

\begin{equation*}
  \begin{split}
 \mcl{R}_k^{(3)}(\mu):=&-\int_{\Om_{R_*}}\der_1^2(\zeta^2 e^{-\mu x_1})\der_1^{k-2}\der_2V\der_1^{k-1}\der_2V+e^{-\mu x_1}\der_1\zeta^2(\der_1^{k-1}\der_2V)^2\,d\rx;
  \end{split}
\end{equation*}

\begin{equation*}
  \begin{split}
  \mcl{R}_k^{(4)}(\mu):=-\int_{\Om_{R_*}}\der_1 (e^{-\mu x_1}\zeta^2)\der_1^{k-2}(a_1^*\der_1V)\der_1^kV +
  \zeta^2e^{-\mu x_1}\der_1^kV\sum_{l=1}^{k-1}\binom{k-1}{l}\der_1^la_1^*\der_1^{k-l}V\,d\rx.
  \end{split}
\end{equation*}

\medskip

{\textbf{Step 4.}} Estimate of $P_k(\mu)$:
By \eqref{G:KZ-condition-ext},
\begin{equation*}
  -\kz^*(k)+\mu a_{11}^*\ge \frac{\lambda}{2}-\mu\max_{\ol{\Om_{R_*}}}|a_{11}^*|\quad\tx{in $\Om_{R_*}$}.
\end{equation*}
So if the constant $\mu>0$ satisfies
\begin{equation}
\label{mu-condition 1}
  \mu\le \frac{\lambda}{4\max_{\ol{\Om_{R_*}}}|a_{11}^*|},
\end{equation}
then
$\displaystyle{-\kz^*(k)+\mu a_{11}^*\ge \frac{\lambda}{4}}$ holds in $\Om_{R_*}$.
In addition, if $\mu$ satisfies
\begin{equation}
\label{mu-condition 2}
  \mu\le \frac{\lambda}{16^2\max_{\ol{\Om_{R_*}}}|\til a_{12}|^2},
\end{equation}
then we get
$|2\mu \til a_{12}\der_1^{k-1}\der_2V\der_1^kV|\le
  \frac{\mu}{16}(\der_1^{k-1}\der_2 V)^2+\frac{\lambda}{16}(\der_1^kV)^2$ by Cauchy-Schwarz inequality, and this yields
\begin{equation}
\label{G:long-d}
\begin{split}
  P_k(\mu)\ge & \int_{\Om_{R_*}} \eps(\der_1^{k+1}V)^2\zeta^2 e^{-\mu x_1}\,d\rx\\
  &+\left(\min\left\{\frac{3\lambda}{16}, \frac{7\mu}{16}\right\}-(k-1)\|\der_1\til a_{12}\|_{L^{\infty}(\Om_{R_*})}\right)
  \int_{\Om_{R_*}}|D\der_1^{k-1}V|^2
  \zeta^2 e^{-\mu x_1}\,d\rx.
  \end{split}
\end{equation}
Therefore, one can fix a constant $\bar{\delta}>0$ sufficiently small depending on $(\lambda, m, R, \|a_{11}^*\|_{L^{\infty}(\Om_{R_*})})$, and fix $\mu>0$ to satisfy both \eqref{mu-condition 1} and \eqref{mu-condition 2} so that if
\begin{equation*}
  \|\til a_{12}\|_{H^{m-1}(\Om_{R_*})}\le \bar{\delta},
\end{equation*}
then it follows from \eqref{G:long-d} that
\begin{equation}
\label{coerceivity-ext}
\begin{split}
  P_k(\mu)\ge & \int_{\Om_{R_*}} \eps(\der_1^{k+1}V)^2\zeta^2 e^{-\mu x_1}\,d\rx\\
  &+\frac 12\min\left\{\frac{3\lambda}{16}, \frac{7\mu}{16}\right\}
  \int_{\Om_{R_*}}|D\der_1^{k-1}V|^2
  \zeta^2 e^{-\mu x_1}\,d\rx.
  \end{split}
\end{equation}
\medskip

{\textbf{Step 5.}} Estimate of $\mcl{R}_k^{(1)}(\mu)$: {\textbf{(5-1)}} First, it follows from \eqref{G-cutoff prop 2} and \eqref{Hk-6} that
\begin{equation*}
  \int_{\Om_{R_*}} |(\der_1\zeta^2)(a_{11}^*{{e^{-\mu x_1}}})(\der_1^kV)^2|\,d\rx\le \til C\|\til f\|_{H^{k-2}(\Om_{R_*})} \|\zeta\der_1^k V\|_{L^2(\Om_{R_*})}.
\end{equation*}
%for a constant $ \til{C}$ having the same dependence as $C^*$ from \eqref{G:final H-k estimate}.
Note that
 \begin{equation*}
  \begin{split}
  &\der_1\left(\der_1a_{11}^*\zeta^2{{e^{-\mu x_1}}}\right)\der_1^{k-1}V\der_1^kV
 =\der_{11}a_{11}^*(\der_1^{k-1}V\zeta)\cdot (\der_1^kV\zeta)e^{-\mu x_1}+\der_1a_{11}^*\der_1(\zeta^2e^{-\mu x_1})\der_1^{k-1}V\der_1^kV.
  \end{split}
\end{equation*}
By applying \eqref{G-cutoff prop 2}, \eqref{Hk-6} and the inductive assumptions, we get
\begin{equation}
\label{Hk-7}
  \begin{split}
  &\int_{\Om_{R_*}} |(k-2)\der_1\left(\der_1a_{11}^*\zeta^2{{e^{-\mu x_1}}}\right)\der_1^{k-1}V\der_1^kV|\,d\rx\\
  &\le \til C \|\zeta\der_1^kV\| \left(\|\til f\|_{H^{k-2}(\Om_{R_*})}
  +\|\der_{11}a_{11}^*\zeta\der_1^{k-1}V\|_{L^2(\Om_{R_*})}\right).
  \end{split}
\end{equation}
\iffalse
Hereafter, we regard that any estimate constant appearing as $\til C$ has the same dependence as $C^*$ from \eqref{G:final H-k estimate}, that is,
\begin{equation*}
  \til C=\begin{cases}
  \til C(r,R, m,  \lambda, \kappa_0, \om, \|(a_{11}^*,\til a_{12}, a_1^*)\|_{H^{m-1}(\Om_{R_*})},  \|\der_{11}\bar{a}_{11}^*\|_{L^{\infty}(\Om_{R_*})})&\mbox{if $m=4$},\\
  \til C(r, R, m, \lambda, \kappa_0, \om,  \|(a_{11}^*,\til a_{12}, a_1^*)\|_{H^{m-1}(\Om_{R_*})})&\mbox{if $m>4$}.
  \end{cases}
\end{equation*}
\fi

{\textbf{(5-2)}} Note that $\der_{11}a_{11}^*\in H^{m-3}(\Om_{R_*})$. To estimate $\|\der_{11}a_{11}^*\zeta\der_1^{k-1}V\|_{L^2(\Om_{R_*})}$, we consider two cases separately.

(Case 1) If $m=4$, we apply the condition \eqref{new perturbation condition for m=4}, Poincar\'{e} inequality, \eqref{G-cutoff prop 1} and the inductive assumption to get
\begin{equation*}
  \begin{split}
  \|\der_{11}a_{11}^*\der_1^{k-1}V\zeta\|_{L^2(\Om_{R_*})}
  &\le \|\bar{a}_{11}^*\|_{C^2(\ol{\Om})}\|\zeta\der_1^{k-1}V\|_{L^2(\Om_{R_*})}
  +C(R)\bar{\delta}\|D(\der_1^{k-1}V\zeta)\|_{L^2(\Om_{R_*})}\\
  &\le \til C \|\til f\|_{H^{k-2}(\Om_{R_*})}+C(R)\bar{\delta}\|\zeta D\der_1^{k-1}V\|_{L^2(\Om_{R_*})}.
  \end{split}
\end{equation*}

{{(Case 2)}} If $m\ge 5$, then $\der_{11}a_{11}^*\in L^{\infty}(\Om_{R_*})$ by Morrey's embedding theorem, \eqref{G-cutoff prop 1} and the inductive assumption so we have
\begin{equation*}
\begin{split}
  \|\der_{11}a_{11}^*\zeta\der_1^{k-1}V\|_{L^2(\Om_{R_*})}
  &\le \|\der_{11}a_{11}^*\|_{L^{\infty}(\Om_{R_*})}\|\zeta\der_1^{k-1}V\|_{L^2(\Om_{R_*})}\\
  &\le \til C\|\til f\|_{H^{k-2}(\Om_{R_*})}.
  \end{split}
\end{equation*}
Back to \eqref{Hk-7}, we now have
\begin{equation*}
  \begin{split}
  &\int_{\Om_{R_*}}|(k-2)\der_1\left(\der_1a_{11}^*\zeta^2 {e^{-\mu x_1}}\right)\der_1^{k-1}V\der_1^kV|\,d\rx\\
&\le \til C\|\zeta\der_1^kV\|_{L^2(\Om_{R_*})}\left(\|\til f\|_{H^{k-2}(\Om_{R_*})}+\bar{\delta}\|\zeta D\der_1^{k-1}V\|_{L^2(\Om_{R_*})}\right).
  \end{split}
\end{equation*}

{\textbf{(5-3)}} Fix $k\ge 4$. Note that $T_k=0$ for $k<4$. By integration by parts,
\begin{equation*}
  \begin{split}
  \int_{\Om_{R_*}} T_k\zeta^2 e^{-\mu x_1}\der_1(\der_1^k V)\,d\rx
  &=-\int_{\Om_{R_*}} \der_1(T_k\zeta^2 e^{-\mu x_1})\der_1^k V\,d\rx\\
  &=A_1+A_2+A_3
  \end{split}
\end{equation*}
for
\begin{equation*}
  \begin{split}
  A_1&:=-\int_{\Om_{R_*}}\sum_{l=2}^{k-2}\binom{k-2}{l}(\der_1^{l+1}a_{11}^*)(\der_1^{k-l}V\zeta)e^{-\mu x_1} \zeta\der_1^kV\,d\rx\\
  A_2&:=-\int_{\Om_{R_*}}\sum_{l=2}^{k-2}\binom{k-2}{l}(\der_1^{l}a_{11}^*)(\der_1^{k-l+1}V\zeta)e^{-\mu x_1} \zeta\der_1^kV\,d\rx\\
  A_3&:=-\int_{\Om_{R_*}} \sum_{l=2}^{k-2}\binom{k-2}{l}(\der_1^{l}a_{11}^*)(\der_1^{k-l}V)\der_1(e^{-\mu x_1}\zeta^2) \der_1^kV\,d\rx.
  \end{split}
\end{equation*}

{\textbf{(5-4)}}
{\emph{Estimate of $A_1$:}}
If $l<k-2$, then $l+1\le m-2$ holds because $k\le m$. Therefore, we apply H\"{o}lder inequaltiy, Sobolev inequality and Poincar\'{e} inequality to get
\begin{equation*}
  \|(\der_1^{l+1}a_{11}^*)\zeta\der_1^{k-l}V\|_{L^2(\Om_{R_*})}
  \le C(R)\|a_{11}^*\|_{H^{m-1}(\Om_{R_*})}
  \|\der_1(\zeta\der_1^{k-l}V)\|_{L^2(\Om_{R_*})}.
\end{equation*}
Since $k-l+1\le k-1$ for all $l<k-2$, we apply the inductive assumption to get
\begin{equation}
\label{Hk-12}
  \sum_{l=2}^{k-3}\binom{k-2}{l}\|(\der_1^{l+1}a_{11}^*)\zeta\der_1^{k-l}V\|_{L^2(\Om_{R_*})}\le \til C\|\til f\|_{H^{k-2}(\Om_{R_*})}.
\end{equation}
For $l=k-2$,
\begin{equation}
\label{Hk-8}
  \|(\der_1^{l+1}a_{11}^*)\zeta\der_1^{k-l}V\|_{L^2(\Om_{R_*})}= \|\der_1^{k-1}a_{11}^*\zeta\der_1^{2}V\|_{L^2(\Om_{R_*})}
\end{equation}
If $k<m$, it follows from Sobolev inequality that $\|\der_1^{k-1}a_{11}^*\|_{L^4(\Om_{R_*})}\le C\|a_{11}^*\|_{H^{m-1}(\Om_{R_*})}$. Since $k\ge 4$ is assumed, by \eqref{G-cutoff prop 1} and the inductive assumption, we have
\begin{equation}
\label{Hk-9}
\begin{split}
\|(\der_1^{k-1}a_{11}^*)\zeta\der_1^{2}V\|_{L^2(\Om_{R_*})}
&\le
C(R)  \|a_{11}^*\|_{H^{m-1}(\Om_{R_*})}\|D(\zeta\der_1^{2}V)\|_{L^2(\Om_{R_*})}\\
&\le \til C\|\til f\|_{H^2(\Om_{R_*})}.
\end{split}
\end{equation}
If $k=m$,
we use \eqref{new perturbation condition for m=4} to get
\begin{equation}
\label{G:long-4}
\begin{split}
  &\|\der_1^{m-1}a_{11}^*\der_1^{2}V\zeta\|_{L^2(\Om_{R_*})}\\
  &\le \begin{cases}
   \|\der_{11}\bar{a}_{11}^*\|_{L^{\infty}({\Om_{R_*}})}\|\zeta\der_1^{2}V\|_{L^2(\Om_{R_*})}+
 C(R)\bar{\delta}\|\zeta\der_1^2V\|_{L^{\infty}(\Om_{R_*})}&\mbox{for $m=4$}\\
\|{a}_{11}^*\|_{H^{m-1}(\Om_{R_*})}\|\zeta\der_1^2V\|_{L^{\infty}(\Om_{R_*})}&\mbox{for $m>4$}
  \end{cases}.
  \end{split}
\end{equation}
Since $\zeta(0)=0$,
\begin{equation*}
  \begin{split}
  |\zeta\der_1^2V(x_1, x_2)|
  &\le \int_0^{x_1}|\der_1(\zeta\der_1^2V)(s, x_2)|\,ds\\
  &\le \sqrt{x_1}\left(\int_0^{x_1}|\der_1(\zeta\der_1^2V)(s, x_2)|^2\,ds\right)^{1/2}.
  \end{split}
\end{equation*}

Since $\der_2 V(x_1,-1)\equiv 0$, we can apply the trace inequality to get
\begin{equation*}
  \begin{split}
  \int_0^{x_1}|\der_1(\zeta\der_1^2V)(s, x_2)|^2\,ds
  &=\int_0^{x_1}|\der_1(\der_1^2V\zeta)(s,-1)+\int_{-1}^{x_2}\der_2\der_1(\zeta\der_1^2V)(s,t)\,dt|^2\,ds\\
  &\le 2\int_0^{x_1}|\der_1(\zeta\der_1^2V)(s,-1)|^2+|\int_{-1}^{x_2}\der_2\der_1(\zeta\der_1^2V)(s,t)\,dt|^2\,ds\\
  %&\le 2\|{\rm Tr}\der_1(\zeta\der_1^2V)\|_{L^2(\Gamw)}^2+2\int_0^{x_1}\int_{-1}^{x_2}|\der_2\der_1(\zeta\der_1^2V)(s,t)|^2\,dt\\
  &\le C(R_*)\left(\|\der_1(\zeta\der_1^2V)\|_{L^2(\Om_{R_*})}+\|D(\der_1(\zeta\der_1^2V))\|_{L^2(\Om_{R_*})}\right)^2.
  \end{split}
\end{equation*}
Therefore, we obtain that
\begin{equation*}
%\label{G:long-5}
  \|\der_1^2V\zeta\|_{L^{\infty}(\Om_{R_*})}\le C(R)\|\der_1(\zeta\der_1^2V)\|_{H^1(\Om_{R_*})}.
\end{equation*}

(Case 1) If $l=k-2$ and $k=m=4$, then we get
\begin{equation}
\label{Hk-10}
\begin{split}
  \|\der_1^{l+1}a_{11}^*\der_1^{k-l}V\zeta\|_{L^2(\Om_{R_*})} \le \til C\|\til f\|_{H^1(\Om_{R_*})}+C(R)\bar{\delta}\|\zeta D\der_1^3V\|_{L^2(\Om_{R_*})}.
  \end{split}
\end{equation}
\medskip

(Case 2) If $l=k-2$ and $k=m>4$, then we get
\begin{equation}
\label{Hk-11}
\begin{split}
  \|\der_1^{l+1}a_{11}^*\der_1^{k-l}V\zeta\|_{L^2(\Om_{R_*})} \le \til C\|\til f\|_{H^3(\Om_{R_*})}.
  \end{split}
\end{equation}

By combining all the estimates \eqref{Hk-8}--\eqref{Hk-11} with \eqref{Hk-12}, we obtain
\begin{equation*}
\begin{split}
  |A_1|\le \|\zeta \der_1^kV\|_{L^2(\Om_{R_*})}(\til C\|\til f\|_{H^{k-1}(\Om_{R_*})}+C(m,R)\bar{\delta}\|\zeta D\der_1^{k-1}V\|_{L^2(\Om_{R_*})}).
  \end{split}
\end{equation*}

One can similarly check that
\begin{equation*}
  \begin{split}
  |A_2|&\le \|\zeta \der_1^kV\|_{L^2(\Om_{R_*})}(\til C\|\til f\|_{H^{k-2}(\Om_{R_*})}+C(m,R)\bar{\delta}\|\zeta \der_1^{k}V\|_{L^2(\Om_{R_*})}),\\
  |A_3|&\le \til C \|\zeta \der_1^kV\|_{L^2(\Om_{R_*})} \|\til f\|_{H^{k-3}(\Om_{R_*})}.
  \end{split}
\end{equation*}
\medskip

{\textbf{(5-5)}} Finally, we conclude that
\begin{equation}
\label{R-1}
\begin{split}
  &|\mcl{R}_k^{(1)}(\mu)|\\
  &\le \|\zeta \der_1^kV\|_{L^2(\Om_{R_*})}\left(\til C\|\til f\|_{H^{k-2}(\Om_{R_*})}+C(m,R)\bar{\delta}\|\zeta D\der_1^{k-1}V\|_{L^2(\Om_{R_*})}\right).
  \end{split}
\end{equation}

{\textbf{Step 6.}} Estimate of $\mcl{R}_k^{(j)}$ for $j=2,3,4$:
By adjusting the arguments in step 5, and using the assumption
$$\|\til a_{12}\|_{H^{m-1}(\Om_{R_*})}\le C(R)\bar{\delta}$$
stated in \eqref{G:pert-condition-extension} and \eqref{G:pert-condition2-extension}, we have
\begin{equation*}
  \begin{split}
  &|\mcl{R}_k^{(2)}(\mu)|\\
  &\le C(m,R)\bar{\delta}\|\zeta\der_1^kV\|_{L^2(\Om_{R_*})}
  \left(\til C\|\til f\|_{H^{k-2}(\Om_{R_*})}+\|\zeta \der_1^{k-2}D^2 V\|_{L^2(\Om_{R_*})}\right).
  \end{split}
\end{equation*}

Using the equation $\eps\der_{111}V+\mcl{L}_*V=\til f$, we can express $\der_{22}V$ as
\begin{equation}
\label{Hk-13}
  \der_{22}V=\til f-\eps\der_{111}V-a_{11}^*\der_{11}V-2\til a_{12}\der_{12}V-a_1^*\der_1V.
\end{equation}
By straightforward computation with using the representation \eqref{Hk-13}, one can directly check that
\begin{equation*}
\begin{split}
  \|\zeta\der_1^{k-2}D^2V\|_{L^2(\Om_{R_*})}\le & \|\til f\|_{H^{k-2}(\Om_{R_*})}+\eps \|\zeta\der_1^{k+1}V\|_{L^2(\Om_{R_*})}\\
  &+C(\|(a_{11}^*, \til a_{12}, a_1^*)\|_{H^{m-1}(\Om_{R_*})})\|\zeta D \der_1^{k-1}V\|_{L^2(\Om_{R_*})}.
  \end{split}
\end{equation*}
So we can further estimate the term $|\mcl{R}_k^{(2)}(\mu)|$ as
\begin{equation}
\label{R-2}
\begin{split}
  &|\mcl{R}_k^{(2)}(\mu)|\\
  \le &\til C\bar{\delta}\|\zeta\der_1^kV\|_{L^2(\Om_{R_*})}
\left(\|\til f\|_{H^{k-2}(\Om_{R_*})}+\eps\|\zeta\der_1^{k+1}V\|_{L^2(\Om_{R_*})}+\|\zeta D\der_1^{k-1} V\|_{L^2(\Om_{R_*})}\right).
  \end{split}
\end{equation}
Finally, one can also easily check that
\begin{equation}
\label{R-3}
  \begin{split}
  |\mcl{R}^{(3)}_k|&\le \til C\|\zeta D\der_1^{k-1}V\|_{L^2(\Om_{R_*})}\|\til f\|_{H^{k-2}(\Om_{R_*})},\\
 |\mcl{R}^{(4)}_k|&\le \til C\|\zeta \der_1^{k}V\|_{L^2(\Om_{R_*})}\|\til f\|_{H^{k-2}(\Om_{R_*})}.
  \end{split}
\end{equation}
\medskip

{\textbf{Step 7.}} We use \eqref{RHS}, \eqref{coerceivity-ext} and \eqref{R-1}--\eqref{R-3}, and apply Cauchy-Schwarz inequality to derive from \eqref{Hk-5} that
\begin{equation}
\label{Hk-final1}
\begin{split}
&\int_{\Om_{R_*}} \eps(\der_1^{k+1}V)^2\zeta^2 e^{-\mu x_1}\,d\rx+\frac 12\min\left\{\frac{3\lambda}{16}, \frac{7\mu}{16}\right\}
  \int_{\Om_{R_*}}|D\der_1^{k-1}V|^2
  \zeta^2 e^{-\mu x_1}\,d\rx\\
&\le \til C\left((1+\mu)^2\|\til f\|^2_{H^{k-1}(\Om_{R_*})}+\bar{\delta}^2(\|\zeta D\der_1^{k-1}V\|_{L^2(\Om_{R_*})}+\eps\|\zeta\der_1^{k+1}V\|_{L^2(\Om_{R_*})})^2\right)
\end{split}
\end{equation}
provided that two positive constants $\bar{\delta}$ and $\mu$ are fixed to satisfy \eqref{mu-condition 1}, \eqref{mu-condition 2} and \eqref{coerceivity-ext}. And, we reduce the constant $\bar{\delta}>0$ further to conclude from \eqref{Hk-final1} that
\begin{equation}
\label{Hk-final2}
\begin{split}
&\int_{\Om_{R_*}} \eps(\der_1^{k+1}V)^2\zeta^2 e^{-\mu x_1}\,d\rx+\frac 14\min\left\{\frac{3\lambda}{16}, \frac{7\mu}{16}\right\}
  \int_{\Om_{R_*}}|D\der_1^{k-1}V|^2
  \zeta^2 e^{-\mu x_1}\,d\rx\\
&\le \til C(1+\mu)^2\|\til f\|^2_{H^{k-1}(\Om_{R_*})}.
\end{split}
\end{equation}
Note that the choice of $\mu$ is independent of $\bar{\delta}$.

Finally, we combine \eqref{Hk-13} with \eqref{Hk-final2} and the inductive assumption to get
\begin{equation*}
  \|\zeta \der_1^{k-2}\der_{22}V\|_{L^2(\Om_{R_*})}\le \til C\|\til f\|_{H^{k-1}(\Om_{R_*})}.
\end{equation*}
This completes the proof.

\end{proof}
\end{proposition}

\subsection{Proof of Theorem \ref{theorem-1-full}}
\label{subsec:2.3}
Now, we prove the main theorem of Section \ref{Section:2}.

%In this subsection, we give the proof of Theorem \ref{theorem-1}.

{\textbf{Step 1.}}  {\emph{Weak solution to \eqref{bvp-g}}}: Put
\begin{equation*}
  a_{21}:=a_{12},\quad a_{22}:=1\quad\tx{in $\Om$}.
\end{equation*}
Define $\mcl{B}_{\mcl{L}}:H^1(\Om)\times H^1(\Om)\rightarrow \R$ by
\begin{equation*}
  \mcl{B}_{\mcl{L}}[v,\phi]:=-\int_{\Om} \sum_{i,j=1}^2 a_{ij}\der_i v\der_j\phi+((\der_1 a_{11}+\der_2a_{12})\der_1 v+\der_1 a_{21}\der_2 v)\phi-a_1\phi\der_1 v\,d\rx.
\end{equation*}
If $v\in H^2(\Om)$ is a strong solution to \eqref{bvp-g}, then it holds that
\begin{equation}
\label{def:weak-sol_0}
  \mcl{B}_{\mcl{L}}[v,\phi]=\int_{\Om} f\phi\,d\rx
\end{equation}
for all $\phi\in C^{\infty}(\ol{\Om})$ with $\phi=0$ on $\Gamen\cup\Gamexg$.

\begin{definition}%[Weak solution to BVP \eqref{bvp-g}]
\label{G:definition of weak sol}
(i) Given a function $f\in L^2(\Om)$, we call $v\in H^1(\Om)$ {\emph{a weak solution to \eqref{bvp-g}}} if \eqref{def:weak-sol_0} holds for all $\phi\in C^{\infty}(\ol{\Om})$ with $\phi=0$ on $\Gamen\cup\Gamexg$.

(ii) Given a function $f\in L^2(\Om)$, we call
$$v^{\eps}\in\{w\in H^1(\Om): \der_{11}w\in L^2(\Om)\}$$
{\emph{a weak solution to \eqref{bvp-aux}}} if
\begin{equation}
\label{def:weak-sol}
-\int_{\Om} \eps \der_{11}v^{\eps}\der_1\phi \,d\rx+  \mcl{B}_{\mcl{L}}[v^{\eps},\phi]=\int_{\Om} f\phi\,d\rx
\end{equation} holds for all $\phi\in C^{\infty}(\ol{\Om})$ with $\phi=0$ on $\Gamen\cup\Gamexg$.
\end{definition}

\quad\\
{\textbf{Step 2.}} {\emph{Partially smooth approximations of $(a_{11}^*, \til a_{12}, a_1^*)$ in $\Om_{R_*}$}}:
Put
\begin{equation*}
  \Om_{R_*}^{\rm ext}:=(-1, R_*+1)\times (-1,1).
\end{equation*}

For $(a_{11}^*, a_1^*, \til a_{12})$ given in \S \ref{subsec:2.2} as an extension of $(a_{11}, a_1, a_{12})$ onto $\Om_{R_*}$, we now define its smooth approximation in $\Om_{R_*}$.

For each $m\ge 4$, let $\mcl{E}_m: H^{m-1}(\Om_{R_*})\rightarrow H^{m-1}(\Om_{R_*}^{\rm ext})$ be an extension operator satisfying the following properties:
\begin{itemize}
\item[(i)] $\mcl{E}_m$ is linear and bounded with its norm $\|\mcl{E}_m\|$ depending only on $(m, R_*)$;
\item[(ii)] $\mcl{E}_m \vphi=0$ holds on $\der \Om_{R_*}^{\rm ext}\cap\{|x_2|=1\}$ whenever $\vphi=0$ on $\der\Om_{R_*}\cap\{|x_2|=1\}$.
\end{itemize}
This type of operator can be easily constructed by applying a method of higher-order reflection about $x_1=0$ and $x_1=R_*$, respectively. Let $\phi:\R\rightarrow \R$ be a smooth mollifier that satisfies the following conditions:
\begin{itemize}
\item[(i)] $\phi(x_1)\ge 0$ for all $x_1\in \R$;
\item[(ii)] ${\rm spt}\phi\subset [-1,1]$;
\item[(iii)] $\int_{\R}\phi(x_1)\,dx_1=1$.
\end{itemize}
For a constant $\tau>0$, define $\phi^{\tau}:\R\rightarrow \R$ by
\begin{equation*}
  \phi^{\tau}(x_1):=\frac{1}{\tau}\phi\left(\frac{x_1}{\tau}\right).
\end{equation*}

For $\tau\in(0,1)$, define
\begin{equation}
\label{smooth approx-coeff}
\begin{split}
a_{11}^{\tau}&:=(\mcl{E}_m a_{11}^*) * \phi^{\tau}\\
a_{12}^{\tau}&:=(\mcl{E}_m\til a_{12})*\phi^{\tau}\\
a_1^{\tau}&:=(\mcl{E}_m a_1^*)*\phi^{\tau}\\
f^{\tau}&:=(\mcl{E}_m f)*\phi^{\tau}.
  \end{split}
\end{equation}
Then we have
\begin{equation}
\label{convergence in sobolev sp}
\begin{split}
  \lim_{\tau\to 0+}\|(a_{11}^{\tau}, a_{12}^{\tau}, a_1^{\tau})-(a_{11}^*, \til a_{12}, a_1^*)\|_{H^{m-1}(\Om_{R_*})}&=0,\\
\tx{and}\,\,  \lim_{\tau\to 0+}\|(a_{11}^{\tau}, a_{12}^{\tau}, a_1^{\tau})-(a_{11}, a_{12}, a_1)\|_{H^{m-1}(\Om_{R})}&=0.
  \end{split}
\end{equation}
Since $m\ge 4$, it follows from Morrey's inequality and \eqref{convergence in sobolev sp} that $(a_{11}^{\tau}, a_{12}^{\tau}, a_1^{\tau})$ converges to $(a_{11}^*, \til a_{12}, a_1^*)$ in $C^1(\ol{\Om_{R_*}})$. This also implies that $(a_{11}^{\tau}, a_{12}^{\tau}, a_1^{\tau})$ converges to $(a_{11}, a_{12}, a_1)$ in $C^1(\ol{\Om_{R}})$. Therefore, it follows from \eqref{KZ-condition} and \eqref{G:KZ-condition-ext} that there exists a constant $\bar{\tau}\in(0,1)$ such that for every $\tau\in(0, \bar{\tau}]$, we have
\begin{equation}
\label{Kz for smooth coeff}
  \begin{split}
  a_1^{\tau}+\frac{2k-3}{2}\der_1 a_{11}\tau-\der_2a_{12}^{\tau}&\le -\frac{4\lambda}{5}\quad\tx{in $\Om$}\\
   a_1^{\tau}+\frac{2k-3}{2}\der_1 a_{11}\tau-\der_2a_{12}^{\tau}&\le -\frac{2\lambda}{5}\quad\tx{in $\Om_{R_*}$}.
  \end{split}
\end{equation}

{\textbf{Step 3.}} {\emph{Galerkin's approximations}}: For each $\tau\in (0, \bar{\tau}]$, define
\begin{equation*}
\mcl{L}^{\tau}w:=a_{11}^{\tau}\der_{11}w+2a_{12}^{\tau}\der_{12}w+\der_{22}w+a_1^{\tau}\der_1w\quad\tx{in $\Om_{R_*}$}.
\end{equation*}

As an approximate boundary value problem of \eqref{bvp-aux}, consider
\begin{equation}\label{bvp-approx1}
  \begin{cases}
  \eps\der_{111}v+\mcl{L}^{\tau}v=f^{\tau}\quad&\mbox{in $\Om$}\\
  v=0,\quad \der_1 v=0 \quad&\mbox{on $\Gamen$}\\
  \der_2v=0\quad&\mbox{on $\Gamw$}\\
  \der_{11}v=0\quad&\mbox{on $\Gamexg$}.
  \end{cases}
\end{equation}

For
\begin{equation*}
  \Gam:=\{x_2\in \R: |x_2|<1\},
\end{equation*}
let $\{\eta_j\}_{j=1}^{\infty}$ be the set of all eigenfunctions of an eigenvalue problem:
\begin{equation}
\label{evp-for galerkin}
  -\eta''=\lambda\eta\,\,\tx{on $\Gam$}, \quad
  \eta'=0\,\,\tx{on $\der\Gam=\{\pm 1\}$}.
\end{equation}
One can take the set $\{\eta_k\}_{k=0}^{\infty}$ such that it forms
\begin{itemize}
\item[(i)] an orthonormal basis of $L^2(\Gam)$,
\item[(ii)] and an orthogonal basis of $H^k(\Gam)$ for $k=1,2$.
\end{itemize}
Note that every $\eta_k$ is smooth on $[-1,1]$.
Define $\langle \cdot, \cdot\rangle$ to be the standard scalar product in $L^2(\Gam)$, that is,
\begin{equation*}
  \langle \xi, \eta \rangle:=\int_{\Gam} \xi(x_2)\eta(x_2)\,dx_2.
\end{equation*}
Fix $n\in \mathbb{N}$. As an $n$-dimensional approximation of a solution to \eqref{bvp-approx1}, set
\begin{equation}
\label{galerkin-n}
\begin{split}
 v_n(x_1, x_2):=\sum_{j=1}^n \vartheta_j(x_1) \eta_j(x_2)
  \end{split}
\end{equation}
for $\rx=(x_1, x_2)\in \Om$.
We shall determine $\vartheta_j$ for $j=1,\cdots, n$ to satisfy
\begin{equation}
\label{G:galerkin-eqns}
  \langle \eps \der_{111}v_n+\mfrak{L}^{\tau}v_n, \eta_k\rangle=\langle f^{\tau}, \eta_k \rangle\quad\tx{for $0<x_1<R$}
\end{equation}
for all $k=1,...,n$, and
\begin{equation}
\label{G:galerkin-bcs}
  \begin{cases}
  v_n=0, \quad \der_1 v_n=0\quad&\tx{on $\Gamen$}\\
\der_{11}v_n=0\quad&\tx{on $\Gam_R$}.
  \end{cases}
\end{equation}
From \eqref{G:galerkin-eqns}--\eqref{G:galerkin-bcs}, we derive a boundary value problem for ${\bm \Theta}_n(x_1):=(\vartheta_1,\cdots, \vartheta_n)(x_1)$ as follows: $\forall k=1,\cdots,n$,
\begin{equation}
\label{G:Galerkin-ODE}
  \begin{cases}
  \eps\vartheta'''_k-\lambda_k\vartheta_k+\sum_{j=0}^n a_{11}^{\tau,jk}\vartheta''_j+(2a_{12}^{\tau, jk}+a_1^{\tau,jk})\vartheta'_j=f^{\tau,k}\quad \tx{for $0<x_1<R$}\\
  \vartheta_k(0)=\vartheta_k'(0)=\vartheta''_k(R)=0
  \end{cases}
\end{equation}
for
\begin{equation*}
  \begin{split}
  a_{11}^{\tau,jk}(x_1)&:=\langle a_{11}(x_1, \cdot)\eta_j, \eta_k\rangle\\
  a_{12}^{\tau,jk}(x_1)&:=\langle a_{12}(x_1, \cdot)\eta_j', \eta_k\rangle\\
  a_{1}^{\tau,jk}(x_1)&:=\langle a_{1}(x_1, \cdot)\eta_j, \eta_k\rangle\\
  f^{\tau,k}(x_1)&:=\langle f^{\tau}(x_1, \cdot)\eta_j, \eta_k\rangle .
  \end{split}
\end{equation*}
Note that $a_{11}^{\tau,jk}$, $a_{12}^{\tau,jk}$, $ a_{1}^{\tau,jk}$ and $f^{\tau,k}$ are smooth for $x_1\in[0,R_*]$.
\medskip

Define ${\bf X}_l^n:[0,R]\rightarrow \R^{n\times 1}$ for $l=1,2,3$ by
\begin{equation*}
  {\bf X}_1^n:=(\vartheta_1, \cdots, \vartheta_n)^T,\quad
  {\bf X}_2^n:={\bf X}'_1,\quad
  {\bf X}_3^n:={\bf X}''_1.
\end{equation*}
Set ${\bf X}^n: [0, R]\rightarrow \R^{3n\times 1}$ as
${\bf X}^n:=\begin{bmatrix}{\bf X}_1^n\\  {\bf X}_2^n\\ {\bf X}_3^n\end{bmatrix}$.
From the first line of \eqref{G:Galerkin-ODE},
\begin{equation*}
  \frac{d}{dx_1}{\bf X}_3^n=\frac{1}{\eps}\left( \mathbb{A}_1^{\tau, n}{\bf X}_1^n+\mathbb{A}_2^{\tau, n}{\bf X}_2^n+ \mathbb{A}_3^{\tau, n}{\bf X}_3^n+[f^{\tau,i}]_{i=1}^n\right)
\end{equation*}
for
\begin{equation*}
  \mathbb{A}_1^{\tau, n}:=[\lambda_i\delta_{ij}]_{i,j=1}^n,\quad \mathbb{A}_2^{\tau, n}:=-[2a_{12}^{\tau,ji}+a_1^{\tau,ji}]_{i,j=1}^n,\quad \mathbb{A}_3^{\tau, n}:=-[a_{11}^{\tau,ji}]_{i,j=1}^n.
\end{equation*}
So we can rewrite the first line of \eqref{G:Galerkin-ODE} as a first order linear system for ${\bf X}$ in the form of
\begin{equation*}
  \frac{d}{dx_1}{\bf X}^n=\mathbb{A}_{\eps}^{\tau, n}{\bf X}^n+{\bf F}^{\tau,n}_{\eps}
\end{equation*}
for smooth matrices $\mathbb{A}_{\eps}^{\tau, n}:[0, R]\rightarrow \R^{3n\times 3n}$ and ${\bf F}^{\tau,n}_{\eps}:[0, R]\rightarrow \R^{3n\times 1}$.
\medskip

Next, define a projection mapping $\Pi:\R^{3n\times 1}\rightarrow \R^{3n\times 1}$ by
\begin{equation*}
  \Pi{\bf X}^n:=\begin{bmatrix}{\bf X}_1^n\\  {\bf X}_2^n\\ {\bf 0}\end{bmatrix}.
\end{equation*}
Clearly, ${\bf X}^n$ yields a solution $v_n$ to \eqref{G:galerkin-eqns}--\eqref{G:galerkin-bcs} if and only if
\begin{equation}
\label{G:X-eqn}
  {\bf X}^n(x_1)=\Pi\int_0^{x_1}(\mathbb{A}_{\eps}^{\tau, n}{\bf X}^n+{\bf F}^{\tau,n}_{\eps})(t)\,dt
  +({\rm Id}-\Pi)\int_R^{x_1}(\mathbb{A}_{\eps}^{\tau, n}{\bf X}^n+{\bf F}^{\tau,n}_{\eps})(t)\,dt.
\end{equation}
Define a linear operator $\mfrak{K}: C^1([0, R];\R^{3n\times 1})\rightarrow C^1([0, R];\R^{3n\times 1})$ by
\begin{equation*}
  \mfrak{K}{\bf X}^n(x_1):=\Pi\int_0^{x_1}\mathbb{A}_{\eps}^{\tau, n}{\bf X}^n(t)\,dt
  +({\rm Id}-\Pi)\int_R^{x_1}\mathbb{A}_{\eps}^{\tau, n}{\bf X}^n(t)\,dt
\end{equation*}
to rewrite \eqref{G:X-eqn} as
\begin{equation*}
  ({\rm Id}-\mfrak K){\bf X}^n=\Pi\int_0^{x_1}{{\bf F}^{\tau,n}_{\eps}}(t)\,dt
  +({\rm Id}-\Pi)\int_L^{x_1}{{\bf F}^{\tau,n}_{\eps}}(t)\,dt.
\end{equation*}
Due to the smoothness of $\mathbb{A}_{\eps}^{\tau, n}$ and ${\bf F}^{\tau,n}_{\eps}$, the operator $\mfrak{K}$ is compact. Furthermore, if ${\bf X}^n\in C^1$ solves \eqref{G:X-eqn}, then
\begin{itemize}
\item[(i)] ${\bf X}^n\in C^{\infty}([0,R])$,
\item[(ii)] and the corresponding $v_n\in C^{\infty}(\ol{\Om})$ solves \eqref{G:galerkin-eqns}--\eqref{G:galerkin-bcs}.
\end{itemize}
So we can adjust the proof of Lemma \ref{lemma G:wp of singular pert prob-main} to show that
\begin{equation}
\label{G:Gal-est-H1}
  \begin{split}
  &\sqrt{\eps}\|\der_{11}v_n\|_{L^2(\Om)}+\|\der_2 v_n\|_{L^2(\Gamexg)}
  +\|v_n\|_{H^1(\Om)}\\
  &\le C(\lambda, \|a_{11}\|_{L^{\infty}(\Om)}, R) \|f\|_{L^2(\Om)}
  \end{split}
\end{equation}
for all $\eps\in(0,\bar{\eps}]$, $n\in \mathbb{N}$ and $\tau\in(0, \bar{\tau}]$.
Then, it follows from the Fredholm alternative theorem that, for each $\eps\in(0,\bar{\eps}]$, $\tau\in(0, \bar{\tau}]$ and $n\in \mathbb{N}$, \eqref{G:Galerkin-ODE} has a unique smooth solution.
In addition, we adjust the proof of Lemma \ref{G:lemma for pre H2 estimate of vm, part 2} to show that, for each $\eps\in(0,\bar{\eps}]$, $\tau\in(0, \bar{\tau}]$ and $n\in \mathbb{N}$,
\begin{equation}
\label{G:Gal-est-H2}
\begin{split}
  &\sqrt{\eps} \|\der_{111} v_n\|_{L^2(\Om\cap\{x_1>2r\})}
  %+\|\der_{12}v^{\eps}\|_{L^2(\Gamexg)}
  +\|D^2v_n\|_{L^2(\Om\cap\{x_1>2r\})}\\
 & \le C(\kappa_0,\lambda, \|(a_{11}, a_{12}, a_1)\|_{L^{\infty(\Om)}}, R, r)\|f\|_{H^1(\Om)}
\end{split}
\end{equation}
for any $r\in(0, \frac{R}{8}]$.
\medskip

{\textbf{Step 4.}}
For any given $\eps\in(0,\bar{\eps}]$ and $\tau\in(0, \bar{\tau}]$, we now have a sequence $\{v_n^{(\eps,\tau)}\}\subset C^{\infty}(\ol{\Om})$ such that
\begin{itemize}
\item[-] $v_n^{(\eps,\tau)}$ solves \eqref{G:galerkin-eqns}--\eqref{G:galerkin-bcs}.
\item[-] $\{v_n^{(\eps,\tau)}\}$ is bounded in $H^1(\Om)$, and $\{\der_{11}v_n^{(\eps,\tau)}\}$ is bounded in $L^2(\Om)$ (see \eqref{G:Gal-est-H1}).
\item[-] $\forall r\in(0, \frac R8]$, $\{v_n^{(\eps,\tau)}\}$ is bounded in $H^2(\Om\cap \{x>2r\})$, and $\{\der_{111}v_n^{(\eps,\tau)}\}$ is bounded in $L^2(\Om\cap\{x_1>2r\})$ (see \eqref{G:Gal-est-H2}).
\end{itemize}
Therefore, by a diagonal process, one can take a subsequence $\{v_{n_j}^{(\eps, \tau)}\}$ of $\{v_n^{(\eps, \tau)}\}$, and find a function $v^{(\eps, \tau)}\in H^1(\Om)\cap H^2_{{\rm loc}}(\Om)$ such that
\begin{itemize}
\item[(i)] $v_{n_j}^{(\eps, \tau)} \rightharpoonup v^{(\eps, \tau)}$ in $L^2(\Gam_0)\cap H^1(\Om)$, and $\der_{11} v_{n_j}^{(\eps, \tau)}\rightharpoonup \der_{11} v^{(\eps, \tau)}$ in $L^2(\Om)$;
\item[(ii)] $v_{n_j}^{(\eps, \tau)} \rightharpoonup v^{(\eps, \tau)}$ in $H^2(\Om\cap \{x_1>2r\})$, and  $\der_{111} v_{n_j}^{(\eps, \tau)}\rightharpoonup \der_{111} v^{(\eps, \tau)}$ in $L^2(\Om\cap\{x_1>2r\})$ for all $r\in(0, \frac R8]$.
\end{itemize}
It follows from (i) that
\begin{equation*}
  -\int_{\Om} \eps\der_{11}v^{(\eps, \tau)}\der_1\phi\,d\rx+\mcl{B}_{\mcl{L^{\tau}}}[v^{(\eps, \tau)},\phi]=\int_{\Om} f^{\tau}\phi\,d\rx
\end{equation*}
for all $\phi\in C^{\infty}(\ol{\Om})$ with $\phi=0$ on $\Gamen\cup\Gamexg$, and that
\begin{equation*}
  v^{(\eps, \tau)}=0\quad\tx{on $\Gam_0$ in the trace sense}.
\end{equation*}
And, it follows from (ii) that
\begin{equation*}
\begin{split}
  \eps\der_{111}v^{(\eps, \tau)}+\mcl{L}^{\tau}v^{(\eps, \tau)}=f^{\tau}\quad &\tx{a.e. in $\Om$},\\
  \der_2 v^{(\eps, \tau)}=0\quad &\tx{on $\Gamw$ in the trace sense}.
  \end{split}
\end{equation*}
Note that the estimates \eqref{G:Gal-est-H1} and \eqref{G:Gal-est-H2} hold for $v^{(\eps, \tau)}$. Therefore, one can take a sequence $\{(\eps_n ,\tau_n)\}\subset (0, \bar{\eps}]\times (0, \bar{\tau}]$, and find a function $u\in H^1(\Om)\cap H^2_{\rm loc}(\Om)$ such that
\begin{itemize}
\item[(i)] $(\eps_n ,\tau_n)\rightarrow (0,0)$ as $n\to \infty$;
\item[(ii)] $v^{(\eps_n ,\tau_n)} \rightharpoonup u$ in $L^2(\Gam_0)\cap H^1(\Om)$;
\item[(iii)] $\eps_n\der_{11}v^{(\eps_n ,\tau_n)} \rightarrow 0$ in $L^2(\Om)$;
\item[(iv)] $v^{(\eps_n ,\tau_n)} \rightharpoonup u$ in $H^2(\Om\cap \{x_1>2r\})$;
\item[(v)] $\eps_n\der_{111} v^{(\eps_n ,\tau_n)}\rightarrow 0$ in $L^2(\Om\cap\{x_1>2r\})$ for all $r\in(0, \frac R8]$.
\end{itemize}
Then it follows from (ii) and (iii) that
\begin{equation}
\label{G:v_weak sol}
  \mcl{B}_{\mcl{L}}[u,\phi]=\int_{\Om} f\phi\,d\rx
\end{equation}
for all $\phi\in C^{\infty}(\ol{\Om})$ with $\phi=0$ on $\Gamen\cup\Gamexg$, and that
\begin{equation*}
  u=0\quad\tx{on $\Gam_0$ in the trace sense}.
\end{equation*}
And, it follows from (iv) and (v) that
\begin{equation*}
\begin{split}
\mcl{L}u=f\quad &\tx{a.e. in $\Om$},\\
  \der_2 u=0\quad &\tx{on $\Gamw$ in the trace sense}.
  \end{split}
\end{equation*}

\medskip

{\textbf{Step 5.}} It follows from \eqref{G:ellipticity} and \eqref{G:v_weak sol} that $v$ is a weak solution to a uniformly elliptic equation in $\Om\cap\{x_1<R_1\}$. By applying a standard elliptic theory along with the compatibility conditions \eqref{comp for even ext} and \eqref{comp for odd ext}, one can show that
\begin{equation}
\label{G:final estimate 1}
  \|u\|_{H^m(\Om\cap\{x_1<\frac{R_1}{2}\})}\le C(m,R_1, \kappa_0, \|(a_{11},a_{12},a_1)\|_{H^{m-1}(\Om)})\|f\|_{H^{m-2}(\Om)}.
\end{equation}
\medskip

{\textbf{Step 6.}} By applying Lemmas \ref{lemma G:wp of singular pert prob-ext} and \ref{G:lemma for pre H2 estimate-ext}, we can also show that the boundary value problem \eqref{bvp-aux-ext}
has a weak solution $V\in H^1(\Om_{R_*})\cap H^2_{{\rm loc}}(\Om_{R_*})$ that satisfies
\begin{equation*}
  \mcl{B}_{\mcl{L}^*}[V,\phi]=\int_{\Om_{R_*}} \til f \phi\,d\rx
\end{equation*}
for all $\phi\in C^{\infty}(\ol{\Om_{R_*}})$ with $\phi=0$ on $\Gamen\cup\Gam_{R_*}$, and  $V$ is a strong solution to \eqref{bvp-aux-ext}. Then it follows from Lemma \ref{lemma:extension-consistence} that
\begin{equation}
\label{v is V}
  u=V\quad\tx{in $\Om$}.
\end{equation}

In addition, by a limiting process with applying Proposition \ref{proposition:2.10}, we get
\begin{equation}
\label{G:Hm estimate1 of v}
\|V\|_{H^1(\Om_{R_*})}+\|D^2\der_1^{k-2}V\|_{L^2(\Om_{R_*}\cap\{r<x_1<R_*-r\})}\le \til{C}\|\til f\|_{H^{k-1}(\Om_{R_*})}
\end{equation}
for all $k=2,\cdots, m$, and any $r>0$ sufficiently small.
\iffalse
Here, the estimate constant $C_*>0$ has the same dependence as $C_*$ from \eqref{G:final H-k estimate}.
, that is,
\begin{equation*}
  C^*=\begin{cases}
  C^*(r,R, m,  \lambda, \kappa_0, \om, \|(a_{11}^*,\til a_{12}, a_1^*)\|_{H^{m-1}(\Om_{R_*})},  \|\der_{11}\bar{a}_{11}^*\|_{L^{\infty}(\Om_{R_*})})&\mbox{if $m=4$}\\
  C^*(r, R, m, \lambda, \kappa_0, \om,  \|(a_{11}^*,\til a_{12}, a_1^*)\|_{H^{m-1}(\Om_{R_*})})&\mbox{if $m>4$}
  \end{cases}.
\end{equation*}
\fi

Fix a constant $\beta_0>0$ sufficiently large to satisfy
\begin{equation*}
\begin{pmatrix}
  a^*_{11}+\beta_0&\til a_{12}\\
  \til a_{12}&1
  \end{pmatrix}\ge \frac{\min\{\beta_0, 1\}}{2}\mathbb{I}_2\quad \tx{in $\ol{\Om_{R_*}}$}.
\end{equation*}
We can regard $V$ as a strong solution to the following uniformly elliptic equation:
\begin{equation*}
  (a_{11}^*+\beta_0)\der_{11}V+2\til a_{12}\der_{12}V+\der_{22}V+a_1\der_1V=\til F\quad\tx{in $\Om_{R_*}$}.
\end{equation*}
for
\begin{equation*}
  \til F:=\til f+\beta_0\der_{11}V.
\end{equation*}
Then, by applying standard elliptic theory together with \eqref{G:Hm estimate1 of v} and bootstrap arguments, we obtain
\begin{equation*}
  \|V\|_{H^m(\Om_{R_*}\cap\{r<x_1<R_*-r\})}\le \til{C}\|\til f\|_{H^{m-1}(\Om_{R_*})}.
\end{equation*}
Then it follows from \eqref{v is V} that
\begin{equation*}
  \|u\|_{H^m(\Om\cap\{x_1>\frac{R_1}{4}\})}\le \til{C}\|f\|_{H^{m-1}(\Om)}.
\end{equation*}
\iffalse
for some constant $C>0$ fixed with the following dependence:
\begin{equation*}
  C=\begin{cases}
  C(r,R, m,  \lambda, \kappa_0, \om, \|(a_{11}, a_{12}, a_1)\|_{H^{m-1}(\Om_{R})},  \|\der_{11}\bar{a}_{11}\|_{L^{\infty}(\Om_{R})})&\mbox{if $m=4$}\\
  C(r, R, m, \lambda, \kappa_0, \om,  \|(a_{11}, a_{12}, a_1)\|_{H^{m-1}(\Om_{R})})&\mbox{if $m>4$}
  \end{cases}.
\end{equation*}
\fi
By combining \eqref{G:final estimate 1} with this estimate, we complete the proof. \qed

\section{Accelerating transonic solution to the steady Euler-Poisson system}\label{Section:3}
\iffalse
In this section, we present the major application of the results for mixed type equation to the Euler-Poisson system.
\fi
In this section, we present a precise version of Theorem \ref{theorem-2-pre}, and give its proof.

\subsection{Preliminary}

\begin{comment}
\begin{equation}\label{UnsteadyEP}
\left\{
\begin{aligned}
& \Div_{\rx} (\rho \bu)=0, \\
&\Div_{\rx} (\rho \bu \otimes \bu) +\nabla_{\bx} p=\rho \nabla_{\rx} \Phi, \\
& \Div_{\rx}(\rho\msE \bu +p\bu)=\rho \bu\cdot \nabla_{\bx}\Phi,\\
& \Delta_{\rx} \Phi=\rho-\barrhoi,%
\end{aligned}%
\right.
\end{equation}
In this paper, we consider ideal polytropic gas, where $p$ and $\msE$ are given by
\begin{equation}
p=S\rho^{\gam}\quad\tx{and}\quad \msE=\frac{|\bu|^2}{2}+\frac{p}{(\gam-1)\rho},
\end{equation}
respectively, for a function $S>0$ and a constant $\gam>1$.
%Here, the function $S$ is to be determined by solving the system \eqref{UnsteadyEP}, and
The function $\ln S$ represents {\emph{the physical entropy}}, and the constant $\gam>1$ is called the {\emph{adiabatic exponent}}. For the Mach number $M$, defined by
$$M:=\frac{|{\bf u}|}{\sqrt{\gam p/\rho}},$$
if $M<1$, then the flow corresponding to $(\rho, {\bf u}, p, \Phi)$ is said to be subsonic. On the other hand, if $M>1$, then the flow is said to be supersonic. Finally, if $M=1$, then the flow is said to be sonic.
\medskip
\end{comment}
Suppose that $(\rho, {\bf u}, S, \Phi)$ is a classical solution to \eqref{steadyEP0}, and satisfies $\rho>0$ and $u_1>0$, where we set the velocity vector field ${\bf u}$ as ${\bf u}(\rx):=u_1(\rx){\bf e}_1+u_2(\rx){\bf e}_2$ for ${\rx}=(x_1, x_2)\in \R^2$. Here, ${\bf e}_j(j=1,2)$ represents the unit vector in the positive $x_j$-direction. we define the vorticity function $\om$ by
\begin{equation*}
\om(\rx):=\der_{x_1}u_2-\der_{x_2}u_1.
\end{equation*}
\iffalse
As we seek for a two-dimensional steady solution to \eqref{UnsteadyEP}, let us set the velocity vector field ${\bf u}$ as ${\bf u}(\rx):=u_1(\rx){\bf e}_1+u_2(\rx){\bf e}_2$ for ${\rx}=(x_1, x_2)\in \R^2$. Here, ${\bf e}_j(j=1,2)$ represents the unit vector in the positive $x_j$-direction. Next, we define the vorticity function $\om$ by
$$\om(\rx):=\der_{x_1}u_2-\der_{x_2}u_1.$$
If the solution satisfies $\rho>0$ and $u_1>0$, then $(\rho, {\bf u}, S, \Phi)$ solves the following nonlinear system:
\fi
Then $(\rho, {\bf u}, S, \Phi)$ solves
\begin{equation}\label{steadyEP}
  \left\{
\begin{aligned}
& \nabla\cdot (\rho \bu)=0, \\
&\om=\frac{1}{u_1}\left(\frac{\rho^{\gam-1}\der_{x_2} S}{(\gam-1)}-\der_{x_2}(\msB-\Phi)\right), \\
& {\bf m}\cdot \nabla S=0,\\
& {\bf m}\cdot \nabla(\msB-\Phi)=0,\\
& \Delta\Phi=\rho-\barrhoi,
\end{aligned}%
\right.
\end{equation}
where ${\bf m}$ and $\msB$ are given by
$$
{\bf m}:=\rho{\bf u},\quad\tx{and}\quad
\msB:=\frac{|{\bf u}|^2}{2}+\frac{\gam S\rho^{\gam-1}}{\gam-1}.
$$
Note that \eqref{steadyEP0} is equivalent to \eqref{steadyEP} as long as $\rho>0$ and $u_1>0$ hold.

If $(\bar{\rho}, {\bf u}, \bar{S}, \bar{\Phi})(x_1)$ with ${\bf u}=u_1{\bf e}_1$ is a one-dimensional solution to \eqref{steadyEP}, then it can be obtained by solving the ODE system
%For one-dimensional solutions to \eqref{steadyEP}, we  solve the ODE system
\begin{equation*}
  \left\{
\begin{aligned}
& \bar{\rho}\bar u_1=J \\
&\bar S=S_0\\
&\frac{\bar u_1^2}{2}+\frac{\gam \bar S\bar{\rho}^{\gam-1}}{\gam-1}-\bar{\Phi}=\kappa_0\\
& \bar{\Phi}''=\bar{\rho}-\barrhoi
\end{aligned}%
\right.
\end{equation*}
for some constants $J>0$, $S_0>0$ and $\kappa_0\in \R$.
By setting $\bar E:=\bar{\Phi}'$, one can rewrite this system as
\begin{equation*}
\bar{\rho}=\frac{J}{\bar u_1},\quad
   \bar S=S_0,\quad
  \begin{cases}\bar u_1'=\frac{\bar E \bar u_1^{\gam}}{\bar u_1^{\gam+1}-\gam S_0J^{\gam-1}}\\
  \bar E'=\frac{J}{\bar u_1}-\barrhoi
  \end{cases}.
\end{equation*}
In the following, the constants $\gam >1$, $J>0$, and $S_0>0$ are fixed.

Define two constants $\barui$ and $\us$ by
\begin{equation*}
  \barui:=\frac{J}{\barrhoi},\quad\tx{and}\quad \us:=\left(\gam S_0J^{\gam-1}\right)^{\frac{1}{\gam+1}}.
\end{equation*}
%let us assume that
%\begin{equation}
%\label{assumption-b0}
%  \us<\barui.
%\end{equation}
For later use, we also define a parameter $\zeta_0$ by
\begin{equation}\label{definition of zeta0}
  \zeta_0:=\frac{\barui}{\us}.
\end{equation}
In this paper, we assume that
\begin{equation}
\label{condition for zeta0}
\zeta_0>1.
\end{equation}

Suppose that $(\bar u_1, \bar E)$ is a $C^1$-solution to
\begin{equation}
\label{EP-1d-reduced}
   \begin{cases}
  \bar u_1'=\frac{\bar E \bar u_1^{\gam}}{\bar u_1^{\gam+1}-\us^{\gam+1}},\\
  \bar E'=\frac{J}{\bar u_1}-\barrhoi.
  \end{cases}
\end{equation}
In addition, suppose that
\begin{equation*}
(\bar u_1, \bar{E})(\ls)=(\us, 0)
\end{equation*}
for some constant $\ls>0$. Next, we define a function $H:(0,\infty)\rightarrow \R$ by
\begin{equation}\label{definition of H}
  H(u)=\int_{\us}^{u} \frac{J}{\barui t^{\gam+1}}(t^{\gam+1}-\us^{\gam+1})
\left(\barui-t\right)dt.
\end{equation}
Also, define a function $\mfrak{H}:(0,\infty)\times \R\rightarrow \R$ by
\begin{equation*}
\mfrak{H}(u, E):=\frac{1}{2}E^2-H(u).
\end{equation*}
Then it can be directly checked by straightforward computations that $(\bar u_1, \bar E)$ satisfies
\begin{equation}
\label{1d-Hamiltonian}
  \mfrak{H}(\bar u_1, \bar E)\equiv 0
\end{equation}
as long as it exists.
\smallskip

\begin{psfrags}
\begin{figure}[htp]
\centering
\psfrag{A}[cc][][0.8][0]{$\phantom{aaaa}(u_0, E_0)$}
\psfrag{B}[cc][][0.8][0]{$u_s$}
\psfrag{C}[cc][][0.8][0]{$\barui$}
\psfrag{u}[cc][][0.8][0]{$u$}
\psfrag{E}[cc][][0.8][0]{$E$}
\includegraphics[scale=0.5]{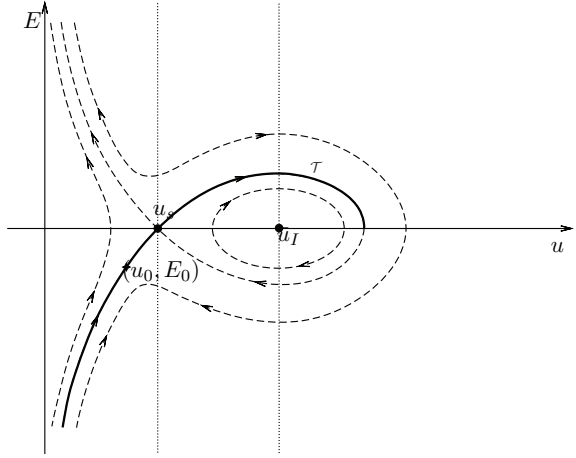}
\caption{The critical trajectory}\label{figure1}
\end{figure}
\end{psfrags}
Define
\begin{equation*}
\mcl{T}:=\{(u, E):\mfrak{H}(u, E)=0\}.
\end{equation*}
We refer to the set $\mcl{T}$ as \emph{critical trajectory} of the system \eqref{EP-1d-reduced}(see Figure \ref{figure1}).
We define a subset $\mcl{T}_{\rm{acc}}$ by
\begin{equation}
\label{definition of T and Tpm}
  \mcl{T}_{\rm{acc}}=\{(u, E)\in\mcl{T}:  (u-u_s) E \ge 0 \}.
\end{equation}

\begin{lemma}%[One dimensional smooth transonic solutions]
\label{lemma-1d-full EP}
Given constants $\gam> 1$, $J>0$ and $S_0>0$, assume that the condition \eqref{condition for zeta0} holds.
Then, the set $\mcl{T}_{\rm{acc}}$ represents the trajectory of an accelerating smooth transonic solution to \eqref{EP-1d-reduced} in the following sense: for any fixed $(u_0, E_0)\in \mcl{T}_{\rm{acc}}$ with $u_0<\us$, there exists a finite constant $l_{\rm max}>0$ depending only on $(\gam, J, S_0, \barui, u_0)$ so that the initial value problem
\begin{equation}
\label{1d-ivp}
\begin{split}
   &\begin{cases}
  \bar u_1'=\frac{\bar E \bar u_1^{\gam}}{\bar u_1^{\gam+1}-\us^{\gam+1}}\\
  \bar E'=\frac{J}{\bar u_1}-\barrhoi
  \end{cases}
\tx{for $x_1>0$},\\
&(\bar u_1, \bar E)(0)=(u_0, E_0)
\end{split}
\end{equation}
has a unique smooth solution for $x_1\in[0, l_{\rm max})$ with satisfying the following properties:
\begin{itemize}
\item[(i)] $\bar{u}_1'(x_1)>0$ on $[0, l_{\rm max})$.
\item[(ii)] $\displaystyle{\lim_{x_1\to l_{\rm max}-}\bar u_1'(x_1)=0}$.
\item[(iii)] $\displaystyle{\mcl{T}_{\rm acc}\cap\{(\bar u_1, \bar E)(x_1):0\le x_1\le l_{\rm max}\}=\mcl{T}_{\rm acc}\cap \{(u, E):u\ge u_0\}}$.
\item[(iv)] There exists a unique constant $\ls\in(0, l_{\rm max})$ depending only on $(\gam, J, S_0, \barui, u_0)$ such that $\bar u_1$ satisfies
    \begin{equation}
    \label{definition:ls}
      \bar u_1(x_1)\begin{cases}
      <\us\quad &\mbox{for $x_1<\ls$ {\emph{(subsonic)}}},\\
      =\us\quad &\mbox{at $x_1=\ls$ {\emph{(sonic)}}},\\
      >\us\quad &\mbox{for $x_1>\ls$ {\emph{(supersonic)}}}.
      \end{cases}
    \end{equation}
\end{itemize}

\begin{proof}
Define a function $F(t)$ by
\begin{equation}\label{definition of F in ode}
  F(t):=\frac{t^{\gam}\sqrt{2H(t)}}{|t^{\gam+1}-\us^{\gam+1}|}\quad\tx{for $t\neq \us$}
\end{equation}
for $H(u)$ given by \eqref{definition of H}.
Clearly, $F(t)$ is smooth for $t\in(0,\infty)\setminus \{\us\}$ as long as $H(t)\ge 0$.

Note that $H(\us)=H'(\us)=0$. It follows from the assumption \eqref{condition for zeta0} that
$$H''(\us)=(\gam+1)J\left(\frac{1}{\us}-\frac{1}{\barui}\right)>0.$$
By applying L'H\^{o}pital's rule, it is obtained that
\begin{equation*}
\lim_{t\to \us}\frac{2H(t)}{(t^{\gam+1}-\us^{\gam+1})^2}=
\lim_{t\to \us} \frac{H''(t)}{(\gam+1)^2t^{2\gam}}=\frac{J}{(\gam+1)\us^{2\gam}}
\left(\frac{1}{\us}-\frac{1}{\barui}\right)>0,
\end{equation*}
so one can extend the definition of $F(t)$ up to $t=\us$ as
\begin{equation*}
  F(\us):=
  \sqrt{\frac{J}{\gam+1}\left(\frac{1}{\us}-\frac{1}{\barui}\right)}.
\end{equation*}
Therefore, $F(t)$ is continuous for all $t>0$ as long as $H(t)\ge 0$ holds.

By Taylor's theorem, there exist two smooth functions $P(t)$ and $Q(t)$ satisfying
\begin{equation*}
\begin{split}
  &H(t)=\frac 12 H''(\us)(t-\us)^2+(t-\us)^3 P(t),\\
  &t^{\gam+1}-\us^{\gam+1}=(\gam+1)\us^{\gam} (t-\us)+(t-\us)^2Q(t),
\end{split}
\end{equation*}
so that the function $F(t)$ can be written as
\begin{equation*}
  F(t)=\frac{t^{\gam}\sqrt{H''(\us)+2(t-\us)P(t)}}{(\gam+1)\us^{\gam}+(t-\us)Q(t)}.
\end{equation*}
This shows that $F(t)$ is smooth and positive for all $t>0$ as long as the inequality $H(t)\ge 0$ holds.

By the definition \eqref{definition of T and Tpm}, one has
\begin{equation*}
  \Tac=\{E=\sgn (u-\us)\sqrt{H(u)}: H(u)\ge 0\}.
\end{equation*}
If $(\bar u_1, \bar E)(x_1)$ is a $C^1$-solution to \eqref{1d-ivp}, then it also solves the initial value problem:
\begin{equation}
\label{ivp-accelerating}
  \begin{cases}
  \bar u_1'=F(\bar u_1)\\
  \bar u_1(0)=u_0
  \end{cases}\quad\tx{and}\quad \bar E=\sgn(\bar u_1-\us)\sqrt{2H(\bar u_1)}.
\end{equation}
One can directly check from \eqref{definition of H} that there exists a constant $u_{\rm max}\in(\barui, \infty)$ satisfying that
\begin{equation*}
  H(t)\begin{cases}
  >0\quad&\mbox{for $0<t<u_{\rm max}$},\\
  =0\quad&\mbox{at $t=u_{\rm max}$},\\
  <0\quad&\mbox{for $t>u_{\rm max}$}.
  \end{cases}
\end{equation*}
Hence it follows from the unique existence theorem of ODEs and the method of continuation that there exists a finite constant $l_{\rm max}>0$ so that the initial value problem \eqref{ivp-accelerating} has a unique smooth solution for $0\le x_1\le l_{\rm max}$ with $(\bar u, \bar E)(l_{\rm max})=(u_{\rm max},0)$. Note that the inequality
\begin{equation*}
 0<u_0<\us<\barui<u_{\rm max}
\end{equation*}
holds. Since $F(t)>0$ for $0<t\le u_{\rm max}$, the velocity function $\bar u_1$ strictly increases with respect to $x_1\in[0, l_{\rm max})$. Therefore, there exists a unique constant $\ls\in(0, l_{\rm max})$ that satisfies \eqref{definition:ls}.

\end{proof}
\end{lemma}

\subsection{Restatement of Theorem \ref{theorem-2-pre}}
\label{subsection-the main result}
\iffalse
The main goal of this work is to construct a two-dimensional solution of the system \eqref{steadyEP} when prescribing the boundary conditions as small perturbations of a one-dimensional accelerating smooth transonic solution introduced in Lemma \ref{lemma-1d-full EP}.
For isentropic irrotational flows, it holds that $\msB-\Phi$ is globally a constant(see \cite{BDX}).
\fi
In this paper, we assume that
\begin{equation*}
  \msB-\Phi= 0\quad\tx{({\emph{the pseudo-Bernoulli law}}  )}
\end{equation*}
for technical simplicity.
Then, we can rewrite \eqref{steadyEP} as
\begin{equation}\label{full EP rewritten}
  \left\{
\begin{aligned}
& \nabla\cdot {\rho{\bf u}}=0, \\
&\nabla\times {\bf u}=\frac{\rho^{\gam-1}\der_{x_2} S}{(\gam-1)u_1},\\
&  {\rho{\bf u}}\cdot \nabla S=0,\\
&\frac 12|{\bf u}|^2+\frac{\gam S \rho^{\gam-1}}{\gam-1}=\Phi,\\
& \Delta \Phi=\rho-\barrhoi.
\end{aligned}%
\right.
\end{equation}
For a constant $L>0$, let us define
\begin{equation}
\label{definition of Omeage L}
  \Om_L:=\{\rx=(x_1, x_2)\in \R^2: 0<x_1<L, |x_2|<1\}.
\end{equation}
The boundary $\der \Om_L$ consists of three parts
\begin{equation}
\label{definition of der Omeage L}
  \Gamen:=\{0\}\times [-1,1],\quad \Gamw:=(0,L)\times\{-1,1\},\quad \Gamex:=\{L\}\times [-1,1],
\end{equation}
which represent the entrance, the wall and the exit of $\Om_L$,
respectively.
\iffalse
\begin{notation}\begin{itemize}
\item[(1)] Throughout this paper, we let ${\bf e}_j$ represent the unit vector in the positive direction of the $x_j$-axis, that is,
${\bf e}_1={(1,0)}$ and ${\bf e}_2={(0,1)}.$

\item[(2)]  For $i=1,2$, the symbols $\der_i$ and $\der_{ij}$ represent $\frac{\der}{\der x_i}$ and $\frac{\der^2}{\der x_i\der x_j}$, respectively. And, any partial derivative of a higher order is denoted similarly.
\end{itemize}
\end{notation}
\fi

\begin{definition}[Background solutions]
\label{definition of background solution-full EP}
Given constants $\gam > 1$, $S_0>0$, $J>0$ and $\zeta_0>1$(see \eqref{definition of zeta0} for the definition), fix a point $(u_0, E_0)\in \Tac$ with $E_0 < 0$. Then Lemma \ref{lemma-1d-full EP} implies that the initial value problem \eqref{EP-1d-reduced} with
$$(\bar u_1, \bar E)(0)=(u_0,E_0)$$
has a unique smooth solution on the interval $[0, l_{\rm max})$. Furthermore, the flow governed by the solution is accelerating. In other words, it holds that
$$
\frac{d\bar u_1}{dx_1}>0\quad\tx{on $[0, l_{\rm max})$.}
$$
For ${\rx}=(x_1, x_2)\in \R^2$ with $x_1\in[0, l_{\rm max})$, let us define $(\bar{\rho}, {\bar{\bf u}}, \bar{\Phi})$ by
\begin{equation*}
\begin{split}
& \bar{\rho}(\rx):=\frac{J}{\bar u_1(x_1)},\quad
\bar{\bf u}(\rx):=\bar u_1(x_1){\bf e}_{1},\quad
\bar{\Phi}(\rx):=\int_0^{x_1}\bar E(t)\,dt+\frac{u_0^2}{2}+\frac{\gam S_0}{\gam-1}\left(\frac{J}{u_0}\right)^{\gam-1}.
  \end{split}
\end{equation*}
Then, $(\rho, {\bf u},S, \Phi)=(\bar{\rho}, \bar{\bf u}, S_0, \bar{\Phi})$ solves the system \eqref{full EP rewritten} in $\Om_{l_{\rm max}}$. We call $(\bar{\rho}, \bar{\bf u}, S_0, \bar{\Phi})$ the background smooth transonic solution to \eqref{full EP rewritten} associated with $(\gam, \zeta_0, J, S_0, E_0)$. Note that the initial value $u_0$ of $\bar u_1$ is uniquely determined depending on $E_0$ as the point $(u_0, E_0)$ lies on the curve $\Tac$, defined by \eqref{definition of T and Tpm}, and the function $\sgn (u-\us)\sqrt{H(u)}$ is monotone for $u\in (0, u_{\rm max}]$. So, it follows from the inequality $E_0<0$ that $0<u<\us$.
\end{definition}

\begin{problem}%[Two dimensional continuous transonic flow with nonzero vorticity]
\label{problem-full EP} Fix $m\in \mathbb{N}$ with $m\ge 4$. Given constants $\gam>1$, $\zeta_0>1$, $J>0$ and $S_0>0$, fix a point $(u_0, E_0)\in \Tac$ with $E_0<0$. Given $(S_{\rm en}, E_{\rm en}, \om_{\rm en}):[-1,1]\rightarrow \R^3$, define
\begin{equation}
\label{definition-perturbation of bd}
      \mfrak P(S_{\rm en}, E_{\rm en}, w_{\rm en}):=\|S_{\rm en}-S_0\|_{C^m([-1,1])}+\|E_{\rm en}-E_0\|_{C^m([-1,1])}
    +\|w_{\rm en}\|_{C^{m+1}([-1,1])}.
    \end{equation}
Assuming that $\mfrak P(S_{\rm en}, E_{\rm en}, w_{\rm en})$ is sufficiently small, find a solution $(\rho, {\bf u}, S, \Phi)$ to \eqref{full EP rewritten} in $\Om_L$ with
the boundary conditions:
  \begin{equation}\label{BC-full EP}
  \begin{split}
  {\bf u}\cdot {\bf e}_{2}=w_{\rm en},\quad S=S_{\rm en},\quad \tx{and}
  \quad \der_{x_1}\Phi=E_{\rm en}\quad&\tx{on $\Gamen$},\\
  {\bf u}\cdot {\bf n}=0\quad\tx{and}\quad \der_{\bf n}\Phi=0\quad&\mbox{on $\Gamw$},\\
  \Phi=\bar{\Phi}\quad&\mbox{on $\Gamex$},
  \end{split}
\end{equation}
where ${\bf n}$ is the inward unit normal vector on $\Gamw$.
\end{problem}

By fixing the constants $\gam$, $\zeta_0$, $J$, and $S_0$, the term $\barrhoi$ from the equation
$\eqref{full EP rewritten}_5$ can be expressed as
\begin{equation}
    \label{barrhoi-expression-n}
      \barrhoi=\frac{J}{\zeta_0(\gam S_0J^{\gam-1})^{\frac{1}{\gam+1}}}.
    \end{equation}
    We call the system \eqref{full EP rewritten} with the expression of \eqref{barrhoi-expression-n} {\emph{the simplified steady Euler-Poisson system}} associated with $(\gam, \zeta_0, J, S_0)$. This expression indicates our choice of a background solution (in the sense of Definition \ref{definition of background solution-full EP}) from which we shall build a two-dimensional solution to \eqref{full EP rewritten} with a \underline{smooth} transonic transition across an interface of codimension 1.

Before stating our main theorem, we list the compatibility conditions required for $(S_{\rm en}, E_{\rm en}, w_{\rm en})$.
\begin{condition} For the fixed constant $m\ge 4$ from Problem \ref{problem-full EP}, we assume that the boundary data $(S_{\rm en}, E_{\rm en}, w_{\rm en})$ satisfy the following compatibility conditions:
\label{conditon:1}
\begin{itemize}
\item[(i)] For all $k=2i-1$ with $i\in \mathbb{N}$ and $k<m$ ,
\begin{equation*}
\frac{d^k S_{\rm en}}{dx_2^k}(x_2)=
                  \frac{d^k E_{\rm en}}{dx_2^k}(x_2)=0\quad\tx{at $|x_2|=1$}.
\end{equation*}

    \item[(ii)] For all $l=2j$ with $j\in \{0\}\cup \mathbb{N}$ and $l\le m$,
    \begin{equation*}
     \frac{d^l w_{\rm en}}{dx_2^l}(x_2)=0 \quad\tx{at $|x_2|=1$}.
    \end{equation*}
\end{itemize}
\end{condition}

\begin{theorem}
\label{theorem-smooth transonic-full EP}
Fix constants $(\gam, \zeta_0, J, S_0, E_0)$ satisfying
\begin{equation*}
\gam>1,\quad \zeta_0>1,\quad  S_0>0.
\end{equation*}
Suppose that $(u_0, E_0)\in \Tac$ with $E_0<0$, which is equivalent to $0<u_0<\us$.
And, let $(\bar{\rho}, {\bar{\bf u}}, \bar{\Phi})$ be the background smooth transonic solution to \eqref{full EP rewritten} associated with $(\gam, \zeta_0, J, S_0, E_0)$. In addition, suppose that the boundary data $(S_{\rm en}, E_{\rm en}, w_{\rm en})$ satisfy Condition \ref{conditon:1}. Then one can fix two constants $\bJ$ and $\ubJ$ depending only on $(\gam, \zeta_0, S_0)$, which satisfy
\begin{equation*}
  0<\bJ<1<\ubJ<\infty
\end{equation*}
so that whenever the background momentum density $J(=\bar{\rho}\bar u_1)$ satisfies
\begin{center}
either $0<J\le \bJ$ or $\ubJ\le J<\infty$,
\end{center}
there exists a constant $d\in(0,1)$ depending on $(\gam, \zeta_0, S_0,J)$ so that if the two constants $u_0$ and $L$ satisfy the condition
\begin{equation}
\label{almost sonic condition1 full EP}
  1-d\le \frac{u_0}{\us}<1 <\frac{\bar u_1(L)}{\us}\le 1+d,
\end{equation}
then $(\bar{\rho}, \bar{\bf u},\bar{\Phi})$ is structurally stable in the following sense: one can fix a constant $\bar{\sigma}>0$ sufficiently small depending only on $(\gam, \zeta_0, S_0, E_0, J, L)$ so that if the inequality
\begin{equation}
\label{smallness of bd}
  \mfrak P(S_{\rm en}, E_{\rm en}, w_{\rm en})\le \bar{\sigma}
\end{equation}
holds, then Problem \ref{problem-full EP} has a unique solution $(\rho, {\bf u}, S, \Phi)\in H^{m-1}(\Om_L)\times H^{m-1}(\Om_L;\R^2)\times H^{m}(\Om_L)\times H^{m}(\Om_L)$ that satisfies the estimate
    \begin{equation}
    \label{solution estimate full EP}
    \begin{split}
    &\|\rho-\bar{\rho}\|_{H^{m-1}(\Om_L)}+\|{\bf u}-\bar{\bf u}\|_{H^{m-1}(\Om_L)}+\|S-S_0\|_{H^{m}(\Om_L)}+\|\Phi-\bar{\Phi}\|_{H^{m}(\Om_L)}\\
    &\le
    C\mfrak P(S_{\rm en}, E_{\rm en}, w_{\rm en})
    \end{split}
    \end{equation}
    for some constant $C>0$ depending only on $(\gam, \zeta_0, S_0, E_0, J, L, m)$.
Moreover, there exists a function $\fsonic:[-1,1]\rightarrow (0, L)$ satisfying that
    \begin{equation}
    \label{sonic boundary is a graph pt}
  M\begin{cases}
  <1\quad&\mbox{for $x_1<\fsonic(x_2)$},\\
  =1\quad&\mbox{for $x_1=\fsonic(x_2)$},\\
  >1\quad&\mbox{for $x_1>\fsonic(x_2)$},
  \end{cases}
    %\begin{split}
    %|{\bf u}|<\sqrt{\gam S\rho^{\gam-1}}\quad&\mbox{for $x_1<\fsonic(x_2)$\quad(subsonic)},\\
    %|{\bf u}|=\sqrt{\gam S\rho^{\gam-1}}\quad&\mbox{for $x_1=\fsonic(x_2)$\quad(sonic)},\\
    %|{\bf u}|>\sqrt{\gam S\rho^{\gam-1}}\quad&\mbox{for $x_1>\fsonic(x_2)$\quad(supersonic)}.
     % \end{split}
    \end{equation}
where the {\emph{Mach number}} $M$ of the system \eqref{full EP rewritten} is defined by
    \begin{equation*}
      M:=\frac{|{\bf u}|}{\sqrt{\gam S\rho^{\gam-1}}}.
    \end{equation*}
  And, the function $\fsonic$ satisfies the estimate
    \begin{equation}
    \label{estimate of sonic boundary pt}
      \|\fsonic-\ls\|_{H^{m-2}((-1,1))}\le C\mfrak P(S_{\rm en}, E_{\rm en}, w_{\rm en})
    \end{equation}
    for some constant $C>0$ depending on $(\gam, \zeta_0, S_0, E_0, J, L)$.

\end{theorem}

\begin{definition}[Sonic interface]
For the function $\fsonic$ given in Theorem \ref{theorem-smooth transonic-full EP}, we refer to the curve
\begin{equation*}
\{x_1=\fsonic(x_1):-1\le x_2\le 1\}
\end{equation*}
as {\emph{the sonic interface}} of the flow governed by $(\rho, {\bf u}, S, \Phi)$.
\end{definition}

%\begin{remark}
%\label{remark about pt main theorem-1}
%It directly follows from a standard extension method of functions in Sobolev spaces and the generalized Sobolev inequality that $\rho$, ${\bf u}$ and $S$ are in $C^1$, and $\Phi$ is in $C^2$ up to the boundary. This implies that Theorem \ref{theorem-smooth transonic-full EP} yields a classical solution to Problem \ref{problem-full EP}.
%\end{remark}

%\begin{remark}
%\label{remark about pt main theorem-2}
%The property \eqref{sonic boundary is a graph pt} implies that the solution $(\rho, {\bf u}, S, \Phi)$ given in Theorem \ref{theorem-smooth transonic-full EP}  yields a transonic  $C^1$-transition across the sonic interface $x_1=\fsonic(x_2)$.

%\end{remark}

There is another example of a sonic interface with a significantly different feature. According to \cite[Theorem 4.1]{BCF}, the function $D{\bf u}$ for the velocity field ${\bf u}$, is discontinuous on the sonic boundaries occurring in the self-similar regular shock reflection or in the self-similar Prandtl-Meyer reflection of potential flows(see \cite{bae2013prandtl, BCF2, CF2, CF3}). The sonic interface across which $D{\bf u}$ is discontinuous is called a {\emph{weak discontinuity}}. From Theorem \ref{theorem-smooth transonic-full EP} in this paper, and \cite[Theorem 4.1]{BCF}, we are given with two types of sonic interfaces: (i) a regular interface which is given as a level set of a function defined in terms of a classical solution, and (ii) a weak discontinuity which is determined by only one side of a flow. This leads to the following question:
    \begin{center}
    \emph{What is a general criterion to determine the type of a sonic interface?
    %How do we identify the type of a sonic boundary without solving a nonlinear boundary value problem?'
    }\end{center}
    A general classification of sonic interfaces should be a fascinating subject to be investigated in the future.

\begin{remark}
\label{remark about pt main theorem-3}
In Theorem \ref{theorem-smooth transonic-full EP}, we require that the background momentum density $J(=\bar{\rho}\bar u_1)$ be smaller or larger, and that the background Mach number $\bar{M}$ be close to 1 (see \eqref{almost sonic condition1 full EP}) to establish the well-posedness of problem \ref{problem-full EP}. As we shall see later, this requirement is given so that we can establish a priori $H^1$-estimate for a solution to a \emph{linear system} consisting of \emph{a Keldysh equation} and a second order elliptic equation, which are weakly coupled together via lower order derivative terms. We must emphasize that we can not directly apply Lemma \ref{lemma G:wp of singular pert prob-main} in this case because the Keldysh equation in the system is coupled with another equation, and coefficients of the coupling terms are not negligible. This is the greatest challenge in proving Theorem \ref{theorem-smooth transonic-full EP}. To the best of our knowledge, this is the first paper to deal with a boundary value problem of such a PDE system.
\end{remark}

Suppose that $(\rho, {\bf u}, S, \Phi)\in H^{m-1}(\Om_L)\times H^{m-1}(\Om_L;\R^2)\times H^{m}(\Om_L)\times H^{m}(\Om_L)$ is a solution to Problem \ref{problem-full EP}. Since ${\bf u}$ is in $H^{m-1}(\Om_L)$, there exists a unique $H^{m+1}$-function $\phi:\ol{\Om_L}\rightarrow \R$ that solves the linear boundary value problem:
\begin{equation}
\label{bvp for phi}
\begin{cases}
  -\Delta\phi=\nabla\times{\bf u}\quad&\mbox{in $\Om_L$},\\
  \der_{x_1}\phi=0\quad&\mbox{on $\Gam_0$},\\
  \phi=0\quad&\mbox{on $\Gam_w\cup\Gam_L$}.
  \end{cases}
\end{equation}
Set
\begin{equation*}
\nabla^{\perp}:={\bf e}_{1}\der_{x_2}-{\bf e}_{2}\der_{x_1},
\end{equation*}
one has $\nabla\times (\nabla^{\perp}\phi)=-\Delta \phi$, and this yields that $\nabla\times({\bf u}-\nabla^{\perp}\phi)=0$ in $\Om_L$. So there exists a function $\vphi:\ol{\Om_L}\rightarrow \R$ such that
\begin{equation}
\label{H-decomposition}
  {\bf u}=\nabla\vphi+\nabla^{\perp}\phi\quad\tx{in $\ol{\Om_L}$}.
\end{equation}
This yields a {\emph{Helmholtz decomposition}} of a two-dimensional vector field ${\bf u}$. The representation \eqref{H-decomposition} indicates that $\vphi$ and $\phi$ are concerned with the compressibility and the vorticity of a flow, respectively, in the sense that $\nabla\cdot {\bf u}=\Delta \vphi$ and $\nabla\times {\bf u}=-\Delta \phi$. In \cite{BDXX}--\cite{BDX}, it is shown that if ${\bf u}\cdot{\bf e}_1=\vphi_{x_1}+\phi_{x_2}\neq 0$, then \eqref{full EP rewritten} can be rewritten in
terms of $(\vphi, \Phi, \phi, S)$ as follows:
\begin{align}
\label{new system with H-decomp2}
&
{\rm div}\left(\varrho(S, \Phi, \nabla\vphi, \nabla^{\perp}\phi)(\nabla\vphi+\nabla^{\perp}\phi)\right)=0,\\
\label{new system with H-decomp3}
&\Delta\Phi=\varrho(S, \Phi,\nabla\vphi, \nabla^{\perp}\phi)-\barrhoi,\\
\label{new system with H-decomp4}
&\Delta\phi=-\frac{ \varrho^{\gam-1}(S, \Phi,\nabla\vphi, \nabla^{\perp}\phi)S_{x_2}}{(\gam-1)(\vphi_{x_1}+\phi_{x_2})},\\
\label{new system with H-decomp5}
&\varrho(S, \Phi,\nabla\vphi, \nabla^{\perp}\phi) (\nabla\vphi+\nabla^{\perp}\phi)\cdot \nabla S=0
\end{align}
for
\begin{equation}
  \label{new system with H-decomp1}
\varrho(S, \Phi,\nabla\vphi, \nabla^{\perp}\phi):=
\left(\frac{\gam-1}{\gam S}\left(\Phi-\frac 12|\nabla\vphi+\nabla^{\perp}\phi|^2\right)\right)^{\frac{1}{\gam-1}}.
\end{equation}

So we can restate Problem \ref{problem-full EP} in terms of $(\vphi, \Phi, \phi, S)$ as follows:
\begin{problem}
\label{problem-HD}
 Fix $m\in \mathbb{N}$ with $m\ge 4$.
Given constants $\gam>1$, $\zeta_0>1$, $J>0$, $S_0>0$, and $(u_0, E_0)\in \Tac$ with $E_0<0$, fix a constant $L\in(0, l_{\rm max})$. Suppose that functions $S_{\rm en}:[-1,1]\rightarrow \R$, $E_{\rm en}:[-1,1]\rightarrow \R$ and $w_{\rm en}:[-1,1]\rightarrow \R$ are fixed so that $(S_{\rm en}, E_{\rm en}, w_{\rm en})$ is sufficiently close to $(S_0, E_0, 0)$ in the sense that $\mfrak P(S_{\rm en}, E_{\rm en}, w_{\rm en})$, given by \eqref{definition-perturbation of bd}, is small. Find a solution $(\vphi, \phi, \Phi, S)$ to the nonlinear system \eqref{new system with H-decomp2}--\eqref{new system with H-decomp5} in $\Om_L$ with satisfying the properties
\begin{equation*}
 (\vphi_{x_1}+\phi_{x_2})>0\quad\tx{and}\quad  \varrho(S, \Phi, \nabla\vphi, \nabla^{\perp}\phi)>0 \quad\tx{in $\ol{\Om_L}$},
\end{equation*}
and the boundary conditions:
  \begin{equation}
\label{BC for new system with H-decomp}
  \begin{split}
  \der_{x_2}\vphi=w_{\rm en},\quad  \der_{x_1}\phi=0,\quad S=S_{\rm en},\quad
  \der_{x_1}\Phi=E_{\rm en}\quad&\tx{on $\Gamen$},\\
  \der_{x_2}\vphi=0,\quad \phi=0,\quad \der_{x_2}\Phi=0\quad&\mbox{on $\Gamw$},\\
  \phi=0,\quad \Phi=\bar{\Phi}\quad&\mbox{on $\Gamex$}.
  \end{split}
\end{equation}

\end{problem}

\begin{definition}\label{definition of background solution-HD}
For the background solution $(\bar{\rho}, \bar{\bf u}, \bar{\Phi})$ associated with $(\gam, \zeta_0, J, S_0, E_0)$ in the sense of Definition \ref{definition of background solution-full EP}, define a function $\bar{\vphi}:[0, l_{\rm max}]\times [-1,1]\rightarrow \R$ by
\begin{equation*}
  \bar{\vphi}(x_1, x_2):=\int_{0}^{x_1}\bar u_1(t)\,dt.
\end{equation*}
Then, $(\vphi, \phi, \Phi, S)=(\bar{\vphi}, 0, \bar{\Phi}, S_0)$ solves Problem \ref{problem-HD} for $(S_{\rm en}, E_{\rm en}, w_{\rm en})=(S_0, E_0, 0)$. Later on, we call $(\bar{\vphi}, 0, \bar{\Phi}, S_0)$ the {\emph{background solution}} to the system \eqref{new system with H-decomp2}--\eqref{new system with H-decomp5} associated with $(\gam, \zeta_0, J, S_0, E_0)$.

\end{definition}

\begin{notation} Let $\|\cdot\|_X$ be a norm in a linear space $X$. Given a vector field ${\bf V}=(V_1,\cdots, V_n)$ with $V_j\in X$ for all $j=1,\cdots, n$, we define $\|{\bf V}\|_X$ by
\begin{equation*}
  \|{\bf V}\|_X:=\sum_{j=1}^n \|V_j\|_{X}.
\end{equation*}
\end{notation}

\begin{theorem}
\label{theorem-HD}  Fix $m\in \mathbb{N}$ with $m\ge 4$.
Given constants $(\gam, \zeta_0, J, S_0, E_0)$ satisfying
\begin{equation*}
\gam>1,\quad \zeta_0>1,\quad S_0>0,\quad E_0<0,
\end{equation*}
suppose that $(u_0, E_0)\in \Tac$, which is equivalent to $0<u_0<\us$.
Let $(\bar{\vphi}, 0, \bar{\Phi}, S_0)$ be the background solution to the system \eqref{new system with H-decomp2}--\eqref{new system with H-decomp5} associated with $(\gam, \zeta_0, J, S_0, E_0)$. In addition, suppose that the boundary data $(S_{\rm en}, E_{\rm en}, w_{\rm en})$ satisfy {\emph{Condition \ref{conditon:1}}}.
Then one can fix two constants $\bJ$ and $\ubJ$ depending only on $(\gam, \zeta_0, S_0)$ with satisfying that
\begin{equation*}
  0<\bJ<1<\ubJ<\infty
\end{equation*}
so that whenever the background momentum density $J(=\bar{\rho}\bar u_1)$ satisfies
\begin{center}
either $0<J\le \bJ$ or $\ubJ\le J<\infty$,
\end{center}
there exists a constant $d\in(0,1)$ depending on $(\gam, \zeta_0, S_0,J)$ so that if
$(E_0, L)$ are fixed to satisfy the condition \eqref{almost sonic condition1 full EP}, then $(\bar{\vphi}, 0, \bar{\Phi}, S_0)$ is structurally stable in the following sense: one can fix a constant $\bar{\sigma}>0$ sufficiently small depending only on $(\gam, \zeta_0, S_0, E_0, J, L, m)$ so that if the inequality \eqref{smallness of bd} holds, then Problem \ref{problem-HD} has a unique solution $(\vphi, \phi, \Phi, S )$ that satisfies the estimate
    \begin{equation}
    \label{solution estimate HD}
    \begin{split}
    \|(\vphi, \Phi, S)-(\bar{\vphi},\bar{\Phi}, S_0)\|_{H^m(\Om_L)}+\|\phi\|_{H^{m+1}(\Om_L)}
    \le
    C\mfrak P(S_{\rm en}, E_{\rm en}, w_{\rm en})
    \end{split}
    \end{equation}
    for some constant $C>0$ depending on $(\gam, \zeta_0, S_0, E_0, J, L, m)$.
Furthermore, there exists a function $\fsonic:[-1,1]\rightarrow (0, L)$ satisfying the property \eqref{sonic boundary is a graph pt} with ${\bf u}=\nabla\vphi+\nabla^{\perp}\phi$ and $\rho=\varrho(S, \Phi, \nabla\vphi, \nabla\phi)$. Furthermore, the function $\fsonic$ satisfies the estimate \eqref{estimate of sonic boundary pt}
    for some constant $C>0$ depending on $(\gam, \zeta_0, S_0, E_0, J, L, m)$.

\end{theorem}

Similarly to the works in \cite{BDXX, BDX3}, Theorem \ref{theorem-smooth transonic-full EP} can easily follow from Theorem \ref{theorem-HD}.
\begin{proof}
[Proof of Theorem \ref{theorem-smooth transonic-full EP}]
If $(\vphi, \phi, \Phi, S)$ solves Problem \ref{problem-HD} and satisfies the estimate \eqref{solution estimate HD}, then $(\rho, {\bf u}, S, \Phi)$ with $\rho=\varrho(S, \Phi, \nabla\vphi, \nabla^{\perp}\phi)$ and ${\bf u}=\nabla\vphi+\nabla^{\perp}\phi$ solves Problem \ref{problem-full EP} and satisfies the estimate \eqref{solution estimate full EP}.
\medskip

Suppose that $(\rho, {\bf u}, \Phi)$ solves Problem \ref{problem-full EP}, and satisfies the estimate \eqref{solution estimate full EP}. First of all, we obtain $\phi$ as a solution to the boundary value problem:
\begin{equation*}
  \begin{cases}
  -\Delta \phi=\nabla\times {\bf u}\quad&\mbox{in $\Om_L$},\\
  \der_{x_1}\phi=0\quad&\mbox{on $\Gamen$},\\
  \phi=0\quad&\mbox{on $\Gam_w\cup\Gamex$}.
  \end{cases}
\end{equation*}
Next, we find $\vphi$ to satisfy
\begin{equation}
\label{recovery of vphi}
\begin{cases}
\nabla\vphi={\bf u}-\nabla^{\perp}\phi\quad&\mbox{in $\Om_L$},\\
\der_{x_2}\vphi=w_{\rm en}\quad&\mbox{on $\Gamen$}.
\end{cases}
\end{equation}
Define
\begin{equation*}
  \vphi(x_1, x_2):=\int_{0}^{x_1}({\bf u}\cdot {\bf e}_1-\der_{x_2}\phi)(t,x_2)\,dt+g(x_2)\quad\tx{in $\Om_L$}
\end{equation*}
for a function $g:[-1,1]\rightarrow \R$ to be determined. By a straightforward computation using the equation $-\Delta\phi=\nabla \times {\bf u}$, we get
\begin{equation*}
  \der_{x_2}\vphi=(\der_{x_1}\phi+u_2)(x_1, x_2)-u_2(0, x_2)+g'(x_2).
\end{equation*}
The function $\vphi$ satisfies \eqref{recovery of vphi} if we choose the function $g$ to be
\begin{equation*}
  g(x_2)=\int_0^{x_2} w_{\rm en}(s)\,ds.
\end{equation*}
Clearly, ${\bf V}:=(\vphi, \phi, \Phi, S)$ solves Problem \ref{problem-HD}. Furthermore, it follows from the estimate \eqref{solution estimate full EP} given in Theorem \ref{theorem-smooth transonic-full EP} that ${\bf V}$ satisfies
\begin{equation}
\label{estimate from full to HD}
  \|(\vphi, \Phi, S)-(\bar{\vphi}, \bar{\Phi}, S_0)\|_{H^m(\Om_L)}+ \|\phi\|_{H^{m+1}(\Om)}\le C_*\mfrak P(S_{\rm en}, E_{\rm en}, \om_{\rm en})
\end{equation}
for some constant $C_*>0$ depending on $(\gam, \zeta_0, S_0, E_0, J, L, m)$.

 Suppose that Problem \ref{problem-full EP} has two solutions ${\bf U}^{(i)}=(\rho^{(i)}, {\bf u}^{(i)}, S^{(i)}, \Phi^{(i)})\,(i=1,2)$ that satisfy the estimate \eqref{solution estimate full EP}. For each $j=1,2$, let ${\bf{V}}^{(j)}=(\vphi^{(j)}, \phi^{(j)}, \Phi^{(j)}, S^{(j)})$ be given by the procedure described in the above. Then it follows from Theorem \ref{theorem-HD} that ${\bf V}^{(1)}={\bf V}^{(2)}$ in $\Om_L$, which yields that ${\bf U}^{(1)}={\bf U}^{(2)}$ in $\Om_L$.
\end{proof}

The rest of the paper is devoted to prove Theorem \ref{theorem-HD}.

%\subsection{Preparations to prove Theorem \ref{theorem-HD}}

\section{Proof of Theorem \ref{theorem-HD}}
\label{subsection:framework}

\subsection{Linearized boundary value problem for iteration }

Fix constants $\gam>1$, $\zeta_0>1$, $S_0>0$ and $E_0<0$. Let $u_0\in(0, \us)$ be fixed  such that $(u_0, E_0)\in \Tac$, and fix a constant $L\in(0, l_{\rm max})$. Let $(\bar{\vphi}, 0, {\bar{\Phi}}, S_0)$ be the background solution to the system \eqref{new system with H-decomp2}--\eqref{new system with H-decomp5} associated with $(\gam, \zeta_0, J, S_0, E_0)$ in the sense of Definition \ref{definition of background solution-HD}.

\begin{definition}
\label{definition:coefficints-iter}
(1) For $z\in \R$ and ${\bf p}, {\bf q}\in \R^2$, denote
\begin{equation}
\label{definition of v}
  {\bf v}(\rx, {\bf p}, {\bf q} ):=\nabla\bar{\vphi}(\rx)+{\bf p}+{\bf q}.\end{equation}
For ${\bf v}={\bf v}(\rx, {\bf p}, {\bf q})$, define
\begin{equation}
\label{definition of pre coeff}
  A_{ij}(\rx, z, {\bf p}, {\bf q} ):=(\gam-1)
      \left(\bar{\Phi}(\rx)+z-\frac 12|{\bf v}|^2\right)\delta_{ij}-
       {\bf v}\cdot {\bf e}_i {\bf v}\cdot {\bf e}_j.
\end{equation}
For $s\in\R$ satisfying $|s|\le \frac{S_0}{2}$, define
\begin{equation}
        \label{definition of rho tilde}
          \til{\rho}(\rx, s, z, {\bf p}, {\bf q}):=\left(\frac{\gam-1}{\gam(S_0+s)}\left(\bar{\Phi}(\rx)+z-\frac 12|{\bf v(\rx, {\bf p}, {\bf q})}|^2\right)\right)^{\frac{1}{\gam-1}}.
        \end{equation}
\quad\\
(2) For $\rx=(x_1, x_2)\in \ol{\Om_L}$, ${\bf p}=(p_1,p_2),\, {\bf q}=(q_1,q_2)$, regard ${\bf v}={\bf v}(\rx, {\bf p}, {\bf q})$ as a 2 by 1 matrix. For ${\bf r}\in \R^2$ and $\mathbb{M}=(m_{ij})_{i,j=1}^2\in \R^{2\times 2}$, define
\begin{equation}
\begin{split}
      Q_1(\rx, {\bf p}, {\bf q},  {\bf r})&:=\frac{\gam+1}{2}\bar u_1'(x_1)\left(p_1+q_1\right)^2
-{\bf r}\cdot({\bf p}+{\bf q}),\\
      R_1(\rx, {\bf p}, {\bf q}, \mathbb{M})&:={\bf v}^{T}{\mathbb{M}}{\bf v}-\left(\bar E(x_1)-(\gam+1)\bar u_1(x_1)\bar u_1'(x_1)\right)q_1.
      \end{split}
\end{equation}
\end{definition}

Let us set
\begin{equation*}
  \umax:=\bar{u}_1(\lmax).
\end{equation*}
By Lemma \ref{lemma-1d-full EP}, it holds that $\umax>\bar u_1(x_1)$ for $x_1\in[0, \lmax)$. Since $(\umax,0)\in \Tac$, it holds that $H(\umax)=0$  thus we have $u_0<\umax<\infty$, for the function $H$ defined by \eqref{definition of H}. This implies that $ 0<\bar u_1(x_1)<\infty$ for all $x_1\in [0,L]$. Therefore, we have
\begin{equation*}
A_{22}(\rx, 0, {\bf 0}, {\bf 0})=(\gam-1)\left(\bar{\Phi}-\frac 12|\nabla\bar{\vphi}|^2\right)=\frac{\gam S_0J^{\gam-1}}{\bar u_1^{\gam-1}(x_1)}>0\quad\tx{in $\ol{\Om_L}$.}
\end{equation*}
Since $\bar u_1$ is smooth on $[0,L]$, and $A_{22}(\rx, z, {\bf p}, {\bf q})$ is smooth for $(\rx,z, {\bf p}, {\bf q})\in \ol{\Om_L}\times \R\times \R^2\times \R^2$, one can fix a small constant $d_0>0$ depending only on $(\gam, \zeta_0, J, S_0, E_0)$ such that if
\begin{equation}
\label{small zpq condition}
  \max\{|z|, |{\bf p}|, |{\bf q}|\}\le d_0\quad\tx{with $|{\bf p}|:=\sqrt{p_1^2+p_2^2}$,\quad $|{\bf q}|:=\sqrt{q_1^2+q_2^2}$},
\end{equation}
then it holds that
\begin{equation}
\label{strict positivity of A22}
  A_{22}(\rx, z, {\bf p}, {\bf q})\ge \frac{\gam S_0J^{\gam-1}}{2\umax^{\gam-1}}\quad\tx{in $\ol{\Om_L}$}.
\end{equation}

\begin{definition} Let us set
\begin{equation*}
  \mcl{R}_{d_0}:=\{(z, {\bf p}, {\bf q})\in \R\times \R^2\times \R^2: {\tx{the condition \eqref{small zpq condition} holds.}}\}.
\end{equation*}
Let $(\rx, z, {\bf p}, {\bf q})\in \ol{\Om_L}\times \mcl{R}_{d_0}$ with $\rx=(x_1,x_2)$. We define coefficients $\til a_{ij}(i,j=1,2)$, $\til a$, $\til b_1$, and $\til b_0$ by
    \begin{equation}
\label{coefficients of L_1}
\begin{split}
&\til a_{ij}(\rx, z,{\bf p}, {\bf q}):=\frac{A_{ij}(\rx, z,{\bf p}, {\bf q})}{A_{22}(\rx, z,{\bf p}, {\bf q})},\\
&\til a(\rx, z,{\bf p}, {\bf q}):=\frac{\bar E(x_1)-(\gam+1)\bar u_1'(x_1)\bar u_1(x_1)}{A_{22}(\rx, z,{\bf p}, {\bf q})},\\
&\til b_1(\rx, z,{\bf p}, {\bf q}):=\frac{\bar u_1(x_1)}{A_{22}(\rx, z,{\bf p}, {\bf q})},\quad \til b_0(\rx, z,{\bf p}, {\bf q}):=\frac{(\gam-1)\bar u_1'(x_1)}{A_{22}(\rx, z,{\bf p}, {\bf q})}.
\end{split}
\end{equation}
\end{definition}

\medskip

Suppose that $(\vphi, \phi, \Phi, S)$ solves Problem \ref{problem-HD}.
And, let us set $(\psi, \Psi, T)$ as
\begin{equation*}
  (\psi, \Psi, T):=(\vphi-\bar{\vphi}, \Phi-\bar{\Phi}, S-S_0).
\end{equation*}
\begin{definition}
\label{definition:coefficients-nonlinear}
(1) Suppose that $(\psi, \Psi, \phi)$ satisfies
\begin{equation}
\label{smallness condition-1}
\max_{\ol{\Om_L}} \{|\Psi|, |D\psi|, |D\phi|\}\le d_0,
\end{equation}
and define
\begin{equation*}
  \begin{split}
  a_{ij}^{(\psi, \phi, \Psi)}(\rx)&:=\til a_{ij}(\rx,\Psi(\rx), \nabla\psi(\rx), \nabla^{\perp}\phi(\rx)),\\
   a^{(\psi, \phi, \Psi)}(\rx)&:=\til a(\rx,\Psi(\rx), \nabla\psi(\rx), \nabla^{\perp}\phi(\rx)),\\
   b_1^{(\psi, \phi, \Psi)}(\rx)&:=\til b_1(\rx,\Psi(\rx), \nabla\psi(\rx), \nabla^{\perp}\phi(\rx)),\\
  b_0^{(\psi, \phi, \Psi)}(\rx)&:=\til b_0(\rx,\Psi(\rx), \nabla\psi(\rx), \nabla^{\perp}\phi(\rx)),\\
 {f}_1^{(\psi, \phi, \Psi)}(\rx)&:=\frac{Q_1(\rx, \nabla\psi(\rx), \nabla^{\perp}\phi(\rx), \nabla\Psi(\rx))+R_1(\rx, \nabla\psi(\rx), \nabla^{\perp}\phi(\rx), D(\nabla^{\perp}\phi)(\rx))}{ A_{22}(\rx, \Psi(\rx), \nabla\psi(\rx), \nabla^{\perp}\phi(\rx))},
  \end{split}
\end{equation*}
for
\begin{equation*}
D(\nabla^{\perp}\phi):=\begin{pmatrix}\der_1\\ \der_2\end{pmatrix}\begin{pmatrix}\der_2&-\der_1\end{pmatrix}\phi=
\begin{pmatrix}\der_{12}\phi&-\der_{11}\phi\\
\der_{22}\phi &-\der_{21}\phi\end{pmatrix}.
\end{equation*}
\quad\\
(2) In addition to \eqref{smallness condition-1}, suppose that $T$ satisfies
\begin{equation}
\label{smallness condition-2}
  \max_{\Om_L}|T|\le \frac{S_0}{2}.
\end{equation}
Let us define
\begin{equation*}
\begin{split}
  &c_0(\rx):=\der_z\til{\rho}(\rx, 0,0, {\bf 0}, {\bf 0}),\quad c_1(\rx):=\der_{\bf p}\til{\rho}(\rx, 0,0, {\bf 0}, {\bf 0})\cdot{\bf e}_1,\\
 &{f}_2^{(T,\psi, \phi, \Psi)}(\rx):= \til{\rho}(\rx, T(\rx), \Psi(\rx), \nabla\psi(\rx), \nabla^{\perp}\phi(\rx))-\til{\rho}(\rx,0, 0, {\bf 0}, {\bf 0})-\mfrak c_0\Psi(\rx)
  -\mfrak c_1 \der_1\psi(\rx).
 \end{split}
\end{equation*}
A direct computation yields
\begin{equation*}
c_0(\rx)=\frac{1}{\gam S_0\bar{\rho}^{\gam-2}(\rx)},\quad
c_1(\rx)=\frac{-\bar u_1(\rx)}{\gam S_0\bar{\rho}^{\gam-2}(\rx)}.
\end{equation*}

\quad\\
(3) Under the condition
\begin{equation}
\label{smallness condition-3}
 \min_{\ol{\Om_L}} {\bf v}(\rx, \nabla\psi(\rx), \nabla^{\perp}\phi(\rx))\cdot {\bf e}_1\ge \frac{u_0}{2},
\end{equation}
define
\begin{equation*}
      {f}_3^{(T,\psi, \phi, \Psi)}(\rx):=\frac{\left(\bar{\Phi}+\Psi(\rx)-\frac 12|{\bf v}(\rx, \nabla\psi(\rx), \nabla^{\perp}\phi(\rx))|^2\right)\der_2T(\rx)}{\gam(S_0+T(\rx)){\bf v}(\rx, \nabla\psi, \nabla^{\perp}\phi)\cdot {\bf e}_1}.
    \end{equation*}

\quad\\
(4) Finally, define a pseudo momentum density field ${\bf m}^{(\psi, \phi, \Psi)}$ by
    \begin{equation}
    \label{definition: vector field m}
      {\bf m}^{(\psi, \phi, \Psi)}(\rx):=\left(\bar{\Phi}(\rx)+\Psi(\rx)-\frac 12|{\bf v(\rx, \nabla\psi(\rx), \nabla^{\perp}\phi(\rx))}|^2\right)^{\frac{1}{\gam-1}}\bf v(\rx, \nabla\psi(\rx), \nabla^{\perp}\phi(\rx)).
    \end{equation}
\end{definition}

If all the conditions of \eqref{smallness condition-1}--\eqref{smallness condition-3} are satisfied, the nonlinear boundary value problem consisting of \eqref{new system with H-decomp2}--\eqref{new system with H-decomp5} and the boundary conditions stated in \eqref{BC for new system with H-decomp} is equivalent to the following problem for $(\psi, \phi, \Psi, T)$:
\begin{itemize}
\item[(i)] {\emph{Equations for $(\psi, \phi, \Psi, T)$ in $\Om_L$}}
\begin{align}
\label{equation for psi}
 \sum_{i,j=1}^2 a_{ij}^{(\psi, \phi, \Psi)}\der_{ij}\psi+  a^{(\psi, \phi, \Psi)}\der_1\psi+
    b_1^{(\psi, \phi, \Psi)}\der_1\Psi
    +b_0^{(\psi, \phi, \Psi)}\Psi={f}_1^{(\psi, \phi, \Psi)},&\\
\label{equation for Psi}
\Delta \Psi-c_0\Psi-c_1\der_1\psi={f}_2^{(T,\psi, \phi, \Psi)},&\\
\label{equation for phi}
  -\Delta \phi= {f}_3^{(T,\psi, \phi, \Psi)},&\\
    \label{equation for T}
 {\bf m}^{(\psi, \phi, \Psi)}\cdot \nabla T=0.
    \end{align}

\item[(ii)]{\emph{Boundary conditions for $\psi$}}
\begin{equation}\label{BC for psi}
  \psi(0, x_2)=\int_{-1}^{x_2}w_{\rm en}(t)\,dt\quad\tx{on $\Gamen$},\quad \der_2 \psi=0\quad\tx{on $\Gamw$}.
\end{equation}

\item[(iii)]{\emph{Boundary conditions for $\phi$}}
\begin{equation}\label{BC for phi}
  \der_1\phi=0\quad\tx{on $\Gamen$},\quad \phi=0\quad\tx{on $\der\Om_L\setminus \Gamen$}.
\end{equation}

\item[(iv)]{\emph{Boundary conditions for $\Psi$}}
\begin{equation}\label{BC for Psi}
  \der_1\Psi=E_{\rm en}-E_0\quad\tx{on $\Gamen$},\quad \der_2\Psi=0\quad\tx{on $\Gamw$},\quad \Psi=0\quad\tx{on $\Gamex$}.
\end{equation}

\item[(v)]{\emph{The boundary condition for $T$}}
\begin{equation}\label{BC for T}
T=S_{\rm en}-S_0\quad\tx{on $\Gamen$}.
\end{equation}

\end{itemize}

\begin{definition}
\label{definition: iteration sets}
For fixed constants $m\in \mathbb{N}$ and $r>0$, define
\begin{equation*}
  \begin{split}
&\iterV^{m+1}_{r}:=\left\{\eta\in H^{m+1}(\Om_L)\;\middle\vert\;
\begin{split}&\|\eta\|_{H^{m+1}(\Om_L)}\le r,\\
& \left(\frac{\der}{\der x_2}\right)^{k}\eta=0\quad\tx{on $\Gamw$ for $k=2(i-1)$, $i\in \mathbb{N}$ with $k<m+1$ }\end{split}\right\},\\
&\iterP^m_r:=\left\{\xi\in H^m(\Om_L)\;\middle\vert\;
\begin{split}&\|\xi\|_{H^m(\Om_L)}\le r,\\
&\left(\frac{\der}{\der x_2}\right)
^k\xi=0\quad\tx{on $\Gamw$ for $k=2i-1$, $i\in \mathbb{N}$ with $k<m$}\end{split}\right\}.
\end{split}
\end{equation*}

\end{definition}

\begin{definition}
\label{definition:approx coeff and fs} Fix $m\ge 4$. For constants $r_1$, $r_2$, $r_3$ to be specified later, fix $\til{T}\in \iterP^m_{r_1}$ and $P=(\tphi, \tpsi, \tPsi)\in \iterV^{m+1}_{r_2}\times \iterP^m_{r_3}\times \iterP^m_{r_3}$. And, let $a_{ij}^P$, $a^P$, $b_1^P$, $b_0^P$, $c_0$ and $c_1$ be given by Definition \ref{definition:coefficients-nonlinear}.
\quad\\
(1) Define a bilinear differential operator $\mfrak{L}_1^P(\cdot,\cdot)$ associated with $P$ by
\begin{equation*}
  \mfrak{L}_1^P(\psi, \Psi):=\sum_{i,j=1}^2 a_{ij}^P\der_{ij}\psi+a^P\der_1\psi+b_1^P\der_1\Psi+b_0^P\Psi.
\end{equation*}
%for the coefficients $a_{ij}^P$, $a^P$, $b_1^P$ and $b_0^P$ given by Definition \ref{definition:coefficients-nonlinear}.
\quad\\
(2) For the coefficients ${c}_0$ and ${c}_1$ given in Definition \ref{definition:coefficients-nonlinear}, define another bilinear differential operator $\mfrak{L}_2(\cdot,\cdot)$ by
\begin{equation*}
      \mfrak{L}_2(\psi, \Psi):=\Delta \Psi-{c}_0\Psi-{c}_1\der_1\psi.
    \end{equation*}
   % Note that $\mfrak{L}_2$ is fixed independently of $\til{T}\in \iterT{(r_1)}$ and $P=(\tphi, \tpsi, \tPsi)\in \iterV{(r_2)}\times \iterP{(r_3)}$.

\end{definition}

\begin{problem}
\label{LBVP1 for iteration}
For fixed $\til{T}\in \iterP^m_{r_1}$ and $P=(\tphi, \tpsi, \tPsi)\in \iterV^{m+1}_{r_2}\times \iterP^m_{r_3}\times \iterP^m_{r_3}$,
solve
\begin{equation}
\label{lbvp for phi}
  \begin{cases}
  -\Delta\phi=f_0^{(\tilT, P)}\quad&\tx{in $\Om_L$},\\
  \der_1\phi=0\quad&\tx{on $\Gamen$},\\
  \phi=0\quad&\tx{on $\Gamw\cup\Gamex$}.
  \end{cases}
\end{equation}

And, solve
    \begin{equation}
    \label{lbvp main}
      \begin{split}
      \begin{cases}
      \mfrak{L}_1^P(\psi, \Psi)=f_1^P\quad&\tx{in $\Om_L$},\\
      \mfrak{L}_2(\psi, \Psi)=f_2^{(\tilT, P)}\quad&\tx{in $\Om_L$},
      \end{cases}\\
      \psi(0, x_2)=\int_{-1}^{x_2}w_{\rm en}(t)\,dt,\quad \der_1\Psi=E_{\rm en}-E_0\quad&\tx{on $\Gamen$},\\
      \der_2\psi=0,\quad \der_2\Psi=0\quad&\tx{on $\Gamw$},\\
      \Psi=0\quad&\tx{on $\Gamex$}.
      \end{split}
    \end{equation}

\end{problem}
Later, we shall fix constants $r_1$, $r_2$ and $r_3$ so that Problem \ref{LBVP1 for iteration} is well-posed for any $\til{T}\in \iterP^m_{r_1}$ and $P=(\tphi, \tpsi, \tPsi)\in \iterV^{m+1}_{r_2}\times \iterP^m_{r_3}\times \iterP^m_{r_3}$.
\medskip

%\begin{remark}
%If one solves a single mixed type PDE with a degeneracy, it may be possible to directly employ an idea from \cite[Chapter 1]{KZ}. But we should point out that the method developed in \cite{KZ} is applicable only if several technical conditions hold. To our surprise, all the technical conditions described in \cite[Chapter 1]{KZ} are satisfied by the equation \eqref{equation for psi} or its variations, which are derived from \eqref{new system with H-decomp2}(see Lemmas \ref{lemma-coefficient at bg} and \ref{corollary-coeff extension}), when the background flow is accelerating (see Lemma \ref{lemma-1d-full EP}). But, there are more difficulties. The linear system in \eqref{lbvp main} consists of two PDEs of different types, and they are weakly coupled by lower order derivative terms. So, in order to establish an a priori $H^1$-estimate of weak solutions to the system \eqref{lbvp main}, it was inevitable for us to add an assumption on the background solutions. Clearly, the condition \eqref{almost sonic condition1 full EP} holds if $L$ is sufficiently small. But the smallness of $L$ is not a necessary condition to guarantee the almost sonic condition. In Appendix \ref{appendix:nozzle length}, we show examples of the parameters $(\gam, J)$ for which the nozzle length $L$ can be large.
%\medskip
%\end{remark}

Given $\til{T}\in \iterP^m_{r_1}$ and $P=(\tphi, \tpsi, \tPsi)\in \iterV^{m+1}_{r_2}\times \iterP^m_{r_3}\times \iterP^m_{r_3}$, let $(\phi, \psi, \Psi)$ be a solution to Problem \ref{LBVP1 for iteration}, and define a vector field ${\bf m}(\Psi, \nabla\psi, \nabla^{\perp}\phi)$ by \eqref{definition: vector field m}. It is clear that if $(\phi, \psi, \Phi)=P$ holds in $\Om_L$, then the vector field ${\bf m}$ satisfies the equation
\begin{equation}
\label{div-free}
  \nabla\cdot {\bf m}(\Psi, \nabla\psi, \nabla^{\perp}\phi)=0\quad\tx{in $\Om_L$}.
\end{equation}
\begin{definition}%[Approximated momentum density field]
\label{definition: momentum density field} Fix $\tilT \in \iterP^m_{r_1}$. Assume that the constants $r_1$, $r_2$, and $r_3$ are appropriately fixed so that the problem \ref{LBVP1 for iteration} is well defined and admits a unique solution $(\phi, \psi, \Phi)$ for each $\til{T}\in \iterP^m_{r_1}$ and $P=(\tphi, \tpsi, \tPsi)\in \iterV^{m+1}_{r_2}\times \iterP^m_{r_3}\times \iterP^m_{r_3}$. In addition, suppose that there exists a unique element $P\in \iterV^{m+1}_{r_2}\times \iterP^m_{r_3}\times \iterP^m_{r_3}$ that satisfies the equation $P=(\phi, \psi, \Phi)$ in $\Om_L$. For such an element $P$,  we define the vector field ${\bf m}(\Psi, \nabla\psi, \nabla^{\perp}\phi)$ as {\emph{the approximated momentum density field associated with $\tilT$}}.
\end{definition}

\begin{problem}
\label{problem-transport equation}
For a fixed $\tilT\in \iterP^m_{r_1}$, let ${\bf m}(\Psi, \nabla\psi, \nabla^{\perp}\phi)$ be the {\emph{approximated momentum density field}} associated with $\tilT$ in the sense of Definition \ref{definition: momentum density field}. Solve the following boundary value problem for $T$:
\begin{equation*}
  \begin{split}
  {\bf m}(\Psi, \nabla\psi, \nabla^{\perp}\phi)\cdot \nabla T=0\quad&\tx{in $\Om_L$},\\
  T=S_{\rm en}-S_0\quad&\tx{on $\Gamen$}.
  \end{split}
\end{equation*}
\end{problem}

\begin{remark}
\label{remark-iteration}
\iffalse
%We add two remarks on Problems \ref{problem-HD}, \ref{LBVP1 for iteration} and  \ref{problem-transport equation}.
\begin{itemize}
\item[(i)] Clearly, Problem \ref{problem-transport equation} is well defined if we prove that Problem \ref{LBVP1 for iteration} is well-posed for each $\tilT\in \iterP^m_{r_1}$ and $P\in \iterV^{m+1}_{r_2}\times \iterP^m_{r_3}\times \iterP^m_{r_3}$, and show that the approximated momentum density field is well defined for each $\tilT\in \iterP^m_{r_1}$.

\item[(ii)] Suppose that $T$ is a solution to Problem \ref{problem-transport equation}, and that it satisfies $T=\tilT$ in $\Om_L$. Let us define $(\vphi, \Phi, S)$ by
    \begin{equation*}
      (\vphi, \Phi, S):=(\bar{\vphi}, \bar{\Phi}, S_0)+(\psi, \Psi, T)\quad\tx{in $\Om_L$}.
    \end{equation*}
    Then $(\vphi, \Phi, \phi, S)$ solves Problem \ref{problem-HD}.
\end{itemize}
\fi
Suppose that $T$ is a solution to Problem \ref{problem-transport equation}, and that it satisfies $T=\tilT$ in $\Om_L$. Let us define $(\vphi, \Phi, S)$ by
    \begin{equation*}
      (\vphi, \Phi, S):=(\bar{\vphi}, \bar{\Phi}, S_0)+(\psi, \Psi, T)\quad\tx{in $\Om_L$}.
    \end{equation*}
    Then $(\vphi, \Phi, \phi, S)$ solves Problem \ref{problem-HD}.
\end{remark}

\subsection{Approximated sonic interfaces}
\label{subsection: main issues}
\iffalse
As mentioned above, the biggest challenge in proving Theorem \ref{theorem-HD} is to establish the well-posedness of the boundary value problem \eqref{lbvp main}. For further discussion,
\fi
In order to understand the greatest difficulty in solving \eqref{lbvp main}, we first need to understand various properties of the differential operator $\mfrak{L}_1^{P}$.

\begin{lemma}
\label{lemma on L_1}
Given constants $(\gam, \zeta_0, J, S_0, E_0)$, let $(\bar{\vphi}, \bar{\Phi})$ be the associated background solution to the system \eqref{new system with H-decomp1}--\eqref{new system with H-decomp5} in the sense of Definition \ref{definition of background solution-HD}. For any $L\in (0, l_{\rm max}]$, one can fix a constant $\bar{\delta}>0$ sufficiently small depending only on $(\gam, \zeta_0, J, S_0, E_0, L)$ so that if we fix the constants $r_1$, $r_2$, and $r_3$ to satisfy the inequality
\begin{equation}
\label{condition:r}
\max\{r_1,r_2, r_3\}\le 2\bar{\delta},
\end{equation}
then, for each $\til{T}\in \iterP^m_{r_1}$ and $P=(\tphi, \tpsi, \tPsi)\in \iterV^{m+1}_{r_2}\times \iterP^m_{r_3}\times \iterP^m_{r_3}$, the coefficients $(a_{ij}^P, a^P, b_1^P, b_0^P)$ with $i,j=1,2$, and the functions $(f_0^{(\tilT, P)}, f_1^P, f_2^{(\tilT, P)})$ are well defined by Definition \ref{definition:approx coeff and fs}. Furthermore, they satisfy the following properties:

\begin{itemize}
\item[(a)] $\displaystyle{a_{ij}^P, a^P, b_1^P, b_0^P, f_0^{(\tilT, P)}, f_1^P, f_2^{(\tilT, P)}\in H^{m-1}(\Om_L)}$.
    \medskip
\item[(b)] $\displaystyle{a_{22}^P= 1}$ and $\displaystyle{a_{12}^P=a_{21}^P}$ in $\Om_L$.
    \medskip

\item[(c)]
For $k=2j-1$, $j\in \mathbb{N}$ with $k<m-1$,
    \begin{equation*}
      \left(\frac{\der}{\der{x_2}}\right)^k (a_{11}^P, a^P, b_1^P, b_0^P, f_1^p, f_2^{(\tilT,P)})=0\quad\tx{on $\Gamw$}.
    \end{equation*}
For $k=2(j-1)$, $j\in \mathbb{N}$ with $k<m-1$,
    \begin{equation*}
     \left(\frac{\der}{\der{x_2}}\right)^k(a_{12}^P, f_0^{(\tilT,P)})=0\quad\tx{on $\Gamw$}.
    \end{equation*}

\item[(d)] For $P_0=(0,0,0),$ let us set $(\bar a_{11}, \bar a, \bar b_1, \bar b_0)$ as
    \begin{equation}
    \label{coefficient at bg}
      (\bar a_{11}, \bar a, \bar b_1, \bar b_0):=(a_{11}^{P_0}, a^{P_0}, b_1^{P_0}, b_0^{P_0}).
    \end{equation}
    Then it holds that
    \begin{equation}
\label{estimate-coefficient-difference}
\begin{split}
  &\|(a^P_{11}, a^P, b^P_1, b^P_0)-(\bar a_{11}, \bar a, \bar b_1, \bar b_0)\|_{H^{m-1}(\Om_L)}+\|a^P_{12}\|_{H^{m-1}(\Om_L)}\\
  &\le C \|(\tphi, \tpsi, \tPsi)\|_{H^{m}(\Om_L)}
\end{split}
\end{equation}
for a constant $C>0$ depending only on $(\gam, \zeta_0, J, S_, E_0)$.
\medskip

\item[(e)] There exists a constant $C>0$ depending only on $(\gam, \zeta_0, J, S_0, E_0)$ to satisfy the following estimates:
    \begin{align*}
    &\|f_0^{(\tilT, P)}\|_{H^{m-1}(\Om_L)}\le C\|\tilT\|_{H^{m}(\Om_L)},\\
    &\|f_1^{P}\|_{H^{m-1}(\Om_L)}\le C\left(\left(1+\|\tPsi\|_{H^{m}(\Om_L)}\right)\|\tphi\|_{H^{m+1}(\Om_L)}+
        \|(\tpsi, \tPsi)\|^2_{H^{m}(\Om_L)}
        \right),\\
        &\|f_2^{(\tilT, P)}\|_{H^{m-1}(\Om_L)}\le C\left(\|\tilT\|_{H^{m-1}(\Om_L)}
        +\|\tphi\|_{H^{m}(\Om_L)}
        +\|(\tpsi, \tPsi)\|^2_{H^{m}(\Om_L)}
        \right).
    \end{align*}
\medskip

\item[(f)] For the differential operator $\mfrak{L}_1^P(\cdot,\cdot)$ given by Definition \ref{definition:approx coeff and fs}, it is a second order differential operator with respect to the first component of a variable (which corresponds to $\psi$ in Definition \ref{definition:approx coeff and fs}(1)).
    \begin{itemize}
\item[($\tx{f}_1$)]
   The operator $\mfrak{L}_1^{P_0}$, as a second order differential operator with respect to its first argument, is of {\emph{mixed type in $\Om_L$ with a degeneracy occurring on $\displaystyle{\Om_L\cap\{x_1=\ls\}}$}} for the constant $\ls$ given from Lemma \ref{lemma-1d-full EP}. More precisely,
     \begin{equation*}
       \tx{the operator $\mfrak L_1^{P_0}$ is}\,\,\begin{cases}
       \tx{elliptic}\quad&\mbox{for $x_1<\ls$},\\
       \tx{degenerate}\quad&\mbox{for $x_1=\ls$},\\
       \tx{hyperbolic}\quad&\mbox{for $x_1>\ls$}.
       \end{cases}
     \end{equation*}
\item[($\tx{f}_2$)] For each $P=(\tphi, \tpsi, \tPsi)\in
\iterV^{m+1}_{r_2}\times \iterP^m_{r_3}\times \iterP^m_{r_3}$, there exists a function  $\gs^P:[-1,1]\rightarrow (0,L)$ so that
\begin{equation*}
       \tx{the operator $\mfrak L_1^{P}$ is}\,\,\begin{cases}
       \tx{elliptic}\quad&\mbox{for $x_1<\gs^P(x_2)$},\\
       \tx{degenerate}\quad&\mbox{for $x_1=\gs^P(x_2)$},\\
       \tx{hyperbolic}\quad&\mbox{$x_1>\gs^P(x_2)$}.
       \end{cases}
     \end{equation*}
In addition, the function $\gs^P$ satisfies the estimate
    \begin{equation}
    \label{estimate of gs}
    \|\gs^P-\ls\|_{H^{m-1}((-1,1))}\le C  \|(\tphi, \tpsi, \tPsi)\|_{H^m(\Om_L)}
    \end{equation}
for a constant $C>0$ depending only on $(\gam, \zeta_0, J, S_, E_0)$. Furthermore, the function $\gs^P$ satisfies the estimate
\begin{equation}
\label{boundes of gs}
   \frac{15}{16} \ls\le \gs^P(x_2)\le \ls+\frac{1}{16}\min\{\ls, L-\ls\} \quad\tx{for $|x_2|\le 1$}.
\end{equation}

\item[($\tx{f}_3$)] For $\lambda_0:=\bar{a}_{11}\left(\frac{5\ls}{8}\right)(>0)$, it holds that
\begin{equation*}
  \begin{bmatrix}
  a_{11}^P & a_{12}^P\\
  a_{12}^P & 1
  \end{bmatrix}\ge \frac{\lambda_0}{2}\quad\tx{in $\ol{\Om_L}\cap\left\{x_1\le \frac{7\ls}{8}\right\}$}.
\end{equation*}

\item[($\tx{f}_4$)] There exists a constant $\lambda_1>0$ depending only on $(\gam, \zeta_0, J, S_, E_0)$ such that
\begin{equation*}
a_{11}^P\le -\lambda_1\quad\tx{in $\ol{\Om_L}\cap\left\{x_1\ge L-\frac{L-\ls}{10}\right\}$. }
\end{equation*}
\end{itemize}

\end{itemize}

\begin{proof} {\textbf{Step 1.}}
By the generalized Sobolev inequality, one can fix a constant $\delta_1>0$ sufficiently small such that if the inequality
\begin{equation*}
  \max\{r_2, r_3\}\le 2\delta_1
\end{equation*}
holds, then for any $P=(\tphi, \tpsi, \tPsi)\in \iterV^{m+1}_{r_2}\times \iterP^m_{r_3}\times \iterP^m_{r_3}$, the coefficients $(a_{ij}^P, a^P, b_1^P, b_0^P)$ for $i,j=1,2$ and the function $f_1^P$ given by Definition \ref{definition:approx coeff and fs} are well defined. And, one can fix a constant $\delta_2>0$ sufficiently small such that if the inequality
\begin{equation*}
  \max\{r_1, r_2, r_3\}\le 2\delta_2
\end{equation*}
holds, then any $\tilT\in \iterP^m_{r_1}$ satisfies \eqref{smallness condition-2}. Moreover, for any $P=(\tphi,\tpsi, \tPsi)\in \iterV^{m+1}_{r_2}\times \iterP^m_{r_3}\times \iterP^m_{r_3}$, the vector field ${\bf v}(\rx, \nabla\tpsi(\rx), \nabla^{\perp}\tphi(\rx))$ given by \eqref{definition of v} satisfies \eqref{smallness condition-3}.
So the functions $f_2^{(\tilT, P)}$ and $f_3^{(\tilT, P)}$ given by Definition \ref{definition:approx coeff and fs} are well defined. We set
\begin{equation}
\label{delta-choice1}
  \bar{\delta}:=\min\{\delta_1, \delta_2\}.
\end{equation}
Then all the properties stated in (a)--(f) can be directly checked.
\medskip

{\textbf{Step 2.}} A direct computation using $\us$ given by \eqref{EP-1d-reduced} yields $$\bar a_{11}=1-\frac{\bar u_1^{\gam+1}}{\us^{\gam+1}},$$
so we have
\begin{equation*}
  \mfrak{L}_1^{P_0}(v,0)=\left(1-\frac{\bar u_1^{\gam+1}}{\us^{\gam+1}}\right)\der_{11}v+\der_{22}v+\bar a\der_1 v.
\end{equation*}
Then Lemma \ref{lemma-1d-full EP} directly implies the statement ($\tx{f}_1$).
\medskip

{\textbf{Step 3.}} Since $a_{12}^P=a_{21}^P$, we can write $\mfrak{L}_1^P(v,0)$ as
\begin{equation*}
  \mfrak{L}_1^P(v,0)=a_{11}^P\der_{11}v+2a_{12}^P\der_{12}v+\der_{22}v+a^{P}\der_1 v.
\end{equation*}
Due to Lemma \ref{lemma-1d-full EP}, the coefficient $\bar a_{11}$ satisfies
$$
\bar{a}_{11}(L)<0<\bar a_{11}(0),\quad\tx{and}\quad \bar{a}_{11}'=-\frac{(\gam+1)\bar u_1^{\gam}}{\us^{\gam+1}}\bar u_1'.$$
The principal coefficients of the operator $\mfrak{L}_1^P$ form a symmetric matrix $${\mathbb A}^{P}=\begin{bmatrix}a_{11}^{P}&a_{12}^{P}\\
a_{12}^{P}&1\end{bmatrix},$$
so all the eigenvalues of ${\mathbb  A}^{P}$ are real. Hence, the operator $\mfrak{L}_1^P$ is hyperbolic if and only if $a^P_{11}-(a^P_{12})^2<0$, and elliptic if and only if $a^P_{11}-(a^P_{12})^2>0$.
By the estimate \eqref{estimate-coefficient-difference} and the generalized Sobolev inequality, it holds that
\begin{equation}
\label{estimate of C1 difference of a_{11}}
\|\det {\mathbb A}^{P}-\det {\mathbb A}^{P_0}\|_{C^1(\ol{\Om_L})}=\|\det {\mathbb A}^{P}-\bar a_{11}\|_{C^1(\ol{\Om_L})}
\le C\bar{\delta}
\end{equation}
for some constant $C>0$ fixed depending only on $(\gam, \zeta_0, J, S_0, E_0,L)$.
Therefore, we can reduce the constant $\bar{\delta}>0$ from the one given in \eqref{delta-choice1} so that the condition \eqref{condition:r} holds, then the following two properties hold:
\begin{equation*}
  \begin{split}
&\max_{\ol{\Gamex}}\det {\mathbb A}^{P}<0<\min_{\ol{\Gamen}}\det {\mathbb A}^{P},\\
\tx{and}\,\,&\max_{\ol{\Om_L}}\der_{x_1}\det {\mathbb A}^{P}\le
    -\frac 12 \min_{\ol{\Om_L}} \frac{(\gam+1)\bar u_1^{\gam+1}}{\us^{\gam+1}}\bar u'<0.
  \end{split}
\end{equation*}
Then the implicit function theorem yields a unique $C^1$-function $\gs^P:[-1,1]\rightarrow \R$ satisfying
\begin{equation*}
\det {\mathbb A}^{P}(x_1, x_2)
\begin{cases}
>0\quad&\mbox{for $x_1<\gs^P(x_2)$},\\
=0\quad&\mbox{for $x_{1}=\gs^P(x_2)$},\\
<0\quad&\mbox{for $x_1>\gs^P(x_2)$}.
\end{cases}
\end{equation*}
Note that
\begin{equation*}
\det {\mathbb A}^{P}(\gs^P(x_2), x_2)-\bar a_{11}(\ls)= 0  \quad\tx{for $-1\le x_2\le 1$},
\end{equation*}
which can be rewritten as
\begin{equation*}
 \bar a_{11}(\gs(x_1))-\bar a_{11}(\ls)=\bar a_{11}(\gs(x_2))-a_{11}(\gs(x_2),x_2)+a_{12}^2(\gs(x_2),x_2)\quad\tx{for $|x_1|\le 1$.}
\end{equation*}
By using this equation and the trace inequality, we can easily derive the estimate \eqref{estimate of gs}. Moreover we can directly check the estimate \eqref{boundes of gs} by using \eqref{estimate of gs} and the generalized Sobolev inequality. This proves the statement (${\tx f}_2$). Finally, the statements (${\tx f}_3$) and (${\tx f}_4$) can be verified by using \eqref{estimate-coefficient-difference}.

\end{proof}
\end{lemma}

\subsection{Well-posedness of \eqref{lbvp main}}
\label{subsection-lbvp-preliminary}
\subsubsection{Main proposition}
Given constants $(\gam, \zeta_0, J, S_0, E_0)$ with
\begin{equation*}
  \gam>1,\quad \zeta_0>1,\quad J>0,\quad S_0>0,\quad E_0<0,
\end{equation*}
let $(\bar{\vphi}, \bar{\Phi})$ be the associated background solution to the system \eqref{new system with H-decomp2}--\eqref{new system with H-decomp5} in the sense of Definition \ref{definition of background solution-HD}. For a fixed constant $L\in(0, l_{\rm max}]$, let $\bar{\delta}$ be given to satisfy Lemma \ref{lemma on L_1}. Assuming that the condition
\begin{equation}
\label{condition: r in sec4}
  \max\{r_1, r_2, r_3\}\le 2\bar{\delta}
\end{equation}
holds, fix
\begin{equation*}
\tilT\in \iterP^m_{r_1}\quad\tx{and}\quad  P=(\tphi, \tpsi, \tPsi)\in \iterV^{m+1}_{r_2}\times \iterP^m_{r_3}\times \iterP^m_{r_3}.
\end{equation*}

\begin{problem}
\label{problem-lbvp for iteration}
Given functions $f_1\in H^{m-1}(\Om_L)$ and $f_2\in H^{m-1}(\Om_L)$ satisfying the compatibility conditions
\begin{equation}
\label{compatibility conditions for f1 and f2}
  \left(\frac{\der}{\der{x_2}}\right)^k f_1=\left(\frac{\der}{\der{x_2}}\right)^k  f_2=0\quad \tx{on $\Gamw$},
\end{equation}
for $k=2j-1$, $j\in \mathbb{N}$ with $k<m-1$,
find $(v,w)$ that solves the following problem:
\begin{equation}
\label{lbvp-main general}
\begin{split}
\begin{cases}
\mfrak{L}_1^P(v, w)=f_1\quad&\tx{in $\Om_L$},\\
\mfrak L_2(v,w)=f_2\quad&\tx{in $\Om_L$},
\end{cases}\\
v=0,\quad \der_1 w=0\quad&\tx{on $\Gamen$},\\
\der_2v=0,\quad \der_2 w=0\quad&\tx{on $\Gamw$},\\
w=0\quad&\tx{on $\Gamex$}.
\end{split}
\end{equation}
\end{problem}

\begin{proposition}
\label{theorem-wp of lbvp for system with sm coeff}
Suppose that $(u_0, E_0)\in \Tac$ with $E_0<0$, which is equivalent to $0<u_0<\us$, and let $(\bar{\vphi}, 0, {\bar{\Phi}}, S_0)$ be the background solution associated with $(\gam, \zeta_0, J, S_0, E_0)$ in the sense of Definition \ref{definition of background solution-HD}. Then, one can fix two constants $\bJ$ and $\ubJ$ satisfying
\begin{equation*}
  0<\bJ<1<\ubJ<\infty
\end{equation*}
so that whenever $J\in(0, \bJ]\cup[\ubJ,\infty)$,
there exists a constant $d\in(0,1)$ such that if
$(E_0, L)$ are fixed to satisfy the condition \eqref{almost sonic condition1 full EP}, then Problem \ref{problem-lbvp for iteration} is well posed for any $P=(\tphi, \tpsi, \tPsi)\in \iterV^{m+1}_{r_2}\times \iterP^m_{r_3}\times \iterP^m_{r_3}$, provided that the constant $\bar{\delta}>0$ from Lemma \ref{lemma on L_1} is adjusted appropriately. In other words, the constant $\bar{\delta}>0$ can be reduced so that, under the condition of \eqref{condition: r in sec4}, for each $P=(\tphi, \tpsi, \tPsi)\in \iterV^{m+1}_{r_2}\times \iterP^m_{r_3}\times \iterP^m_{r_3}$, Problem \ref{problem-lbvp for iteration} admits a unique solution $(v,w)\in H^m(\Om_L)\times H^m(\Om_L)$. Furthermore, the solution satisfies the estimate
\begin{equation}
\label{estimate of v and w}
      \|v\|_{H^m(\Om_L)}+\|w\|_{H^m(\Om_L)}
      \le C\left(\|f_1\|_{H^{m-1}(\Om_L)}
      +\|f_2\|_{H^{m-2}(\Om_L)}\right)
    \end{equation}
    for some constant $C>0$.
    \smallskip

    In the above, the parameters $\ubJ$ and $\bJ$ are fixed depending only on $(\gam, \zeta_0, S_0)$. And, the constants $d\in(0,1)$ and $\bar{\delta}$ are fixed depending only on $(\gam, \zeta_0, S_0,J)$ and $(\gam,\zeta_0, S_0, E_0, J, L)$, respectively. Finally, the estimate constant $C$ is fixed depending only on $(\gam,\zeta_0, S_0, E_0, J, L)$.

\end{proposition}

Once Proposition \ref{theorem-wp of lbvp for system with sm coeff} is proved, the well-posedness of \eqref{lbvp main} follows directly.

\begin{corollary}[Well-posedness of \eqref{lbvp main}]
\label{corollary:wp of mixed system in iteration}
Under the same assumptions as in Proposition \ref{theorem-wp of lbvp for system with sm coeff}, for each
\begin{equation*}
\tilT\in \iterP^m_{r_1}\quad\tx{and}\quad   P=(\tphi, \tpsi, \tPsi)\in \iterV^{m+1}_{r_2}\times \iterP^m_{r_3}\times \iterP^m_{r_3},
\end{equation*}
the boundary value problem \eqref{lbvp main} has a unique solution $(\psi, \Psi)\in [H^m(\Om_L)]^2$ that satisfies the estimate
\begin{equation}
\label{estimate-linear-pot}
\begin{split}
 & \|\psi\|_{H^m(\Om_L)}+\|\Psi\|_{H^m(\Om_L)}\\
 &\le C\left(\|\tilT\|_{H^{m-1}(\Om_L)}+\|\tphi\|_{H^{m+1}(\Om_L)}+\|(\tpsi, \tPsi)\|_{H^{m}(\Om_L)}^2+\mfrak P(S_0, E_{\rm en}, \om_{\rm en})\right)
  \end{split}
\end{equation}
for some constant $C>0$ depending only on $(\gam,\zeta_0, S_0, E_0, J, L)$. In the above estimate, the term $\mcl{P}(S_0, E_{\rm en}, \om_{\rm en})$ is defined by \eqref{definition-perturbation of bd}.

\begin{proof}
For two functions $v$ and $w$ defined in $\ol{\Om_L}$, define $\psi$ and $\Psi$ by
\begin{equation*}
\begin{cases}
\psi(x_1,x_2):=v(x_1, x_2)+\int_0^{x_2}w_{\rm en}(t)\,dt,\\
\Psi(x_1, x_2):=w(x_1, x_2)+(x_1-L)(E_{\rm en}(x_2)-E_0).
\end{cases}
\end{equation*}
It can be directly checked that $(\psi, \Psi)$ solves \eqref{lbvp main} if and only if $(v,w)$ solves
\begin{equation}
\label{lbvp-mod}
\begin{split}
\begin{cases}
\mfrak{L}_1^P(v, w)=f_1^*\quad&\tx{in $\Om_L$},\\
\mfrak L_2(v,w)=f_{2}^*\quad&\tx{in $\Om_L$},
\end{cases}\\
v=0,\quad \der_1 w=0\quad&\tx{on $\Gamen$},\\
\der_2v=0,\quad \der_2 w=0\quad&\tx{on $\Gamw$},\\
w=0\quad&\tx{on $\Gamex$},
\end{split}
\end{equation}
for $f_1^*$ and $f_2^*$ given by
\begin{equation*}
  \begin{cases}
  f_1^*:=f_1^P-\mfrak{L}_1^P(\int_0^{x_2}w_{\rm en}(t)\,dt, (x_1-L)(E_{\rm en}(x_2)-E_0)),\\
  f_2^*:=f_2^{(\tilT, P)}-\mfrak{L}_2(\int_0^{x_2}w_{\rm en}(t)\,dt, (x_1-L)(E_{\rm en}(x_2)-E_0)).
  \end{cases}
\end{equation*}
If $w_{\rm en}$ and $E_{\rm en}$ satisfy Condition \ref{conditon:1}, then it directly follows from Lemma \ref{lemma on L_1} (f) that
\begin{equation*}
  \left(\frac{\der}{\der{x_2}}\right)^k f_1^*=\left(\frac{\der}{\der{x_2}}\right)^k  f_2=0\quad \tx{on $\Gamw$},
\end{equation*}
for $k=2j-1$, $j\in \mathbb{N}$ with $k<m-1$.
Therefore, Proposition \ref{theorem-wp of lbvp for system with sm coeff} yields a unique solution $(v,w)\in[H^{m}(\Om_L)]^2$ to \eqref{lbvp-mod}, so the boundary value problem \eqref{lbvp main} has a unique solution $(\psi, \Psi)\in [H^{m}(\Om_L)]^2$. Moreover, the estimate \eqref{estimate-linear-pot} is a result directly following from the estimate \eqref{estimate of v and w} and Lemma \ref{lemma on L_1} (e).
\end{proof}
\end{corollary}

\medskip

\subsubsection{Proof of Proposition \ref{theorem-wp of lbvp for system with sm coeff}}
Throughout this section, let $P=(\tphi, \tpsi, \tPsi)\in  \iterV^{m+1}_{r_2}\times \iterP^m_{r_3}\times \iterP^m_{r_3}$ be fixed unless otherwise specified.
\begin{notation}
(1) For simplicity, we denote $(a_{ij}^P, a^P, b_1^P, b_0^P)$, and its approximated sonic interface $\gs^P$(see Lemma \ref{lemma on L_1}
 ($\tx{f}_2$)) by $(a_{ij}, a, b_1, b_0)$ and $\gs$, respectively.\\
(2) The terms $(a_{ij}, a, b_1, b_0)$ and $\gs$ for $P=(0,0,0)(=:P_0)$ are denoted by $(\bar a_{ij}, \bar a, \bar b_1, \bar b_0)$ and $\ls$(which is a constant), respectively (see Lemma \ref{lemma on L_1} (d) and (f)).
\end{notation}

\begin{proclaim}
(1) For the remainder of the article, we state that an estimate constant is fixed depending only on the data if it is fixed depending only on $(\gam, \zeta_0, S_0, E_0, J, L)$ and $m$.
(2) Unless otherwise specified, any estimate constant $C$ appearing in the following is assumed to be fixed depending only on the data.
\end{proclaim}
\medskip

In general, Proposition \ref{theorem-wp of lbvp for system with sm coeff} can be proved by adjusting the proof of Theorem \ref{theorem-1-full}. But there is an essential difference in that we need to establish a priori $H^1$-estimate of a solution $(v,w)$ to \eqref{lbvp-main general}, which consists of two second-order PDEs coupled by lower-order derivative terms. Once an a priori $H^1$ estimate is obtained, the proof can be completed using a bootstrap argument. Therefore, the most important ingredient in proving Proposition \ref{theorem-wp of lbvp for system with sm coeff} is the following lemma:

\begin{lemma}
\label{proposition-H1-apriori-estimate}
Suppose that $(v,w)$ is a smooth solution to \eqref{lbvp-main general} associated with $P=(\tphi, \tpsi, \tPsi)\in \iterV^{m+1}_{r_2}\times \iterP^m_{r_3}\times \iterP^m_{r_3}$.
There exist two constants $\bJ$ and $\ubJ$ satisfying
\begin{equation*}
  0<\bJ<1<\ubJ<\infty
\end{equation*}
so that, for any $J\in (0, \bJ]\cup [\ubJ, \infty)$, one can fix a constant $d\in(0,1)$, and reduce the constant $\bar{\delta}>0$ further from the one given in Lemma \ref{lemma on L_1}  so that if
\begin{itemize}
\item[(i)]$(E_0, L)$ are fixed to satisfy the condition \eqref{almost sonic condition1 full EP},
\item[(ii)]
\begin{equation*}
  \max\{r_2, r_3\}\le 2\bar{\delta},
\end{equation*}
\end{itemize}
then $(v,w)$
satisfies the estimate
    \begin{equation}
    \label{A priori H1 estimate}
    \begin{split}
      &\|\der_{1}v\|_{L^2(\Gamen)}+\|Dv\|_{L^2(\Gamex)}+\|v\|_{H^1(\Om_L)}+\|w\|_{H^1(\Om_L)}\\
      &\le C\left(\|f_1\|_{L^2(\Om_L)}+\|f_2\|_{L^2(\Om_L)}\right)
      \end{split}
    \end{equation}
    for some constant $C>0$.
\medskip

    In the above, the constants $\bJ$ and $\ubJ$ are fixed depending only on $(\gam, \zeta_0, S_0)$. For each $J\in (0, \bJ]\cup [\ubJ, \infty)$, the constant $d\in(0,1)$ is depending on $(\gam, \zeta_0, S_0,J)$. Finally, the constant $\bar{\delta}>0$ is adjusted depending only on the data (in the sense of Proclamation (1)).

\end{lemma}

\begin{proof}
%The proof is divided in five steps.

{\textbf{Step 1.}}
For  $i,j=1,2$, set
\begin{equation}
\label{def of coeff pert}
(d a_{ij}, d a, d b_1, d b_0)
:=(a_{ij}, a, b_1, b_0)-(\bar a_{ij}, \bar a, \bar b_1, \bar b_0)\quad\tx{in $\Om_L$}.
\end{equation}
Next, set
\begin{equation*}
 \mfrak{L}_1:=\mfrak{L}_1^{P_0}\quad\tx{and}\quad
 d \mfrak{L}_1^{P}:=\mfrak{L}_1^{P}-\mfrak{L}_1.
\end{equation*}
Then we have
\begin{equation*}
 d \mfrak{L}_1^{P}(v,w)=d a_{11}\der_{11}v+2a_{12}\der_{12}v+d a\der_1 v+d b_1\der_1 w+d b_0w\quad\tx{in $\Om_L$}.
\end{equation*}
And, rewrite \eqref{lbvp-main general} as
\begin{equation*}
\begin{split}
\begin{cases}
\mfrak L_1(v, w)=f_1-d \mfrak{L}_1^{P}(v,w)=:F_1\quad&\tx{in $\Om_L$},\\
\mfrak L_2(v,w)=f_2\quad&\tx{in $\Om_L$},
\end{cases}\\
v=0,\quad \der_1 w=0\quad&\tx{on $\Gamen$},\\
\der_2v=0,\quad \der_2 w=0\quad&\tx{on $\Gamw$},\\
w=0\quad&\tx{on $\Gamex$}.
\end{split}
\end{equation*}

For a function $G=G(x_1)$ to be determined later, denote
\begin{equation}
\label{definition of I1 and I2}
 I_1:=\int_{\Om_L}  G\der_1v{\mfrak L}_1(v,w)\,d\rx,\quad
 I_2:=\int_{\Om_L} w{\mfrak L}_2(v,w)\,d\rx.
\end{equation}
Clearly, it holds that
\begin{equation}
\label{I1+I2 in H1 estimate}
  I_1+I_2=\int_{\Om_L} F_1G\der_1v+f_2w\,d\rx.
\end{equation}
By using Definition \ref{definition:coefficients-nonlinear}, and integrating by parts with using the boundary conditions for $(v,w)$, one can check that
\begin{equation*}
\begin{split}
  I_1=&\frac 12\int_{\Gamex}(\bar a_{11}(\der_1v)^2-(\der_2v)^2)G\,dx_2-\frac 12\int_{\Gam_0} \bar a_{11}(\der_1v)^2 G\,dx_2\\
  &+\int_{\Om_L}\left(\left(-\frac 12\bar a_{11}'+\bar a\right)G-\frac 12 \bar a_{11}G'\right) (\der_1 v)^2+G'\frac{(\der_2v)^2}{2}\,d\rx\\
  &+\int_{\Om_L} (\bar b_1\der_1v\der_1w+\bar b_0w\der_1v)G\,d\rx,\\
 \tx{and}\quad I_2
  =&\int_{\Om_L}\left(-|\nabla w|^2-\frac{\bar{\rho}}{\gam S_0\bar{\rho}^{\gam-1}} w^2\right)+\frac{J}{\gam S_0\bar{\rho}^{\gam-1}}w\der_1v \,d\rx.
  \end{split}
\end{equation*}
\smallskip

The key of the proof is to  choose the function $G$ as
\begin{equation}
\label{definition of G}
  G(x_1)=\bar{\rho}^{\eta}(x_1)
\end{equation}
for a constant $\eta>0$ to be determined later.
\smallskip

Define a function $\alp_1$ by
\begin{equation}
\label{main-coeff}
\begin{split}
\alp_1:=&  -\frac 12(\bar a_{11}G)'+\bar aG=\frac{\bar{\rho}'}{\bar{\rho}} \frac G2 \left((\gam-1+\eta)\frac{\bar u_1^{\gam+1}}{\us^{\gam+1}}+(2-\eta)\right).
  \end{split}
\end{equation}
We rearrange the terms in $-(I_1+I_2)$ as follows:
\begin{equation*}
-(I_1+I_2)=T_{\rm bd}+T_{\rm coer}+T_{\rm mix}
\end{equation*}
for
\begin{align}
\label{definition-Jbd}
T_{\rm bd}:=&\frac{G(0)}{2} \int_{\Gam_0}\bar a_{11}(\der_1v)^2 \,dx_2-\frac{G(L)}{2}\int_{\Gamex}(\bar a_{11}(\der_1v)^2-(\der_2v)^2)\,dx_2,\\ %(>0)\\
\label{definition-Jp}
T_{\rm coer}:=&-\int_{\Om_L}\alp_1(\der_1v)^2
  +\eta \frac{\bar{\rho}'}{\bar{\rho}}G\frac{(\der_2 v)^2}{2}\,d\rx+
  \int_{\Om_L}|\nabla w|^2+\frac{\bar{\rho}}{\bar c^2} w^2\,d\rx,\\
  \label{definition-Jc}
T_{\rm mix}:=&  \int_{\Om_L}\frac{(\gam-1)\bar u}{\gam S_0\bar{\rho}^{\gam-1}}\frac{\bar \rho'}{\bar\rho}G w\der_1v
-G\frac{\bar u}{\gam S_0\bar{\rho}^{\gam-1}}\der_1w\der_1v-\frac{J}{\gam S_0\bar{\rho}^{\gam-1}}\der_1v w\,d\rx.
\end{align}
Note that the proof of Lemma \ref{lemma on L_1} (see Step 3 in the proof) gives that
\begin{equation*}
  \bar a_{11}>0\,\,\tx{on $\Gamen$}\quad\tx{and}\quad \bar a_{11}<0\,\,\tx{on $\Gamex$},
\end{equation*}
and this implies that
 \begin{equation*}
  T_{\rm bd}\ge 0.
 \end{equation*}
Later, we give an improved estimate for a lower bound of $T_{\rm bd}$, so that it can be used to handle the singular perturbation problem \eqref{bvp-sing-pert}.
\smallskip

In the following steps, we shall fix $\eta>0$ and find $(\bJ, \ubJ)$ with
\begin{equation*}
  0<\bJ<1<\ubJ<\infty
\end{equation*}
so that whenever $J\in(0, \bJ]\cup[\ubJ,\infty)$, there exists a constant $d\in(0,1)$ depending on $(\gam, \zeta_0, S_0,J)$ such that if
$(E_0, L)$ are fixed to satisfy the condition \eqref{almost sonic condition1 full EP}, then it holds that
\begin{equation*}
T_{\rm coer}+T_{\rm mix}\ge \lambda \int_{\Om_L} |Dv|^2+|Dw|^2\,d\rx
\end{equation*}
for some constant $\lambda>0$.
%\begin{note}
%We must point out that the particular choice for the range of $J$ is given for the sole purpose of controlling the term $T_{\rm mix}$ that arises due to the coupling structure of the system in \eqref{lbvp-mod}.
%\end{note}
\medskip

{\textbf{Step 2.}} Define a function $\beta$ by
\begin{equation}
\label{definition of omega}
  \beta:=\frac{(\gam-1)\bar u_1}{\gam S_0\bar{\rho}^{\gam-1}}\frac{\bar \rho'}{\bar\rho}G-\frac{J}{\gam S_0\bar{\rho}^{\gam-1}},
\end{equation}
and rewrite $T_{\rm mix}$ as
\begin{equation*}
T_{\rm mix}=\int_{\Om_L}\beta w \der_1v-G\frac{\bar u_1}{\gam S_0\bar{\rho}^{\gam-1}}\der_1w\der_1v\,d\rx.
\end{equation*}
Owing to the boundary condition $w=0$ on $\Gamex$, the Poincar\'e inequality gives
\begin{equation*}
\int_{\Om_L}w^2\,d\rx\le L^2 \int_{\Om_L}(\der_1w)^2\,d\rx.
\end{equation*}
By using this estimate and the Cauchy-Schwarz inequality, one has
\begin{equation}
\label{estimate-Tmix}
\begin{split}
|T_{\rm mix}|
  &\le \frac 14 \int_{\Om_L}(\der_1 w)^2\,d\rx+
  \int_{\Om_L}
  2\left(\left(G\frac{\bar u_1}{\gam S_0\bar{\rho}^{\gam-1}}\right)^2+L^2\beta^2\right)(\der_1 v)^2\,d\rx.
  \end{split}
\end{equation}
Define
\begin{equation}
\label{definition-alp2}
  \alp_2:=2\left(\left(G\frac{\bar u_1}{\gam S_0\bar{\rho}^{\gam-1}}\right)^2+L^2\beta^2\right).
\end{equation}
For $\alp_1$ given by \eqref{main-coeff}, let us set
\begin{equation}
\label{definition of beta*}
\alp:=-\alp_1-\alp_2.
\end{equation}
Then it follows from \eqref{definition-Jp} and \eqref{estimate-Tmix} that
\begin{equation}
\label{inequality 1 of Jp+JC}
  T_{\rm coer}+T_{\rm mix}\ge
  \int_{\Om_L} \alp (\der_1 v)^2-\frac{\eta}{2}\frac{\bar \rho'}{\bar \rho}G (\der_2v)^2\,d\rx
  +\frac 34\int_{\Om_L}|\nabla w|^2+\frac{\bar{\rho}}{\gam S_0\bar{\rho}^{\gam-1}} w^2\,d\rx.
\end{equation}
As long as $\eta$ is fixed as a positive constant, it follows from Lemma \ref{lemma-1d-full EP} that
\begin{equation*}
  \min_{\ol{\Om_L}}\frac{\eta}{2}\frac{\bar u_1'}{\bar u_1}G>0.
\end{equation*}
%Next, we shall find conditions to achieve  $\displaystyle{\min_{\ol{\Om_L}}\alp>0}$.
\medskip

{\textbf{Step 3.}} Note that $(\bar u_1, \bar E)(x_1)$ satisfies
\begin{equation*}
\bar{E}^2=2H(\bar u_1)\quad\tx{for $x_1\in (0, l_{\rm max}]$}
\end{equation*}
for the function $H$ given by \eqref{definition of H}. Define
\begin{equation}\label{definition of kappa}
  \kappa:=\frac{\bar u_1}{\us},\quad E(\kappa):=\bar{E}(\bar u_1),\quad \tx{and}\quad \mcl{F}(\kappa):=  \int_1^{\kappa}
\left(1-\frac{t}{\zeta_0}\right)
\left(1-\frac{1}{t^{\gam+1}}\right)
\,
dt.
\end{equation}
Then, the function $E(\kappa)$ satisfies
\begin{equation}
\label{definition of F}
\begin{split}
  E^2(\kappa)=2\us J\mcl{F}(\kappa).
\end{split}
\end{equation}
By the definition of $\Tac$(see \eqref{definition of T and Tpm}), it holds that
\begin{equation*}
  (\kappa-1)E(\kappa)\ge 0,
\end{equation*}
so we obtain that
\begin{equation}
\label{expression of E in tau}
   E(\kappa)=\begin{cases}
  -\sqrt{2\us J\mcl{F}(\kappa)}&\quad\mbox{for $\kappa<1$},\\
  \sqrt{2\us J\mcl{F}(\kappa)}&\quad\mbox{for $\kappa\ge 1$}.
  \end{cases}
\end{equation}
Define $\mcl{H}(\kappa)$ by
\begin{equation}
\label{definition of mcl H}
    \mcl{H}(\kappa):=
  \left|\frac{ \kappa^{\gam-1}\sqrt{\mcl{F}(\kappa)}}{\kappa^{\gam+1}-1}\right|.
\end{equation}
Substituting \eqref{expression of E in tau} into the differential equation for $\bar u_1$ given in \eqref{EP-1d-reduced} yields
\begin{equation}
\label{ode of rho in kappa}
\begin{split}
  \frac{\bar u'_1}{\bar u_1}
   &=\frac{\sqrt{2\us J}}{\us^2}\mcl{H}(\kappa),\quad\tx{and}\quad
   \frac{\bar{\rho}'}{\bar{\rho}}=
    -\frac{\sqrt{2\us J}}{\us^2}
  \mcl{H}(\kappa).
 \end{split}
\end{equation}

By using the definition of $\us$ given in \eqref{EP-1d-reduced}, one can express $\us$ as
\begin{equation}
\label{definition of us}
\us=h_0J^{\frac{\gam-1}{\gam+1}}
\quad\tx{with $h_0:=(\gam S_0)^{\frac{1}{\gam+1}}$}.
\end{equation}
Substitute this expression into \eqref{ode of rho in kappa} to get
\begin{equation}
\label{ode for rho in kappa}
    \frac{\bar{\rho}'}{\bar{\rho}}=-\sqrt 2 h_0^{-\frac 32} J^{\frac{2-\gam}{\gam+1}}\mcl{H}(\kappa).
\end{equation}
Then, the function $\alp_1$(see \eqref{main-coeff}) can be expressed as
\begin{equation}
\label{alp1-Jkappa-expre}
\begin{split}
  -\alp_1= \sqrt 2 h_0^{-\frac 32} J^{\frac{2-\gam}{\gam+1}} \mcl{H}(\kappa) \frac G2
  \left((\gam-1+\eta)\kappa^{\gam+1}+2-\eta\right).
  \end{split}
\end{equation}
Using \eqref{definition of kappa}, \eqref{definition of us} and the expression $\displaystyle{\bar{\rho}=\frac{J}{\bar u_1}}$, one can explicitly check that
\begin{equation*}
\begin{split}
  &\left(G\frac{\bar u_1}{\gam S_0\bar{\rho}^{\gam-1}}\right)^2=\frac{G^2\kappa^{2\gam}
  J^{\frac{-2(\gam-1)}{\gam+1}}}{h_0^2},\\
  \tx{and}\quad &\beta=-\sqrt 2 (\gam-1)h_0^{-\frac 52}\kappa^{\gam}\mcl{H}(\kappa)J^{\frac{3-2\gam}{\gam+1}}G-
h_0^{-2}\kappa^{\gam-1}J^{\frac{3-\gam}{\gam+1}}.
  \end{split}
\end{equation*}

Since $\displaystyle{\min_{x_1\in[0, L]}\bar{u}'_1(x_1)>0}$ holds by Lemma \ref{lemma-1d-full EP}, the function $\bar u_1:[0, L]\rightarrow (0, \infty)$ is invertible. So we use the first equation in \eqref{ode of rho in kappa} to get
\begin{equation}
\label{L-computation-n}
  L=\int_{u_0}^{\bar u_1(L)}\frac{\bar u_1}{\bar u'_1}\frac{d\bar u_1}{\bar u_1}=\sqrt{\frac{h_0^3}{2}}
  J^{\frac{\gam-2}{\gam+1}}\int_{\kappa_0}^{\kappa_L}\frac{1}{\kappa \mcl{H}(\kappa)}\,d\kappa
\end{equation}
for
\begin{equation*}
 \kappa_0:=\frac{u_0}{\us},\quad \kappa_L:=\frac{\bar u_1(L)}{\us}.
\end{equation*}
Define
\begin{equation}
\label{definition of lambda}
\lambda(\kappa_0, \kappa_L):=\left(\int_{\kappa_0}^{\kappa_L} \frac{1}{\tau\mcl{H}(\kappa)}\,d\kappa\right)^2.
\end{equation}
A straightforward computation gives
\begin{equation*}
  (L\beta)^2=\frac{1}{2h_0}\kappa^{2(\gam-1)}
  \lambda(\kappa_0, \kappa_L)J^{\frac{-2(\gam-1)}{\gam+1}}
  \left(\sqrt 2(\gam-1)h_0^{-\frac 12} \kappa\mcl{H}(\kappa)G+J^{\frac{\gam}{\gam+1}}\right)^2.
\end{equation*}
In terms of $(J,\kappa)$, $\alp_2$ can be expressed as
\begin{equation}
\label{alp2-Jkappa-expre}
\begin{split}
  {\alp}_2=&\frac{2}{h_0^2}\kappa^{2\gam} J^{\frac{-2(\gam-1)}{\gam+1}}G^2\\
  &+
  \frac{1}{h_0}\kappa^{2(\gam-1)}\lambda(\kappa_0, \kappa_L)J^{\frac{-2(\gam-1)}{\gam+1}}
  \left(\sqrt 2(\gam-1)h_0^{-\frac 12} \kappa\mcl{H}(\kappa)G+J^{\frac{\gam}{\gam+1}}\right)^2.
  \end{split}
\end{equation}

Define
\begin{equation*}
G_*:=J^{\frac{2-\gam}{\gam+1}}G.
\end{equation*}
By using \eqref{definition of G} and \eqref{definition of us}, the term $G_*$ can be expressed in terms of $(J, \kappa)$ as
\begin{equation*}
  G_*=\kappa^{-\eta}h_0^{-\eta}J^{\frac{2-\gam+2\eta}{\gam+1}}.
\end{equation*}
By using \eqref{alp1-Jkappa-expre}, \eqref{alp2-Jkappa-expre} and the expression in the above, we represent $-\alp_1$ and $\alp_2$ as
\begin{equation*}
\begin{split}
-\alp_1&=\frac{\sqrt 2}{2}h_0^{-\frac 32} G_*\mcl{H}(\kappa)\left((\gam-1+\eta)\kappa^{\gam+1}+2-\eta\right),\\
  \alp_2&=\frac{2}{h_0^{2+\eta}}\kappa^{2\gam-\eta}J^{\frac{-\gam+2\eta}{\gam+1}}G_*+
  \frac{1}{h_0}\kappa^{2(\gam-1)}\lambda(\kappa_0, \kappa_L)J^{\frac{2}{\gam+1}}
  \left(\sqrt{2}(\gam-1)h_0^{\frac 32-\eta}\kappa^{1-\eta}\mcl{H}(\kappa)J^{\frac{2\eta-\gam}{\gam+1}}+1\right)^2.
\end{split}
\end{equation*}
Then the function $\alp$ given by \eqref{definition of beta*} can be expressed as
\begin{equation}
\label{new expression of alp}
  \alp(\kappa;\kappa_0, \kappa_L, J)=\om_1(\kappa; J, \eta)G_*-\om_2(\kappa;\kappa_0, \kappa_L, J, \eta)
\end{equation}
for
\begin{equation}
\label{expression of alp}
  \begin{split}
  &\om_1(\kappa;J, \eta):=\frac{\sqrt 2}{2}h_0^{-\frac 32}\mcl{H}(\kappa)((\gam-1)\kappa^{\gam+1}+\eta(\kappa^{\gam+1}-1)+2)-
  \frac{2}{h_0^{2+\eta}}\kappa^{2\gam-\eta}J^{\frac{2\eta-\gam}{\gam+1}},\\
  &\om_2(\kappa;\kappa_0, \kappa_L, J, \eta):=\frac{1}{h_0}\kappa^{2(\gam-1)}\lambda(\kappa_0, \kappa_L)J^{\frac{2}{\gam+1}}
  \left(\sqrt 2(\gam-1)h_0^{\frac 32-\eta}\kappa^{1-\eta}\mcl{H}(\kappa)J^{\frac{2\eta-\gam}{\gam+1}}+1\right)^2.
  \end{split}
\end{equation}

\medskip

{\textbf{Step 4.}}
From \eqref{definition of F}, it can be easily checked that $\mcl{F}(1)=0$ and $\mcl{F}'(1)=0$. By applying L$'$H\^{o}pital's rule, one can directly check that
\begin{equation}
\label{H at 1}
 \mcl{H}(1)= \lim_{\kappa\to 1} \mcl{H}(\tau)=\lim_{\kappa\to 1}\frac{\sqrt{\frac 12 \mcl{F}''(1)(\kappa-1)^2}}{\kappa^{\gam+1}-1}=\frac{\sqrt{\frac 12 \mcl{F}''(1)}}{\gam+1}=\frac{\sqrt{1-\frac{1}{\zeta_0}}}{\sqrt{2(\gam+1)}}.
 \end{equation}
 This result, combined with \eqref{condition for zeta0} implies that $\mcl{H}(1)>0$. Therefore, $\mcl{H}(\kappa)$, given by \eqref{definition of mcl H}, satisfies that
\begin{equation}\label{positivity of mcl H}
  \mcl{H}(\kappa)>0\quad\tx{for all $\kappa>0$.}
\end{equation}
It follows from \eqref{definition of lambda} that
\begin{equation*}
\lim_{{\kappa_0\to 1-}\atop{\kappa_L\to 1+}}\lambda(\kappa_0, \kappa_L)=0.
\end{equation*}
Hence we obtain the following important result:
\begin{equation}
\label{alp limit}
  \lim_{{\kappa_0\to 1-}\atop{\kappa_L\to 1+}}{\alp}(1 ;\kappa_0, \kappa_L, J,\eta)=
  h_0^{-\eta}J^\frac{2-\gam+2\eta}{\gam+1}
  \left(\frac 12h_0^{-\frac 32}\sqrt{1-
  \frac{1}{\zeta_0}}\sqrt{\gam+1}
  -\frac{2}{h_0^{2+\eta}}J^{\frac{2\eta-\gam}{\gam+1}}\right).
\end{equation}

Now, we consider two cases:\\
\phantom{a}\quad (Case 1)\,\,$\eta>0$ is fixed to satisfy $2\eta-\gam>0$,\\
\phantom{a}\quad (Case 2)\,\,$\eta>0$ is fixed to satisfy $2\eta-\gam<0$.
%\begin{itemize}
%\item[(Case 1)] $\eta>0$ is fixed to satisfy $2\eta-\gam>0$
%\item[(Case 2)] $\eta>0$ is fixed to satisfy $2\eta-\gam<0$
%\end{itemize}
\medskip

 Returning to \eqref{definition of G}, fix the constant $\eta$ as $\eta=\frac 34\gam$ that yields $2\eta-\gam>0$. This corresponds to (Case 1). Then, one can fix a constant $\bar{J}>0$ sufficiently small depending only on $(\gam, \zeta_0, S_0)$ so that whenever the background momentum density $J$ satisfies the inequality
\begin{equation*}
  0<J\le \bar J,
\end{equation*}
we obtain from \eqref{alp limit} that
\begin{equation*}
 \lim_{{\kappa_0\to 1-}\atop{\kappa_L\to 1+}}{\alp}(1 ;\kappa_0, \kappa_L, J, \frac 34 \gam)\ge \frac 14 h_0^{-\frac 34(\gam+2)}J^{\frac{4+\gam}{2(\gam+1)}}\sqrt{(\gam+1)\left(1-\frac{1}{\zeta_0}\right)}.
\end{equation*}
Note that ${\alp}(\kappa; \kappa_0, \kappa_L, J, \frac 34 \gam)$ is continuous with respect to $(\kappa, \kappa_0, \kappa_L)$. Therefore, one can fix a small constant $d$ with $d\in (0,1)$ depending only on $(\gam, \zeta_0, S_0, J)$ so that if the inequality
\begin{equation*}
  1-d\le\kappa_0< 1< \kappa_L \le 1+d
\end{equation*}
holds,
then we have
\begin{equation}
\label{positive lower bound 1 of alp}
   \min_{\kappa\in[\kappa_0, \kappa_L]}{\alp}(\kappa ;\kappa_0, \kappa_L, J, \frac 34\gam)\ge \frac 18 h_0^{-\frac 34(\gam+2)}J^{\frac{4+\gam}{2(\gam+1)}}\sqrt{(\gam+1)\left(1-\frac{1}{\zeta_0}\right)}.
\end{equation}
\medskip

 Next, fix $\eta$ as $\eta=\frac{\gam}{4}$ which yields $2\eta-\gam<0$. This corresponds to (Case 2).  Then, one can fix a constant $\underline{J}>1$ sufficiently large depending only on $(\gam, \zeta_0, S_0)$ so that whenever the background momentum density $J$ satisfies the inequality
\begin{equation*}
  J\ge \underbar{J},
\end{equation*}
then we obtain that
\begin{equation*}
  \lim_{{\kappa_0\to 1-}\atop{\kappa_L\to 1+}}{\alp}(1 ;\kappa_0, \kappa_L, J, \frac{\gam}{4})\ge \frac 14 h_0^{-\frac 14(\gam+6)}
  J^{\frac{4-\gam}{\gam+1}}
  \sqrt{(\gam+1)\left(1-\frac{1}{\zeta_0}\right)}>0.
\end{equation*}
Therefore, one can fix a small constant $d\in(0,1)$ depending only on $(\gam, \zeta_0, S_0, J)$ so that if the inequality
\begin{equation*}
  1-d\le\kappa_0< 1< \kappa_L \le 1+d
\end{equation*}
holds, then we have
\begin{equation}
\label{positive lower bound 2 of alp}
 \min_{\kappa\in[\kappa_0, \kappa_L]} {\alp}(\kappa ;\kappa_0, \kappa_L,J, \frac{\gam}{4})\ge \frac 18 h_0^{-\frac 14(\gam+6)}
  J^{\frac{4-\gam}{\gam+1}}
  \sqrt{(\gam+1)\left(1-\frac{1}{\zeta_0}\right)}.
\end{equation}

\medskip

{\textbf{Step 5.}} Returning to to \eqref{inequality 1 of Jp+JC}, we shall use \eqref{definition of kappa}, \eqref{definition of us} and \eqref{ode for rho in kappa} to express the coefficients $-\frac{\eta}{2}\frac{\bar{\rho}'}{\bar{\rho}}G$ and $\frac{\bar{\rho}}{\gam S_0 \bar{\rho}^{\gam-1}}$ on the right-hand side in terms of $(\kappa, J)$ to get
\begin{equation*}
%\label{inequality 1 of Jp+JC in kappa and J}
\begin{split}
  T_{\rm coer}+T_{\rm mix}\ge&
  \int_{\Om_L} \alp (\der_1 v)^2+
  \frac{\eta}{\sqrt 2}
 h_0^{-(\frac 32+\eta)}\kappa^{-\eta}\mcl{H}(\kappa)
  J^{\frac{2+2\eta-\gam}{\gam+1}} (\der_2v)^2\,d\rx\\
  &+\frac 34\int_{\Om_L}|\nabla w|^2+\frac{1}{h_0^3}\kappa^{\gam-2}J^{\frac{2(2-\gam)}{\gam+1}} w^2\,d\rx.
  \end{split}
\end{equation*}
For the constant $\eta$ fixed as
\begin{equation}
\label{choice of eta}
  \eta=\begin{cases}
  \frac{3\gam}{4}\quad&\mbox{for $J\le \bJ$},\\
  \frac{\gam}{4}\quad&\mbox{for $J\ge \ubJ$},
  \end{cases}
\end{equation}
define
\begin{equation*}
  \lambda_0:=\min_{ \kappa\in[\kappa_0, \kappa_L]} \left\{{\alp}(\kappa ;\kappa_0, \kappa_L, J), \,\, \frac{\eta}{\sqrt 2}
 h_0^{-(\frac 32+\eta)}\kappa^{-\eta}\mcl{H}(\kappa)
  J^{\frac{2+2\eta-\gam}{\gam+1}},\,\, \frac{1}{h_0^3}\kappa^{\gam-2}J^{\frac{2(2-\gam)}{\gam+1}},\,\,
  \frac 34 \right\}.
\end{equation*}
By \eqref{positivity of mcl H}, \eqref{positive lower bound 1 of alp} and \eqref{positive lower bound 2 of alp}, it is clear that $\lambda_0$ is positive, and we obtain that
\begin{equation}
\label{estimate of Jp+Jc}
T_{\rm coer}+T_{\rm mix}\ge \lambda_0\int_{\Om_L}|Dv|^2+|Dw|^2+w^2\,d\rx.
\end{equation}
Note that the choice of $\lambda_0$ depends only on $(\gam, \zeta_0, S_0, J)$.
\medskip

In treating a boundary value problem derived from \eqref{lbvp-main general} with a singular perturbation, which we shall introduce later (as an analogy of \eqref{bvp-aux}), it is important to have a coerceivity of $T_{\rm bd}$, given by \eqref{definition-Jbd}, as a functional of $v$.
By rewriting $G(0)$, $G(L)$, $\bar{a}_{11}(0)$ and $\bar{a}_{11}(L)$ in terms of $(\kappa_0, \kappa_L, J)$, we can express $T_{\rm bd}$ as
\begin{equation*}
  T_{\rm bd}=\frac{ h_0^{-\eta}J^{\frac{2\eta}{\gam+1}}}{2}
  \left(\kappa_0^{-\eta}(1-\kappa_0^{\gam+1})\int_{\Gamen} (\der_1 v)^2\,dx_2
  +\kappa_L^{-\eta}\int_{\Gamex}(\kappa_L^{\gam+1}-1) (\der_1 v)^2+ (\der_2 v)^2\,dx_2\right).
\end{equation*}
Since $\kappa_0<1<\kappa_L$, one can fix a constant $\lambda_{\rm bd}>0$ satisfying
\begin{equation}\label{estimate of Jbd}
 T_{\rm bd}\ge \lambda_{\rm bd}\left(\int_{\Gamen} (\der_1 v)^2\,dx_2+
\int_{\Gamex} |Dv|^2\,dx_2\right).
\end{equation}

Finally, we combine \eqref{estimate of Jp+Jc} and \eqref{estimate of Jbd} to obtain that
\begin{equation}
\label{lower bound of -(I1+I2)}
\tx{LHS of \eqref{I1+I2 in H1 estimate}}
  \le -\lambda_0\int_{\Om_L}|Dv|^2+|Dw|^2+w^2\,d\rx
  -\lambda_{\rm bd}
  \left(\int_{\Gamen}(\der_1v)^2\,dx_2
  +\int_{\Gamex}|Dv|^2\,dx_2\right).
\end{equation}

\medskip
{\textbf{Step 6.}} By using the definition of $a_{12}$ given by \eqref{coefficients of L_1} and the compatibility condition $\der_2\tpsi=0$ on $\Gamw$, it can be directly checked that
\begin{equation}
\label{compatibility condition of a12}
  a_{12}=0\quad\tx{on $\Gamw$}.
\end{equation}
Then we integrate by parts and use \eqref{compatibility condition of a12} to get
\begin{equation*}
\begin{split}
&\int_{\Om_L} d\mfrak{L}_1^{P}(v,w) G\der_1 v\,d\rx\\
&=\left(\int_{\Gamex}-\int_{\Gamen}\right)Gd a_{11}\frac{(\der_1 v)^2}{2}\,dx_2
-\int_{\Om_L}\left(\frac 12 \der_1(G\,da_{11})+G(\der_2a_{12}-d a)\right)(\der_1 v)^2\,d\rx\\
&\phantom{==}
+\int_{\Om_L} G(d b_1\der_1 w+d b_0 w)\der_1 v\,d\rx.
\end{split}
\end{equation*}
If condition \eqref{condition: r in sec4} holds, then Lemma \ref{lemma on L_1}(d) combined with the generalized Sobolev inequality yields
\begin{equation}
\label{estimate of delta coefficients}
  \|(d a_{11}, a_{12}, d a, d b_1, d b_0)\|_{C^1(\ol{\Om_L})}\le C(r_2+r_3)
\end{equation}
for some constant $C>0$. So we obtain that
\begin{equation}
\begin{split}
\label{weak estimate of approx problem remainder}
  &\left|\int_{\Om_L} d\mfrak{L}_1^P(v,w) G\der_1 v\,d\rx\right|\\
  &\le C_*(r_2+r_3)\left(\|\der_1v\|^2_{L^2(\Gamen)}+\|\der_1v\|^2_{L^2(\Gamex)}+\|\der_1 v\|^2_{L^2(\Om_L)}+\|w\|_{H^1(\Om_L)}^2\right)
  \end{split}
\end{equation}
for some constant $C_*>0$.

Notice that the function $G$ is already fixed in the previous steps, and we know that the maximum of $|G|$ over $\ol{\Om_L}$ depends only on $(\gam, \zeta_0, S_0, J, E_0)$. So the estimate constant $C_*$ can be fixed depending only on $(\gam, \zeta_0, S_0, J, E_0,L)$.

Therefore, we can estimate the right-hand side of \eqref{I1+I2 in H1 estimate} as
\begin{equation*}
\begin{split}
  &|{\tx{RHS of \eqref{I1+I2 in H1 estimate}}}|\\
  &\le
C_*\bar{\delta}\left(\|\der_1v\|^2_{L^2(\Gamen)}+\|\der_1v\|^2_{L^2(\Gamex)}+\|\der_1 v\|^2_{L^2(\Om_L)}+\|w\|_{H^1(\Om_L)}^2\right)+\int_{\Om_L}|f_1 G\der_1 v+f_2 w|\,d\rx.
  \end{split}
\end{equation*}
Finally, the proof of Lemma \ref{proposition-H1-apriori-estimate} can be completed by reducing $\bar{\delta}>0$ and applying the Cauchy-Schwarz inequality.\end{proof}

\begin{remark}
\label{remark: nozzle length L}
%Now we further discuss on the horizontal length of $\Om_L$.
In Lemma \ref{proposition-H1-apriori-estimate}, the essential condition to achieve an a priori $H^1$-estimate of $(v,w)$ is that the function $\kappa(=\frac{\bar u_1}{\us})$ needs to be close to 1(see \eqref{almost sonic condition1 full EP}). This condition is satisfied if $\kappa_0$ is fixed in $[1-d,1)$ for a sufficiently small constant $d>0$, and if the nozzle length $L$ is fixed to be suitably small. However, we should point out that this condition does not necessarily imply that the nozzle length $L$ must be small. In Appendix \ref{appendix:nozzle length}, we provide examples of $(\gam, J)$ for which the nozzle length $L$ can be large.
\end{remark}

Using Lemma \ref{proposition-H1-apriori-estimate}, the rest of the proof of Proposition \ref{theorem-wp of lbvp for system with sm coeff} can be given by adjusting the proof of Theorem \ref{theorem-1-full}. So we briefly explain how to complete the proof of Proposition \ref{theorem-wp of lbvp for system with sm coeff}.

For a constant $\eps>0$ fixed sufficiently small, consider the following boundary value problem, given as a singular perturbation of \eqref{lbvp-main general}:
\begin{equation}
\label{bvp-sing-pert}
  \begin{split}
  &\begin{cases}
  \eps \der_{111}v+\mfrak{L}_1^P(v, w)=f_1\quad&\tx{in $\Om_L$}\\
  \mfrak{L}_2(v, w)=f_2\quad&\tx{in $\Om_L$}
  \end{cases}\\
  &\begin{cases}
  v=0,\quad \der_1v=0\quad&\tx{on $\Gamen$}\\
  \der_2 v=0\quad&\tx{on $\Gamw$}\\
 \der_{11}v=0\quad&\tx{on $\Gamex$}
  \end{cases}\\
  &\begin{cases}
  \der_1 w=0\quad&\tx{on $\Gamen$}\\
  \der_2w=0\quad&\tx{on $\Gamw$}\\
  w=0\quad&\tx{on $\Gamex$}.
  \end{cases}
  \end{split}
\end{equation}

%%%%%%%%%%%%%%%%%%%%%%%%%%%%%%%%%%%%%%%%%%%%%%%%%%%%%%%%%%
%%%%%%%%%%%%%%%%%%%%%%%%%%%%%%%%%%%%%%%%%%%%%%%%%%%%%%%%%%
\newcommand \itersetent{\mcl{I}_{\rm ent}(r_1)}
\newcommand \itersetdelta{\mcl{I}_{\rm vor}(2\bar{\delta})\times \mcl{I}_{\rm pot}(2\bar{\delta})}
%%%%%%%%%%%%%%%%%%%%%%%%%%%%%%%%%%%%%%%%%%%%%%%%%%%%%%%%%%
%%%%%%%%%%%%%%%%%%%%%%%%%%%%%%%%%%%%%%%%%%%%%%%%%%%%%%%%%%

Using the idea given in the proof of Lemma \ref{lemma G:wp of singular pert prob-main}, and using Lemma \ref{proposition-H1-apriori-estimate}, one can establish an a priori $H^1$-estimate of a solution to \eqref{bvp-sing-pert} as follows:

\begin{lemma}
\label{lemma:wp of singular pert prob-main}
Assume that the background solution $(\bar u_1, \bar E)$, the nozzle length $L$, and the parameters $(r_2, r_3)$ are fixed to satisfy all the conditions stated in Lemma \ref{proposition-H1-apriori-estimate}.
Then, one can fix  two constants $\bar{\eps}>0$, and $\bar{\delta}>0$ sufficiently small depending only on the data so that whenever
\begin{equation*}
0<\eps<\bar{\eps},\quad \tx{and}\quad \max\{r_2, r_3\}\le 2\bar{\delta}
\end{equation*}
hold, if $(v^{\eps},w^{\eps})$ is a smooth solution to \eqref{bvp-sing-pert}, then it satisfies
\begin{equation}
\label{a priori estimate1 of vm and wm}
\begin{split}
  &\sqrt{\eps}\|\der_{11}v^{\eps}\|_{L^2(\Om_L)}
  +\|\der_1v^{\eps}\|_{L^2(\Gamen)}+\|Dv^{\eps}\|_{L^2(\Gamex)}
  +\|v^{\eps}\|_{H^1(\Om_L)}+\|w^{\eps}\|_{H^1(\Om_L)}
  \\
  &\le C\left(\|f_1\|_{L^2(\Om_L)}+\|f_2\|_{L^2(\Om_L)}\right).
  \end{split}
\end{equation}

\end{lemma}

Note that, as a second-order differential operator for $w$, $\mfrak{L}_2$ in \eqref{bvp-sing-pert} is uniformly elliptic in $\Om_L$. Since the a priori $H^1$ estimate of $(v^{\eps}, w^{\eps})$ is achieved in Lemma \ref{lemma:wp of singular pert prob-main}, the a priori $H^2$ estimate of $w^{\eps}$ can be easily obtained by applying a standard elliptic theory along with a bootstrap argument. Then one can apply Lemma \ref{G:lemma for pre H2 estimate of vm, part 2} and use the $H^2$-estimate of $w^{\eps}$ to establish an a priori $H^2$-estimate of $v^{\eps}$ away from $\Gamen$.

\begin{definition}
%[A weak solution to \eqref{lbvp-main general}]
\label{definition of weak solution}
For $v,w,V,W\in H^1(\Om_L)$, define two bilinear operators $\mcl{B}_1^P$ and $\mcl{B}_2$ by
\begin{equation*}
\begin{split}
\mcl{B}_1^P[(v,w),V]:=
\int_{\Om_L}-\left(a_{11}\der_1v\der_1V+a_{12}(\der_1v\der_2V
+\der_2v\der_1V)+\der_2v\der_2V+a\der_1vV\right)&\\
-(\der_1a_{11}\der_1v+\der_2a_{12}\der_1v+\der_1a_{12}\der_2 v)V+(b_1\der_1 w+b_0w)V\,&d\rx,
\end{split}
\end{equation*}
\begin{equation*}
 \mcl{B}_2[(v,w),W]:=-\int_{\Om_L} \nabla w\cdot \nabla W+(\frac{1}{\gam S_0\bar{\rho}^{\gam-2}}w-\frac{\bar u_1}{\gam S_0\bar{\rho}^{\gam-2}}\der_1 v)W\,d\rx.
\end{equation*}

We say that $(v,w)\in H^1(\Om_L)\times H^1(\Om_L)$ is {\emph{a weak solution}} to \eqref{lbvp-main general} if it satisfies
\begin{equation}
\label{weak-viscous-eqns}
    \begin{cases}
 \mcl{B}_1^P[(v,w),\phi_1]=\int_{\Om_L} f_1 \phi_1\,d\rx,\\
\mcl{B}_2[(v,w),\phi_2]=\int_{\Om_L} f_2 \phi_2\,d\rx
\end{cases}
\end{equation}
  for any test functions $\phi_1, \phi_2\in C^{\infty}(\ol{\Om_L})$ with $\phi_1$ vanishing near $\Gamen\cup \Gamex$, and $\phi_2$ vanishing near $\Gamex$.
\end{definition}

In order to find a weak solution $(v,w)\in [H^1(\Om)\cap H^2_{\rm loc}]^2$ to \eqref{lbvp-main general} in the sense of Definition \ref{definition of weak solution}, we take the following three steps.
\begin{itemize}
\item[(1)] Approximate $a_{ij}$, $a$, $b_1$, $b_0$, $f_1$, $f_2$ by a sequence of smooth functions $a_{ij}^{(\tau)}$, $a^{(\tau)}$, $b_1^{(\tau)}$,$b_0^{(\tau)}$, $f_1^{(\tau)}$, $f_2^{(\tau)}$ for $\tau\in(0, \bar{\tau}]$, respectively, for some small $\bar{\tau}>0$ such that
    \begin{equation*}
     \lim_{\tau\to 0+} \|(a_{ij}, a, b_1, b_0, f_1, f_2)-(a_{ij}^{(\tau)}, a^{(\tau)}, b_1^{(\tau)}, b_0^{(\tau)}, f_1^{(\tau)}, f_2^{(\tau)})\|_{H^{m-1}(\Om_L)}=0.
    \end{equation*}

\item[(2)] Apply the Galerkin approximation method.
\item[(3)] Apply Lemma \ref{lemma:wp of singular pert prob-main} and Lemma \ref{G:lemma for pre H2 estimate of vm, part 2}
\end{itemize}
Then we obtain a weak solution $(v,w)$ that satisfies
\begin{equation*}
  \|v\|_{H^1(\Om_L)}+\|w\|_{H^2(\Om_L)}\le  C(\|f_1\|_{L^2(\Om_L)}+\|f_2\|_{L^2(\Om_L)})
\end{equation*}
for some constant $C>0$ depending only on the data. In addition, it follows from Lemma \ref{G:lemma for pre H2 estimate of vm, part 2} that, for any $d\in(0, \frac{L}{4}]$, we have
\begin{equation*}
  \|v\|_{H^2(\Om_L\cap\{x_1>d\})}\le C_d(\|f_1\|_{H^1(\Om_L)}+\|f_2\|_{L^2(\Om_L)})
\end{equation*}
for some constant $C_d>0$ depending only on the data and $d$, provided that
\begin{equation*}
  \max\{r_2, r_3\}\le 2\bar{\delta}
\end{equation*}
for $\bar{\delta}>0$ fixed sufficiently small depending only on the data.

In order to achieve a global $H^m$-estimate of $(v,w)$, one can proceed as follows.
\begin{itemize}
\item[(1)] Apply a standard elliptic theory to get a global $H^k$-estimate of $w$ for $3\le k\le m$.
   \smallskip

\item[(2)] Next, employ the method of differential operator extension(see \S \ref{subsec:2.2}) and use the result from the previous step to get a global $H^k$-estimate of $v$.
     \smallskip

\item[(3)] And, use a bootstrap argument and employ the idea of the first step along with the estimate obtained from the second step to get a global $H^{k+1}$-estimate of $w$ for $k+1\le m$.
     \smallskip

\item[(4)] Finally, repeat the procedure described in the second step to get a global $H^{k+1}$-estimate of $v$ for $k+1\le m$.
     \smallskip
\end{itemize}
In each step prescribed in the above, the constant $\bar{\delta}$ can be adjusted as needed. This completes the proof of Proposition \ref{theorem-wp of lbvp for system with sm coeff}. \qed

\subsection{Well-posedness of Problem \ref{problem-HD}}
\label{seciton:pf of main theorem}
Now we are ready to prove Theorem \ref{theorem-HD} by applying Proposition \ref{theorem-wp of lbvp for system with sm coeff}.
%%%%%%%%%%%%%%%%%%%%%%%%%%%%%%%
\newcommand \itersetinrs {\iterseta}
\newcommand \itersetforvor{\iterT}
\newcommand \itersetforpot{\iterP}
%%%%%%%%%%%%%%%%%%%%%%%%%%%%%%%
\smallskip

\begin{proof}[Proof of Theorem \ref{theorem-HD}]
Let the constant $\bar{\delta}>0$ be from Proposition \ref{theorem-wp of lbvp for system with sm coeff}. Suppose that the constants $r_1,r_2$ and $r_3$ from the iterations sets $\iterP_{r_1}^m$ and $\iterV_{r_2}^{m+1}\times \iterseta$, given by Definition \ref{definition: iteration sets}, satisfy
\begin{equation}
\label{condition for r-0}
  \max\{r_1, r_2+r_3\}\le \bar{\delta}.
\end{equation}

\medskip

{\textbf{Step 1.}}\emph{Claim 1: One can fix the constants $r_1$, $r_2$ and $r_3$ so that, for any given $\tilT\in \iterP_{r_1}^m$ and $P=(\tphi,\tpsi, \tPsi)\in \iterV_{r_2}^{m+1}\times \iterseta$, Problem \ref{LBVP1 for iteration} has a unique solution $(\phi, \psi, \Psi)\in \iterV_{r_2}^{m+1}\times \iterseta$.}
\smallskip

Note that Corollary \ref{corollary:wp of mixed system in iteration} directly follows from Proposition \ref{theorem-wp of lbvp for system with sm coeff}. Therefore, for any given $\tilT\in \iterP_{r_1}^m$ and $P=(\tphi,\tpsi, \tPsi)\in \iterV_{r_2}^{m+1}\times \iterseta$, the linear boundary value problem \eqref{lbvp main} associated with $(\tilT, P)$ has a unique solution $(\psi, \Psi)\in H^m(\Om_L)\times H^m(\Om_L)$, and it satisfies the estimate
\begin{equation}
       \label{estimate:system}
      \|\psi\|_{H^m(\Om_L)}+\|\Psi\|_{H^m(\Om_L)}
      \le C_1\left(r_1+r_2+r_3^2+\mfrak P(S_0, E_{\rm en}, \om_{\rm en})\right)
    \end{equation}
    for $\mfrak P(S_0, E_{\rm en}, \om_{\rm en})$ defined by \eqref{definition-perturbation of bd}.

Due to the $H^{m-1}$-estimate and the compatibility conditions for the function $f_0^{(\til T,P)}$ stated in Lemma \ref{lemma on L_1}(f), the elliptic boundary value problem \eqref{lbvp for phi} has a unique solution $\phi\in H^{m+1}(\Om_L)$ that satisfies the estimate
\begin{equation}
    \label{estimate:phi}
      \|\phi\|_{H^{m+1}(\Om_L)}\le C_2r_1.
    \end{equation}
In \eqref{estimate:system} and \eqref{estimate:phi}, the  constants $C_1$ and $C_2$ in the estimates are fixed depending only on the data.
\smallskip

 For the rest of the proof, any constant in estimates is regarded to be fixed depending only on the data unless otherwise specified.
\smallskip

Set
\begin{equation}
\label{definition of C_*}
  C_*:=\max\{C_1, C_2\}.
\end{equation}

\begin{itemize}
\item[(i)] If the constant $r_1$ satisfies the inequality
\begin{equation}
\label{condition for r-1}
r_1\le \frac{r_2}{C_*},
\end{equation}
then we have
\begin{equation*}
\|\phi\|_{H^{m+1}(\Om_L)}\le r_2.
\end{equation*}

\item[(ii)] If the constants $r_1$, $r_2$, $r_3$ and $\mfrak P(S_0, E_{\rm en}, \om_{\rm en})$ satisfy
\begin{equation}
\label{condition for r-2}
\begin{split}
&C_*(r_1+r_2)\le \frac{r_3}{3},\quad C_*r_3\le \frac 13\quad\tx{and}\quad C_*\mfrak P(S_0, E_{\rm en}, \om_{\rm en})\le \frac{r_3}{3},
\end{split}
\end{equation}
then we have
\begin{equation}
\label{potential estimate fixed pt}
\|(\psi, \Psi)\|_{H^m(\Om_L)}\le r_3.
\end{equation}
\end{itemize}

We rewrite the equation $\der_2\mfrak{L}_1^P(\psi, \Psi)=\der_2 f_1^P$ in $\Om_L$ as
\begin{equation*}
\der_{22}\psi= f_1^P-\left(a_{11}^P+2a_{12}^P+a^P\der_1\psi+b_1^P\der_1\Psi+b_0^P\Psi\right) \quad\tx{in $\Om_L$}.
\end{equation*}
So we can directly check by using the properties (c) and (f) stated in Lemma \ref{lemma on L_1}, and  the slip boundary condition $\der_2\psi=0$ on $\Gamw$ that $\psi$ satisfies the compatibility condition
$$\der_2^k\psi=0\quad\tx{on $\Gamw$}$$
for $k=2i-1$, $i\in \mathbb{N}$ with $k<m$.
In addition, one can similarly check that
\begin{equation*}
  \der_2^k\Psi=0\,\,\tx{on $\Gamw$ for $k=2i-1$, $i\in \mathbb{N}$ with $k<m$},
\end{equation*}
and
\begin{equation*}
\der_2^{l}\phi=0\,\, \tx{on $\Gamw$ for $l=2(j-1)$, $j\in \mathbb{N}$ with $l<m+1$}.
\end{equation*}
Therefore if we fix $(r_1, r_2, r_3)$ to satisfy the conditions \eqref{condition for r-1} and \eqref{condition for r-2}, then we can conclude that
\begin{equation*}
(\phi, \psi, \Psi)\in \iterV_{r_2}^{m+1}\times \iterseta.
\end{equation*}
This verifies Claim 1.
\smallskip

{\textbf{Step 2.}} For a fixed $\tilT\in \iterP_{r_1}^m$, let us consider a nonlinear system for $(\phi, \psi, \Psi)$
\begin{equation}
\label{nlsystem-2nd order-o}
\begin{cases}
-\Delta \phi=f_0^{(\tilT, (\phi, \psi, \Psi))}\quad &\tx{in $\Om_L$},\\
\mfrak{L}_1^{(\phi, \psi, \Psi)}(\psi, \Psi)
=f_1^{(\phi, \psi, \Psi)}\quad &\tx{in $\Om_L$},\\
\mfrak{L}_2(\psi, \Psi)=f_2^{(\tilT, (\phi, \psi, \Psi))}\quad &\tx{in $\Om_L$},
\end{cases}
\end{equation}
with the boundary conditions
\begin{equation}
\label{BCs:nlsystem-2nd order-o}
\begin{split}
\quad \der_1\phi=0,\quad \psi(0, x_2)=\int_{-1}^{x_2}\om_{\rm en}(t)\,dt,\quad \der_1\Psi=E_{\rm en}-E_0\quad&\tx{on $\Gamen$},\\
\phi=0,\quad \der_2\psi=0,\quad \der_2\Psi=0\quad&\tx{on $\Gamw$,}\\
\phi=0,\quad \Psi=0\quad&\tx{on $\Gamex$}.
\end{split}
\end{equation}
This problem is obtained as an approximated nonlinear boundary value problem for $(\phi, \psi, \Psi)$ by replacing $T$ with $\tilT\in \iterP_{r_1}^m$ in \eqref{equation for psi}--\eqref{equation for phi}.
By applying the Schauder fixed point theorem, we can show that, for each $\tilT\in \iterP_{r_1}^m$, there exists at least one solution $(\phi, \psi, \Psi)\in \iterV_{r_2}^{m+1}\times\iterseta$ to the associated nonlinear problem of \eqref{nlsystem-2nd order-o} and \eqref{BCs:nlsystem-2nd order-o}.
Further details can be given by adjusting the proof of \cite[Theorems 1.6 and 1.7]{BDXX}, so we skip details. The only difference is that we should apply Proposition \ref{theorem-wp of lbvp for system with sm coeff} to guarantee a continuity of an iteration mapping with respect to an appropriately chosen norm.
\smallskip

{\textbf{Step 3.}} Given $\tilT\in \iterP_{r_1}^m$, let $(\phi^{(1)}, \psi^{(1)}, \Psi^{(1)})$ and $(\phi^{(2)}, \psi^{(2)}, \Psi^{(2)})$ be two solutions to the problem \eqref{nlsystem-2nd order-o}--\eqref{BCs:nlsystem-2nd order-o}, and suppose that both solutions are contained in $\iterV_{r_2}^{m+1}\times \iterseta$. Set
\begin{equation}
\label{diff of sol-1}
 (u,v,w):=(\phi_1, \psi_1, \Psi_1)- (\phi_2, \psi_2, \Psi_2)\quad\tx{in $\Om_L$},
\end{equation}
and
\begin{equation*}
d_1:=\|u\|_{H^2(\Om_L)},\quad d_2:=\|(v,w)\|_{H^1(\Om_L)},\quad d:=d_1+d_2.
\end{equation*}
Then we have
\begin{equation}
\label{nlsystem-2nd order}
\begin{cases}
-\Delta u=F_0\quad &\tx{in $\Om_L$},\\
\mfrak{L}_1^{ (\phi_1, \psi_1, \Psi_1)}(v,w)
=F_1\quad &\tx{in $\Om_L$},\\
\mfrak{L}_2(v,w)=F_2\quad &\tx{in $\Om_L$},
\end{cases}
\end{equation}
with the boundary conditions
\begin{equation}
\label{BCs:nlsystem-2nd order-u}
\begin{split}
\der_1u=0,\quad v(0, x_2)=0,\quad \der_1w=0\quad&\tx{on $\Gamen$},\\
u=0,\quad \der_2 v=0,\quad \der_2 w=0\quad&\tx{on $\Gamw$},\\
u=0,\quad w=0\quad&\tx{on $\Gamex$},
\end{split}
\end{equation}
for
\begin{equation*}
  \begin{split}
  F_0&:=f_0^{(\tilT, (\phi_1, \psi_1, \Psi_1))}-f_0^{(\tilT, (\phi_2, \psi_2, \Psi_2))},\\
  F_1&:=f_1^{ (\phi_1, \psi_1, \Psi_1)}-f_1^{ (\phi_2, \psi_2, \Psi_2)}+(\mfrak{L}_1^{(\phi_2, \psi_2, \Psi_2)}-\mfrak{L}_1^{(\phi_1, \psi_1, \Psi_1)})(\psi_2, \Psi_2),\\
  F_2&:=f_2^{(\tilT, (\phi_1, \psi_1, \Psi_1))}-f_2^{(\tilT, (\phi_2, \psi_2, \Psi_2))}.
  \end{split}
\end{equation*}
By adjusting the argument in Step 1 from \cite[Proof of Theorem 1.7]{BDXX} together with Lemma \ref{proposition-H1-apriori-estimate}, we can show that
\begin{align}
\label{contraction-1}
 & d_1\le C_{\sharp}r_1(d_1+d_2),\\
 \label{contraction-new}
 &  d_2\le C_{\sharp}\left(d_1+(r_1+r_2+r_3+\mfrak P(S_0, E_{\rm en}, \om_{\rm en}))d_2\right)
\end{align}
for some constant $C_{\sharp}>0$.
The estimate \eqref{contraction-1} is easily obtained by investigating the linear elliptic boundary value problem for $u$, which is a part of the problem \eqref{nlsystem-2nd order}--\eqref{BCs:nlsystem-2nd order-u}. The estimate \eqref{contraction-new} is obtained by applying Lemma  \ref{proposition-H1-apriori-estimate} to the boundary value problem for $(v,w)$ stated in \eqref{nlsystem-2nd order}--\eqref{BCs:nlsystem-2nd order-u}.

If the constant $r_1$ satisfies the inequality
\begin{equation}
\label{condition for r-3}
  r_1\le \frac{1}{2C_{\sharp}},
\end{equation}
then the estimate \eqref{contraction-1} immediately yields that
\begin{equation}\label{contraction-2}
  d_1\le 2C_{\sharp}r_1d_2.
\end{equation}
We combine \eqref{contraction-2} with \eqref{contraction-new}
to get
\begin{equation}\label{contraction-3}
  d_2\le C_{\sharp}\left((2C_{\sharp}+1)r_1+r_2+r_3+\mfrak P(S_0, E_{\rm en}, \om_{\rm en})\right)d_2.
\end{equation}
Therefore, if $(r_1, r_2, r_3)$ are fixed sufficiently small and if $\mfrak P(S_0, E_{\rm en}, \om_{\rm en})$ is sufficiently small to satisfy the condition
\begin{equation}
\label{condition for r-4}
  C_{\sharp}\left((2C_{\sharp}+1)r_1+r_2+r_3+\mfrak P(S_0, E_{\rm en}, \om_{\rm en})\right)<1,
\end{equation}
then we obtain that $d_1=d_2=0$, from which it follows that
\begin{equation*}
  (\phi_1, \psi_1, \Psi_1)=(\phi_2, \psi_2, \Psi_2)\quad\tx{in $\ol{\Om_L}$}.
\end{equation*}
\smallskip

{\textbf{Step 4.}} In the following, we assume that all conditions \eqref{condition for r-0}, \eqref{condition for r-1}, \eqref{condition for r-2}, \eqref{condition for r-3} and \eqref{condition for r-4} are satisfied.

For a fixed $\tilT\in \iterP_{r_1}^m$, let ${\bf m}(\Psi, \nabla\psi, \nabla^{\perp}\phi)$ be the momentum density field associated with $\tilT$ in the sense of Definition \ref{definition: momentum density field}. By the results obtained in the previous three steps, the vector field ${\bf m}(\Psi, \nabla\psi, \nabla^{\perp}\phi)$ is well defined in $\Om_L$. Note that it is divergence-free, that is, it satisfies the equation \eqref{div-free}. Next, we prove the well-posedness of Problem \ref{problem-transport equation}.
%(see \S \ref{subsection:framework}).

\begin{lemma}
\label{lemma:wp for ivp, entropy}
Assuming that the conditions \eqref{condition for r-0}, \eqref{condition for r-1} and \eqref{condition for r-2} are satisfied, one can fix a constant $\bar{r}_3>0$ depending only on the data so that if the inequality
\begin{equation}
\label{condition for r-6}
0<r_3\le \bar{r}_3
\end{equation}
holds, then, for each $\tilT\in \iterP_{r_1}^m$, Problem \ref{problem-transport equation} has a unique solution $T\in H^m(\Om_L)$ that satisfies
\begin{equation}
\label{estimate of T}
  \|T\|_{H^m(\Om_L)}\le C_{\flat}\|S_{\rm en}-S_0\|_{C^m([-1,1])}
\end{equation}
for some constant $C_{\flat}>0$ fixed depending only on the data.
\end{lemma}
This lemma is analogous to \cite[Lemma 3.5]{BDXX} except that there is one significant difference in its proof. More details are given after the proof of Theorem \ref{theorem-HD}.

{\textbf{Step 5.}}
In addition to the conditions \eqref{condition for r-0}, \eqref{condition for r-1}, \eqref{condition for r-2} and \eqref{condition for r-6}, if the inequality
\begin{equation}
\label{condition for r-7}
  \mfrak P(S_{\rm en}, E_{\rm en}, \om_{\rm en})\le \frac{r_1}{2C_{\flat}}
\end{equation}
holds, then it follows from Lemma \ref{lemma:wp for ivp, entropy} that, for each $\tilT\in \iterP_{r_1}^m$, the solution $T$ to Problem \ref{problem-transport equation} satisfies the estimate
\begin{equation*}
  \|T\|_{H^m(\Om_L)}\le \frac{r_1}{2}.
\end{equation*}
Furthermore, by using the representation \eqref{expression of T} and the compatibility conditions for $S_{\rm en}$ stated in Condition \ref{conditon:1}(i), it can be checked that the solution $T$ satisfies the compatibility conditions
\begin{equation*}
  \der_2^kT=0\quad\tx{on $\Gamw$}
\end{equation*}
for $k=2i-1$, $i\in \mathbb{N}$ with $k<m$.
This implies $T\in \iterP_{r_1}^m$. Then we can apply the Schauder fixed point theorem and the uniqueness of a solution to problem \ref{problem-transport equation} to conclude that the nonlinear boundary value problem of \eqref{equation for psi}--\eqref{BC for T} has at least one solution $(T, \phi, \psi, \Psi)\in \iterP_{r_1}^m\times \iterV_{r_2}^{m+1}\times \iterseta$ provided that $(r_1, r_2, r_3, \mcl{P}(S_{\rm en}, E_{\rm en}, \om_{\rm en}) )$ satisfies all the conditions \eqref{condition for r-0}, \eqref{condition for r-1}, \eqref{condition for r-2}, \eqref{condition for r-6}, and \eqref{condition for r-7}. To verify this statement, one can simply follow the argument in the proof of \cite[Theorem 1.7]{BDXX}. For further details, see \S 3.2.2 in \cite{BDXX}.

\medskip

{\textbf{Step 6.}} For the solution $(T,\phi, \psi, \Psi)$ to the boundary value problem of \eqref{equation for psi}--\eqref{BC for T}, let us set
\begin{equation*}
  (S, \vphi, \Phi):=(S_0, \bar{\vphi}, \bar{\Phi})+(T, \psi, \Psi)\quad\tx{in $\Om_L$}.
\end{equation*}
For the constant $\bar{r}_3$ from Lemma \ref{lemma:wp for ivp, entropy}, one can fix a constant $\hat{r}_3\in(0, \bar{r}_3]$ depending only on the data so that whenever
\begin{equation}\label{condition for r_3-equiv}
  r_3\le \hat{r}_3,
\end{equation}
we can combine the first inequality in \eqref{condition for r-2} with \eqref{condition for r_3-equiv} to ensure that $(\vphi, \phi, \Phi, S)$ satisfy all the conditions \eqref{smallness condition-1}--\eqref{smallness condition-3}. Then it follows that $(\vphi, \phi, \Phi, S)$ is a solution to Problem \ref{problem-HD}.
\smallskip

Now, we fix three constants $r_1$, $r_2$ and $r_3$ as follows:
\begin{equation}
\label{choice of r's}
  \begin{split}
  r_3&=\frac{\mfrak P(S_{\rm en}, E_{\rm en}, \om_{\rm en})}{\max\{\frac{1}{3C_*}, \frac{C_*+1}{12C_*C_{\flat}}\}}=:\kappa \mfrak P(S_{\rm en}, E_{\rm en}, \om_{\rm en}),\\
  r_1&=\frac{r_3}{6C_*(2C_*+1)}=\frac{\kappa}{6C_*(2C_*+1)} \mfrak P(S_{\rm en}, E_{\rm en}, \om_{\rm en})=:\mu_1\mfrak P(S_{\rm en}, E_{\rm en}, \om_{\rm en}),\\
  r_2&=\frac{r_3}{3(2C_*+1)}=\frac{\kappa}{3(2C_*+1)} \mcl{P}(S_{\rm en}, E_{\rm en}, \om_{\rm en})=:\mu_2\mfrak P(S_{\rm en}, E_{\rm en}, \om_{\rm en}),
  \end{split}
\end{equation}
for the constants $C_*$ and $C_{\flat}$ from \eqref{definition of C_*} and \eqref{estimate of T}, respectively. If the term $\mfrak P(S_{\rm en}, E_{\rm en}, \om_{\rm en})$ satisfies the inequality
\begin{equation}
\label{condition for data perturbation-2}
\begin{split}
  &\mfrak P(S_{\rm en}, E_{\rm en}, \om_{\rm en})\\
  &\le \min\left\{\frac{1}{3\kappa C_*}, \frac{\bar{\delta}}{\mu_1+\mu_2+\kappa}, \frac{1}{2C_{\sharp}\mu_1}, \frac{\hat{r}_3}{\kappa}, \frac{1}{2C_{\sharp}((2C_{\sharp}+1)\mu_1+\mu_2+\kappa+1)}\right\}=:\sigma_*
  \end{split}
\end{equation}
for the constants $\bar{\delta}$, $C_{\sharp}$ and $\hat{r}_3$ from \eqref{condition for r-0}, \eqref{contraction-1}(or \eqref{contraction-new}) and \eqref{condition for r_3-equiv}, respectively, then all the conditions \eqref{condition for r-0}, \eqref{condition for r-1}, \eqref{condition for r-2}, \eqref{condition for r-3}, \eqref{condition for r-4}, \eqref{condition for r-7} and \eqref{condition for r_3-equiv} are satisfied. Since $(T, \phi, \psi, \Psi)$ is contained in the iteration set $\iterP_{r_1}^m\times \iterV_{r_2}^{m+1}\times \iterseta$, it follows from \eqref{choice of r's} that $(\vphi, \phi, \Phi, S)$ satisfies the estimate \eqref{solution estimate HD}.
\smallskip

Suppose that $(\vphi^{(1)}, \Phi^{(1)}, \phi^{(1)}, S^{(1)})$ and $(\vphi^{(2)}, \Phi^{(2)}, \phi^{(2)}, S^{(2)})$ are two solutions to Problem \ref{problem-HD}. In addition, suppose that they satisfy the estimate \eqref{solution estimate HD} with the term $\mcl{P}(S_{\rm en}, E_{\rm en}, \om_{\rm en}) $ satisfying the condition \eqref{condition for data perturbation-2}. For each $j=1$ and $2$, set
\begin{equation*}
  (\psi^{(j)}, \Psi^{(j)}, T^{(j)}):=
  (\vphi^{(j)}, \Phi^{(j)}, S^{(j)})-(\bar{\vphi}, \bar{\Phi}, S_0),\quad P^{(j)}:=(\phi^{(j)},\psi^{(j)}, \Psi^{(j)}).
\end{equation*}
Define
\begin{equation*}
  (\til u, \til v, \til w, \wtil{Y}):=(\phi^{(1)}, \vphi^{(1)}, \Phi^{(1)}, T^{(1)})-(\phi^{(2)}, \vphi^{(2)}, \Phi^{(2)}, T^{(2)}),
\end{equation*}
and
\begin{equation*}
\sigma:=\mfrak P(S_{\rm en}, E_{\rm en}, \om_{\rm en}).
\end{equation*}
By adjusting the argument in Step 3, one can show that
\begin{align}
\label{estimate:final-0}
  &\|\til u\|_{H^2(\Om_L)}\le C\left(\|\wtil{Y}\|_{H^1(\Om_L)}+\sigma(\|\til u\|_{H^2(\Om_L)}+\|(\til v, \til w)\|_{H^1(\Om_L)})\right),\\
\label{estimate:final-1}
  &\|(\til v, \til w)\|_{H^1(\Om_L)}\le C\left(\|\til u\|_{H^2(\Om_L)}+\sigma\|(\til v, \til w) \|_{H^1(\Om_L)}\right).
\end{align}

By following the argument in the proof of \cite[Theorem 1.7]{BDXX} (see \S 3.2.2 in \cite{BDXX}), one can find a constant $\hat{\sigma}\in (0, \sigma_*]$ depending only on the data so that if $\sigma\le \hat{\sigma}$, then it holds that
\begin{equation}
\label{estimate:final-2}
  \|\wtil{Y}\|_{H^1(\Om_L)}\le C\sigma \left(\|\wtil{Y}\|_{H^1(\Om_L)}+\|(\til v, \til w)\|_{H^1(\Om_L)}+\|\til u\|_{H^2(\Om_L)}\right).
\end{equation}
Therefore, one can fix a constant $\bar{\sigma}\in(0,\hat{\sigma}]$ sufficiently small depending only on the data so if
$\sigma\le \bar{\sigma}$, then it follows from the estimates \eqref{estimate:final-0}--\eqref{estimate:final-2} that $$(\til u, \til v, \til w, \wtil{Y})=(0,0,0,0)\quad\tx{in $\Om_L$}.$$
This proves the uniqueness of a solution to Problem \ref{problem-HD}.

\smallskip

{\textbf{Step 7.}}
For a solution $(\vphi, \Phi, \phi, S)$ to Problem \ref{problem-HD}, define its associated Mach number $M(\vphi, \Phi, \phi, S)$ by
\begin{equation*}
  M(\vphi, \Phi, \phi, S):=\frac{|\nabla\vphi+\nabla^{\perp}\phi|}{\sqrt{\gam S\varrho^{\gam-1}(S, \Phi, \nabla\vphi, \nabla^{\perp}\phi)}}
\end{equation*}
for $\varrho(S, \Phi, \nabla\vphi, \nabla^{\perp}\phi)$ given by \eqref{new system with H-decomp1}.

Note that $T(=S-S_0)\in \iterP_{r_1}^m$ and $P:=(\phi,\psi, \Psi)(=(\phi, \vphi, \Phi)-(0, \bar{\vphi}, \bar{\Phi}))\in \iterV_{r_2}^{m+1}\times \iterseta$. For $(a_{ij}^P)_{i,j=1,2}$ given by Definition \ref{definition:approx coeff and fs}, define a function $D:\Om_L\rightarrow \R$ by
\begin{equation*}
  D(\rx)=\det \begin{pmatrix}a_{11}^P(\rx) & a_{12}^P(\rx)\\ a_{12}^P(\rx) & 1 \end{pmatrix}\quad\tx{for $\rx=(x_1, x_2)\in \Om_L$}.
\end{equation*}
It follows from Lemma \ref{lemma on L_1} ($\tx{f}_2$) that there exists a unique function $\gs^P:[-1,1]\rightarrow (0,L)$ satisfying that
\begin{equation}\label{signs of D}
  D(x_1, x_2)\begin{cases}
  >0\quad&\mbox{for $x_1<\gs^P(x_2)$},\\
  =0\quad&\mbox{for $x_1=\gs^P(x_2)$},\\
  <0\quad&\mbox{for $x_1>\gs^P(x_2)$}.
  \end{cases}
\end{equation}
By a direct computation with Definitions \ref{definition:coefficints-iter} and \ref{definition:coefficients-nonlinear}, it can be directly checked that
\begin{equation*}
  D(\rx)=\frac{\left(\gam S\varrho^{\gam-1}(S, \Phi, \nabla\vphi, \nabla^{\perp}\phi)\right)^2}{A_{22}^2(\rx, \Psi, \nabla\psi, \nabla^{\perp}\phi)}\left(1-M^2(\vphi, \Phi, \phi, S)\right).
\end{equation*}
Therefore, \eqref{signs of D} is equivalent to the following:
\begin{equation}\label{mag. of M}
M(\vphi, \Phi, \phi, S)\begin{cases}
  <1\quad&\mbox{for $x_1<\gs^P(x_2)$},\\
  =1\quad&\mbox{for $x_1=\gs^P(x_2)$},\\
  >1\quad&\mbox{for $x_1>\gs^P(x_2)$}.
  \end{cases}
\end{equation}
Define a function $\fsonic$ by
\begin{equation*}
  \fsonic(x_2):=\gs^P(x_2)\quad\tx{for $-1\le x_2\le 1$}.
\end{equation*}
Then, it follows from \eqref{estimate of gs}, \eqref{estimate:phi}, \eqref{potential estimate fixed pt} and \eqref{choice of r's} that the {\emph{sonic interface function}} $\fsonic$ satisfies the estimate \eqref{estimate of sonic boundary pt}. This completes the proof of Theorem \ref{theorem-HD}. \end{proof}

\begin{proof}[Proof of Lemma \ref{lemma:wp for ivp, entropy}]
First of all, we fix a constant $\bar{r}_3$ to satisfy
\begin{equation*}
  \bar{r}_3\le \min\left\{1, \frac{C_*}{C_{\sharp}(2C_{\sharp}+3C_*+3)}\right\}
\end{equation*}
so that whenever the inequality $0<r_3\le \bar{r}_3$ holds, the conditions \eqref{condition for r-3} and \eqref{condition for r-4} are satisfied by \eqref{condition for r-0}, \eqref{condition for r-1} and \eqref{condition for r-2}.
\smallskip

For a fixed $\tilT\in \iterP_{r_1}^m$, let ${\bf m}(\Psi, \nabla\psi, \nabla^{\perp}\phi)$ be the approximated momentum density field associated with $\tilT$ in the sense of Definition \ref{definition: momentum density field}. Since the solution $(\phi,\psi, \Psi)$ to the boundary value problem \eqref{nlsystem-2nd order-o}--\eqref{BCs:nlsystem-2nd order-o} lies in the iteration set $\iterV_{r_2}^{m+1}\times \iterseta$, by using the first inequality stated in \eqref{condition for r-2}, one can directly check the estimate
\begin{equation}
\label{estimate for m-1}
  \|{\bf m}(\Psi, \nabla\psi, \nabla^{\perp}\phi)-{\bf m}(0, {\bf 0}, {\bf 0})\|_{H^{m-1}(\Om_L)}\le Cr_3.
\end{equation}

Note that ${\bf m}(0, {\bf 0}, {\bf 0})=J\hat{\bf e}_1$. Next, we shall improve the regularity of ${\bf m}(\Psi, \nabla\psi, \nabla^{\perp}\phi)$ near the entrance boundary $\Gamen$.
\smallskip

{\emph{Claim: In $\Om_L\cap\{x_1<\frac{\ls}{2}\}$(see Lemma \ref{lemma-1d-full EP} for the definition of the constant $\ls$), the following estimate holds for some constant $C>0$ fixed depending only on the data:
\begin{equation}
\label{estimate for m-2}
\begin{split}
  &\|{\bf m}(\Psi, \nabla\psi, \nabla^{\perp}\phi)-{\bf m}(0, {\bf 0}, {\bf 0})\|_{H^m(\Om_L\cap\{x_1<\frac{\ls}{2}\})} \le Cr_3.
  \end{split}
\end{equation}
}}

{\emph{Verification of the claim:}}
Define a function $\breve{\psi}$ by
\begin{equation*}
  \breve{\psi}(x_1,x_2):=\psi(x_1, x_2)-\int_{-1}^{x_2}\om_{\rm en}(t)\,dt.
\end{equation*}
Then it follows from the boundary condition $\psi(0, x_2)=\int_{-1}^{x_2}\om_{\rm en}(t)\,dt$ that
\begin{equation*}
\breve{\psi}=0\quad\tx{on $\Gamen$}.
\end{equation*}
Next, we rewrite the equation $\mfrak{L}_1^{(\phi,\psi, \Psi)}(\psi, \Psi)=f_1^{(\phi, \psi, \Psi)}$ in terms of $\breve{\psi}$ as
\begin{equation}
\label{equation for u}
  a_{11}\der_{11}\breve{\psi}+2a_{12}\der_{12}\breve{\psi}+\der_{22}\breve{\psi}=F\quad\tx{in $\Om_L$}
\end{equation}
for
\begin{equation*}
  \begin{split}
  &F:=f_1^{(\phi, \psi, \Psi)}-\mfrak{L}_1^{(\phi, \psi, \Psi)}\left(\int_{-1}^{x_2}\om_{\rm en}(t)\,dt, \Psi\right)-a^{(\phi, \psi, \Psi)}\der_1 \breve{\psi},\\
  &(a_{11}, a_{12}):=(a_{11}^{(\phi, \psi, \Psi)}, a_{12}^{(\phi, \psi, \Psi)}).
  \end{split}
\end{equation*}
By using the compatibility conditions of $\om_{\rm en}$ stated in Condition \ref{conditon:1} and Lemma \ref{lemma on L_1}, we can directly check the following properties:
\begin{equation*}
\begin{split}
&\der_2 \breve{\psi}=0\quad\tx{on $\Gamw$},\\
&\der_2 a_{11}=0,\,\,a_{12}=\der_2^2a_{12}=0,\,\,\der_2 F=0\quad\tx{on $\Gamw$}.
\end{split}
\end{equation*}
Most importantly, it follows from Lemma \ref{lemma on L_1} ($\tx{f}_3$) that
the equation \eqref{equation for u} is uniformly elliptic in $\Om_L\cap\{x_1<\frac 34 \ls\}$. Then we can use the estimates given in Lemma \ref{lemma on L_1} (e) and the estimate \eqref{potential estimate fixed pt} to show that
\begin{equation*}
  \|\psi\|_{H^{m+1}(\Om_L\cap\{x_1<\frac{\ls}{2}\})}\le Cr_3.
\end{equation*}
This together with \eqref{definition: vector field m} yields
\begin{equation}
\label{estimate for m-2-2}
  \|{\bf m}(\Psi, \nabla\psi, \nabla^{\perp}\phi)-{\bf m}(0, {\bf 0}, {\bf 0})\|_{H^m(\Om_L\cap\{x_1<\frac{\ls}{2}\})}
  \le Cr_3.
\end{equation}
The claim is verified.
\medskip

\newcommand \vtheta{\vartheta}
For ${\bf m}(\rx):={\bf m}(\Psi(\rx), \nabla\psi(\rx), \nabla^{\perp}\phi(\rx))$, define a function $\vtheta:\ol{\Om_L}\rightarrow \R$ by
\begin{equation*}
  \vtheta(x_1, x_2):=\int_{-1}^{x_2} {\bf m}(x_1, t)\cdot \hat {\bf e}_1\,dt.
\end{equation*}
Then, one can directly check from \eqref{div-free} that
\begin{equation}
\label{derivative of vtheta}
  (\der_{x_1}\vtheta, \der_{x_2}\vtheta)=(-{\bf m}\cdot \hat {\bf e}_2, {\bf m}\cdot \hat {\bf e}_1)\quad\tx{in $\ol{\Om_L}$}.
\end{equation}
Owing to the estimate \eqref{estimate for m-1}, we can further reduce the constant $\bar{r}_3$ depending only on the data so that if $r_3\in (0,\bar{r}_3)$, then we have
    \begin{equation}
    \label{estimate of vtheta-2}
      \der_2\vtheta(\rx)\ge \frac J2\quad\tx{for all $\rx\in \ol{\Om_L}$}.
    \end{equation}
It follows from \eqref{derivative of vtheta} and the boundary conditions stated in \eqref{BCs:nlsystem-2nd order-o} that $\vartheta$ satisfies the boundary conditions
\begin{equation*}
\der_{x_1}\vartheta(x_1, \pm 1)=0\quad\tx{for all $0<x_1<L$. }
\end{equation*}
So we have
\begin{equation*}
  \vtheta(x_1,-1)=0\quad\tx{and}\quad \vartheta(x_1,1)=\vartheta(0,1)=:\bar{\vartheta}\quad\tx{for all $0\le x_1\le L$}.
\end{equation*}
Moreover, it follows from the estimate \eqref{estimate for m-2} and the Morrey's inequality that
\begin{equation}
\label{estimate of vtheta-1}
  \|\vartheta-J(1+x_2)\|_{C^{m-1}(\ol{\Om_L\cap\{x_1<\frac{\ls}{3}\}})}\le
  Cr_3.
\end{equation}

Then we can define a {\emph{Lagrangian mapping}} $\mathscr{L}:\ol{\Om_L}\rightarrow [-1,1]$ by the implicit relation
\begin{equation}
\label{definition:vtheta}
  \vtheta(0, \mathscr{L}(\rx)):=\vtheta(\rx).
\end{equation}
Then
\begin{equation}
\label{expression of T}
  T(\rx):=(S_{\rm en}-S_0)\circ \mathscr{L}(\rx)
\end{equation}
solves Problem \ref{problem-transport equation}.

\newcommand \lmap{\mathscr{L}}
By direct calculations using the definition of $\lmap$ given by \eqref{definition:vtheta} together with the estimates  \eqref{estimate for m-1}, \eqref{estimate of vtheta-2} and \eqref{estimate of vtheta-1},
we can directly check that
\begin{equation*}
  \|\lmap\|_{H^{m-1}(\Om_L)}\le C.
\end{equation*}
Finally, we need to check if $D^m\lmap\in L^2(\Om_L)$.
\medskip

Suppose that $m=4$.
By a lengthy but straightforward computation, we get
\begin{equation*}
 D^4\lmap(\rx)=\frac{{\sum_{j=1}^5 \mathscr{D}_j}}{\der_{x_2}\vtheta(0, \lmap(\rx))}
\end{equation*}
for
\begin{equation*}
\begin{split}
 \mathscr{D}_1&=D^4\vtheta,\quad
 \mathscr{D}_2=-6\der_{x_2}^3\vtheta(0, \lmap)(D\lmap)^2D^2\lmap,\\
 \mathscr{D}_3&=-3\der_{x_2}^2\vtheta(0, \lmap)(D^2\lmap)^2,\quad
 \mathscr{D}_4=-4\der_{x_2}^2\vtheta(0, \lmap)D\lmap D^3\lmap,\\
 \mathscr{D}_5&=-\der_{x_2}^4\vtheta(0, \lmap)(D\lmap)^4.
 \end{split}
\end{equation*}
By using \eqref{derivative of vtheta} and applying the Sobolev inequality, we can directly estimate the terms $\mathscr{D}_j$ for $j=1,2,3,4$ as follows:
 \begin{equation*}
 \begin{split}
 &\|\mathscr{D}_1\|_{L^2(\Om_L)}\le C\|{\bf m}\|_{H^3(\Om_L)},\\
 &\|\mathscr{D}_j\|_{L^2(\Om_L)}\le C\|\vtheta\|_{C^3(\ol{\Om_L\cap\{x_1<\frac{\ls}{3}\}})}
 \|\lmap\|^3_{H^3(\Om_L)}\quad\tx{for $j=2,3,4$.}
 \end{split}
 \end{equation*}
Moreover, applying the generalized Sobolev inequality, we have
\begin{equation*}
  \|\mathscr{D}_5\|_{L^2(\Om_L)}
  \le C\|\lmap\|_{H^3(\Om_L)}^3\|\der_{x_2}^4\vtheta(0, \lmap)\|_{L^2(\Om_L)}.
\end{equation*}
For each $t\in (0, L)$, the function $\lmap(t, \cdot): [-1,1]\rightarrow [-1,1]$ is $C^1$-diffeomorphic. By \eqref{estimate for m-2} and \eqref{derivative of vtheta}, the function $\der_{x_2}^4\vtheta(0, \lmap(t,\cdot))$ belongs to $L^2((-1,1))$ owing to the trace theorem. Then we apply the Fubini's theorem and use \eqref{estimate for m-2}--\eqref{estimate of vtheta-2} to get
\begin{equation*}
\begin{split}
  \|\der_{x_2}^4\vtheta(0, \lmap)\|_{L^2(\Om_L)}^2
  &=\int_{0}^L\int_{\Om_L\cap\{x_1=t\}}|\der_{x_2}^4\vtheta(0, \lmap(x_1,x_2))|^2\,dx_2dt\\
  &=\int_0^L\int_{\Om_L\cap\{x_1=t\}}\frac{|\der_{x_2}^4\vtheta(0, y)|^2}{\der_{x_2}\lmap(t, \lmap^{-1}(t,y))}\,dy\,dt\\
  &\le C
  \end{split}
\end{equation*}
where $\lmap^{-1}$ represents the inverse of $\lmap(t,\cdot)$ as a mapping of $x_2\in[-1,1]$.
For the case of $m>4$, it can be similarly checked that $D^m\lmap\in L^2(\Om_L)$. Then, the estimate \eqref{estimate of T} follows directly.

\end{proof}

\appendix

\section{A further discussion on the nozzle length $L$}
\label{appendix:nozzle length}

Returning to the proof of Lemma \ref{proposition-H1-apriori-estimate}, we provide examples of $(\gam, J)$ for which the length of the nozzle $L$ can be large while the condition \eqref{almost sonic condition1 full EP} is satisfied.
\smallskip

By \eqref{L-computation-n} and \eqref{definition of lambda}, one can express the nozzle length $L$ as
\begin{equation}
\label{new expression of L}
L:=\sqrt{\frac{h_0^3}{2}}J^{\frac{\gam-2}{\gam+1}}\sqrt{\lambda(\kappa_0,\kappa_L)}.
\end{equation}
Define
\begin{equation*}
  \mfrak{p}:=\frac{\sqrt 2}{2}h_0^{-\frac 32}\mcl{H}(\kappa)((\gam-1)\kappa^{\gam+1}+\eta(\kappa^{\gam+1}-1)+2)\kappa^{-\eta}h_0^{-\eta},
\end{equation*}
and
\begin{equation*}
  \hat{\alp}:=\frac{\alp}{J^{\frac{2-\gam+2\eta}{\gam+1}}\mfrak{p}}.
\end{equation*}
for $\alp$ given by \eqref{definition of beta*}.
\medskip

{\textbf{Example 1.}} Returning to Step 4 of the proof of Lemma \ref{proposition-H1-apriori-estimate},
where the parameter $\eta$ in \eqref{definition of G} is fixed as $\eta=\frac 34\gam$, we take a closer look at how the parameter $\bJ$ is fixed.

By \eqref{new expression of alp} and \eqref{expression of alp}, we can roughly express $\hat{\alp}$ as
\begin{equation*}
  \hat{\alp}=
  1-\mu_1(\kappa,h_0,\gam)J^{\frac{\gam}{2(\gam+1)}}
  -\mu_2(\kappa,h_0,\gam)\lambda(\kappa_0, \kappa_L)J^{-\frac{\gam}{2(\gam+1)}}\left(J^{\frac{\gam}{2(\gam+1)}}+\mu_3(\kappa, h_0, \gam)\right)^2.
\end{equation*}
Next, we express $\lambda(\kappa_0, \kappa_L)$ as
\begin{equation}
\label{new expression of lambda}
\lambda(\kappa_0, \kappa_L)=\eps_0J^{\frac{\gam}{2(\gam+1)}}.
\end{equation}
Then we have
\begin{equation*}
  \hat{\alp}=1-\mu_1(\kappa,h_0,\gam)J^{\frac{\gam}{2(\gam+1)}}
  -\mu_2(\kappa,h_0,\gam)\eps_0\left(J^{\frac{\gam}{2(\gam+1)}}+\mu_3(\kappa, h_0, \gam)\right)^2.
\end{equation*}
Assume that
\begin{equation*}
  |\kappa-1|\le q_0
\end{equation*}
for some constant $q_0>0$. Then one can fix a constant $\bar{\mu}>0$ depending only on $(\gam, S_0, \zeta_0, q_0)$ to satisfy
\begin{equation*}
 \max_{{j=1,2,3,}\atop{|\kappa-1|\le q_0}}  |\mu_j(\kappa, h_0, \gamma)|\le \bar{\mu}
\end{equation*}
so we have
\begin{equation*}
  \hat{\alp}\ge 1-\bar{\mu}J^{\frac{\gam}{2(\gam+1)}}-
  \bar{\mu}\eps_0\left(J^{\frac{\gam}{2(\gam+1)}}+\bar{\mu}\right)^2.
\end{equation*}
Now we fix a constant $\bJ>0$ sufficiently small depending only on $(\gam, S_0, \zeta_0, q_0)$ such that, whenever the background momentum density $J$ satisfies $0<J\le \bJ$, it holds that
\begin{equation*}
  1-\bar{\mu}J^{\frac{\gam}{2(\gam+1)}}\ge \frac 23.
\end{equation*}
Then, we fix a constant $\eps_0>0$ sufficiently small depending only on $(\gam, S_0, \zeta_0, q_0)$ so that if $J\in(0, \bJ]$, then we finally obtain that
\begin{equation}
\label{lower bd of hat alp}
  \hat{\alp}\ge \frac 13,
\end{equation}
and this yields the result comparable to \eqref{positive lower bound 1 of alp}.

Substituting \eqref{new expression of lambda} into \eqref{new expression of L} yields
\begin{equation*}
  L=\sqrt{\frac{\eps_0 h_0^3}{2}}J^{\frac{5\gam-8}{4(\gam+1)}}.
\end{equation*}
This expression is valid for any $\gam>1$ as long as the inequality $0<J\le \bJ$ holds. Suppose that $1<\gam<\frac{8}{5}$. Note that the constant $\eps_0h_0^3$ is fixed independently of $J\in(0, \bJ]$. Therefore, the value of $L$ becomes large if $J$ is fixed sufficiently small.

\medskip

{\textbf{Example 2.}} Returning to Step 4 in the proof of Lemma \ref{proposition-H1-apriori-estimate},
where the parameter $\eta$ in \eqref{definition of G} is fixed as $\eta=\frac{\gam}{4}$, repeat the argument given in Example 1 to get
\begin{equation*}
    \hat{\alp}=1-\mu_1(\kappa,h_0,\gam)J^{\frac{-\gam}{2(\gam+1)}}
  -\mu_2(\kappa,h_0,\gam)\eps_0\left(J^{\frac{-\gam}{2(\gam+1)}}+\mu_3(\kappa, h_0, \gam)\right)^2.
\end{equation*}
Then one can easily see that there exist a constant $\ubJ>1$ sufficiently large, and a small constant $\eps_0>0$ depending only on $(\gam, S_0, \zeta_0, q_0)$ so that, for any $J\in [\ubJ, \infty)$, if $\lambda(\kappa_0, \kappa_L)$ is expressed as
\begin{equation*}
  \lambda(\kappa_0, \kappa_L)=\eps_0J^{\frac{-\gam}{2(\gam+1)}},
\end{equation*}
then the estimate \eqref{lower bd of hat alp} holds true, thus the result comparable to \eqref{positive lower bound 2 of alp} is obtained. In this case, the nozzle length $L$ is expressed as
\begin{equation*}
  L=\sqrt{\frac{\eps_0 h_0^3}{2}}J^{\frac{3\gam-8}{4(\gam+1)}}.
\end{equation*}
Therefore, we conclude that the value of $L$ becomes large if $\gam>\frac{8}{3}$ and $J$ is fixed sufficiently large.

\vspace{.25in}
\noindent
{\bf Acknowledgments:}
The research of Myoungjean Bae was supported in part by  Samsung Science and Technology Foundation under Project Number SSTF-BA1502-51 and the Ministry of Education of the Republic of Korea and the National Research Foundation of Korea (NRF-RS-2025-00553734). The research of Ben Duan was supported in part by National Key R\&D Program of China No. 2024YFA1013303, NSFC Grant Nos. 12271205 and 12171498. And, the research of Chunjing Xie was partially supported by NSFC grants 11971307, 1221101620, and 12161141004, Fundamental Research Grants for Central Universities, Natural Science Foundation of Shanghai 21ZR1433300, Program of Shanghai Academic Research Leader 22XD1421400.

\bigskip
\bibliographystyle{acm}
\bibliography{References_EP_smooth_tr_new}

\end{document}